\documentclass{amsart}

\pdfoutput=1

\makeatletter
\@namedef{subjclassname@1991}{2020 Mathematics Subject Classification}
\makeatother

\usepackage{etex,amsfonts, amsmath, amsthm, amssymb, stmaryrd, epsfig, graphics,
  psfrag, latexsym, mathtools, mathrsfs,enumitem, subfig,
  longtable, booktabs, yfonts, centernot,ifthen,eso-pic} 
\usepackage{xcolor}
\usepackage[all,cmtip]{xy}
\usepackage[percent]{overpic}
\usepackage{tikz} \usetikzlibrary{matrix,arrows,shapes,calc}
\definecolor{darkblue}{rgb}{0,0,0.4} 
 \usepackage[colorlinks=true, citecolor=darkblue, filecolor=darkblue, linkcolor=darkblue,urlcolor=darkblue]{hyperref} 
\usepackage[all]{hypcap}
\usepackage{xr-hyper}

\newcounter{OutlineCompletion}
\AddToShipoutPicture{%
 \AtPageLowerLeft{
\begin{tikzpicture}[remember picture, overlay,xshift=4.25in,scale=0.3,xscale=3,yscale=2,yshift=0.3in,xshift=-3cm,every node/.style={outer sep=0pt,inner sep=1pt,scale=0.3}]

\draw[white] (1,0) rectangle (5,3);

\ifthenelse{\theOutlineCompletion>27}
{\node (cob) at (2,0) {$\Cob$};}{}

\ifthenelse{\theOutlineCompletion>27}
{\node (cobe) at (1,0) {$\CobE$};}{}

\ifthenelse{\theOutlineCompletion>35}
{\node (cobd) at (1,1.5) {$\CobD$};}{}

\ifthenelse{\theOutlineCompletion>39}
{\node (cobdmore) at (1,0.75) {$\wh{\CobD}$};}{}

\ifthenelse{\theOutlineCompletion>0}
{\node (ab) at (2,0.5) {$\AbelianGroups$};}{}

\ifthenelse{\theOutlineCompletion>0}
{\node (gab) at (4,0) {${\mGrAbelianGroups}$};}{}

\ifthenelse{\theOutlineCompletion>0}
{\node (kom) at (4.5,0.5) {${\mComplexes}$};}{}

\ifthenelse{\theOutlineCompletion>25}
{\node (burn) at (1.5,1) {$\BurnsideCat$};}{}

\ifthenelse{\theOutlineCompletion>69}
{\node (mcobd) at (2,1.5) {$\mCobD$};}{}

\ifthenelse{\theOutlineCompletion>49}
{\node (mburn) at (2.75,1) {$\mBurnside$};}{}

\ifthenelse{\theOutlineCompletion>0}
{\node (mab) at (3.5,0.5) {$\mAbelianGroups$};}{}

\ifthenelse{\theOutlineCompletion>19}
{\node (sym) at (4,1.5) {$\mSpectra$};}{}

\ifthenelse{\theOutlineCompletion>69}
{\node (mcobds) at (1.5,1.5) {$\underline{\CobD}$};}{}

\ifthenelse{\theOutlineCompletion>75}
{\node (mtR) at (2,3) {$\CCat{N}\ttimes{\mTshape{m}{n}}$};}{}

\ifthenelse{\theOutlineCompletion>54}
{\node (mhR) at (1.2,2.5) {$\mHshape{m}$};}{}

\ifthenelse{\theOutlineCompletion>54}
{\node (miR) at (2.8,2.5) {$\mHshape{n}$};}{}

\ifthenelse{\theOutlineCompletion>70}
{\node (mtsR) at (3.5,3) {$\strictify{\CCat{N}\ttimes\mTshape{m}{n}}$};}{}

\ifthenelse{\theOutlineCompletion>54}
{\node (mtcR) at (4.5,3) {$\mTshape{m}{n}^0$};}{}

\ifthenelse{\theOutlineCompletion>54}
{\node (mhsR) at (3,1.5) {$\mHshape{m}^0$};}{}

\ifthenelse{\theOutlineCompletion>54}
{\node (misR) at (5,1.5) {$\mHshape{n}^0$};}{}

\ifthenelse{\theOutlineCompletion>27}
{\draw[->] (cob) -- (ab);}{}

\ifthenelse{\theOutlineCompletion>29}
{\draw[->,dashed] (cobe) -- (burn);}{}

\ifthenelse{\theOutlineCompletion>39}
{\draw[->,green] (cobdmore) -- (cobd);}{}

\ifthenelse{\theOutlineCompletion>39}
{\draw[->,dashed] (cobd) -- (burn);}{}

\ifthenelse{\theOutlineCompletion>25}
{\draw[->] (burn) -- (ab);}{}

\ifthenelse{\theOutlineCompletion>0}
{\draw[->,blue!50!white] (mab) -- (gab);}{}

\ifthenelse{\theOutlineCompletion>0}
{\draw[->,blue!50!white] (mab) -- (kom);}{}

\ifthenelse{\theOutlineCompletion>0}
{\draw[->] (kom) -- (gab);}{}

\ifthenelse{\theOutlineCompletion>49}
{\draw[->] (mburn) -- (mab);}{}

\ifthenelse{\theOutlineCompletion>39}
{\draw[->] (cobdmore) -- (cobe);}{}

\ifthenelse{\theOutlineCompletion>27}
{\draw[->] (cobe) -- (cob);}{}

\ifthenelse{\theOutlineCompletion>84}
{\draw[->] (mcobd) -- (mburn);}{}

\ifthenelse{\theOutlineCompletion>89}
{\draw[->,brown] (mtR) -- (mcobd);}{}

\ifthenelse{\theOutlineCompletion>86}
{\draw[->] (mhR) -- (mcobd);}{}

\ifthenelse{\theOutlineCompletion>86}
{\draw[->] (miR) -- (mcobd);}{}

\ifthenelse{\theOutlineCompletion>109}
{\draw[->] (mburn) -- (sym);}{}

\ifthenelse{\theOutlineCompletion>19}
{\draw[->] (sym) -- (gab);}{}

\ifthenelse{\theOutlineCompletion>19}
{\draw[->] (sym) -- (kom);}{}

\ifthenelse{\theOutlineCompletion>75}
{\draw[->,blue!50!white] (mhR) -- (mtR);}{}

\ifthenelse{\theOutlineCompletion>75}
{\draw[->,blue!50!white] (miR) -- (mtR);}{}

\ifthenelse{\theOutlineCompletion>70}
{\draw[->,blue!50!white] (mhsR) -- (mtsR);}{}

\ifthenelse{\theOutlineCompletion>70}
{\draw[->,blue!50!white] (misR) -- (mtsR);}{}

\ifthenelse{\theOutlineCompletion>54}
{\draw[->,blue!50!white] (mhsR) -- (mtcR);}{}

\ifthenelse{\theOutlineCompletion>54}
{\draw[->,blue!50!white] (misR) -- (mtcR);}{}

\ifthenelse{\theOutlineCompletion>0}
{\draw[->,red!70!white] (mab) -- (ab);}{}

\ifthenelse{\theOutlineCompletion>49}
{\draw[->,red!70!white] (mburn) -- (burn);}{}

\ifthenelse{\theOutlineCompletion>69}
{\draw[->,red!70!white] (mcobds) -- (cobd);}{}

\ifthenelse{\theOutlineCompletion>69}
{\draw[->,green] (mcobd) -- (mcobds);}{}

\ifthenelse{\theOutlineCompletion>109}
{\draw[->,brown] (mtsR) -- (sym);}{}

\ifthenelse{\theOutlineCompletion>110}
{\draw[->,brown] (mtcR) -- (sym);}{}

\ifthenelse{\theOutlineCompletion>109}
{\draw[->] (mhsR) -- (sym);}{}

\ifthenelse{\theOutlineCompletion>109}
{\draw[->] (misR) -- (sym);}{}

\ifthenelse{\theOutlineCompletion>75}
{\draw[->,green] (mtR) -- (mtsR);}{}

\ifthenelse{\theOutlineCompletion>54}
{\draw[->,green] (mhR) -- (mhsR);}{}

\ifthenelse{\theOutlineCompletion>54}
{\draw[->,green] (miR) -- (misR);}{}

\ifthenelse{\theOutlineCompletion>99}
{\draw[->,brown] (mtsR) -- (mab);}{}

\ifthenelse{\theOutlineCompletion>59}
{\draw[->,brown] (mtcR) -- (kom);}{}

\ifthenelse{\theOutlineCompletion>59}
{\draw[->] (mhsR) -- (mab);}{}

\ifthenelse{\theOutlineCompletion>59}
{\draw[->] (misR) -- (mab);}{}

\ifthenelse{\theOutlineCompletion>110}
{\draw[->,dotted,brown,thick] ($(mtsR)!0.3!(sym)$) -- ($(mtcR)!0.3!(sym)$);}{}

\ifthenelse{\theOutlineCompletion>99}
{\draw[->,dotted,brown,thick] ($(mtsR)!0.8!(mab)$) -- ($(mtcR)!0.8!(kom)$);}{}
\end{tikzpicture}
}}

\graphicspath{{draws/}{}}

\numberwithin{equation}{section}

\newtheorem{thm}{Theorem}
\newtheorem{theorem}[thm]{Theorem}

\newtheorem{lem}{Lemma}[section]               
\newtheorem{lemma}[lem]{Lemma}

\newtheorem{corollary}[lem]{Corollary}               

\newtheorem{proposition}[lem]{Proposition}
\newtheorem{citethm}[lem]{Theorem}
\theoremstyle{definition}

\newtheorem{definition}[lem]{Definition} 
  
\newtheorem{conjecture}[lem]{Conjecture}  
  
\newtheorem{construction}[lem]{Construction}
\theoremstyle{remark}     

\newtheorem{principle}[lem]{Principle}
\newtheorem{remark}[lem]{Remark}

\newtheorem{example}[lem]{Example}

\newtheorem{convention}[lem]{Convention}

\numberwithin{figure}{section}
\numberwithin{table}{section}

\newcommand{\R}{\mathbb{R}}

\newcommand{\N}{\mathbb{N}}

\newcommand{\lsub}[2]{{}_{#1}#2}

\newcommand{\mc}{\mathcal}

\newcommand{\wh}{\widehat}
\newcommand{\wt}{\widetilde}

\renewcommand{\emptyset}{\varnothing}

\newcommand{\Wmirror}{\overline}

\newcommand{\smas}{\wedge}

\newcommand{\from}{\colon}
\newcommand{\into}{\hookrightarrow}
\newcommand{\too}{\longrightarrow}

\newcommand{\onto}{\twoheadrightarrow}

\renewcommand{\th}{^{\text{th}}}
\newcommand{\st}{^{\text{st}}}

\renewcommand{\hat}{\widehat}

\DeclareMathOperator{\sh}{sh}

\DeclareMathOperator{\Diff}{Diff}

\DeclareMathOperator{\Ob}{Ob}
\DeclareMathOperator{\Id}{Id}

\DeclareMathOperator{\colim}{colim}

\DeclareMathOperator{\Hom}{Hom}
\DeclareMathOperator{\HomComplex}{\underline{Hom}}
\DeclareMathOperator{\Fun}{Fun}

\DeclareMathOperator{\Tor}{Tor}

\DeclareMathOperator{\twoHom}{2Hom}

\newcommand{\dg}{\textit{dg} }

\newcommand{\KhCx}{\mc{C}_{\mathit{Kh}}}

\newcommand{\Cat}{\mathscr{C}}

\newcommand{\Dat}{\mathscr{D}}
\newcommand{\Eat}{\mathscr{E}}

\newcommand{\Forget}{\mathcal{F}_{\!\!\scriptscriptstyle\mathit{orget}}}

\newcommand{\op}{\mathrm{op}}

\newcommand{\gr}{\mathrm{gr}}
\newcommand{\intgr}{\gr_{q}}

\newcommand{\KhSpace}{\mathcal{X}_\mathit{Kh}}

\newcommand{\CubeCat}[1]{\underline{2}^{#1}}

\newcommand{\co}{\colon}
\newcommand{\bdy}{\partial}
\newcommand{\RR}{\R}
\newcommand{\DD}{\mathbb{D}}

\newcommand{\pt}{\mathrm{pt}}

\newcommand{\NN}{\N}
\newcommand{\ZZ}{\mathbb{Z}}

\newcommand{\mathcenter}[1]{\vcenter{\hbox{$#1$}}}

\DeclareMathOperator{\Cone}{Cone}

\newcommand{\BurnsideCat}{\mathscr{B}}
\newcommand{\CCat}[1]{\CubeCat{#1}}
\newcommand{\AbelianGroups}{\mathsf{Ab}}
\newcommand{\mAbelianGroups}{\underline{\AbelianGroups}}
\newcommand{\GrAbelianGroups}{\mathsf{Ab}_*}
\newcommand{\mGrAbelianGroups}{\underline{\GrAbelianGroups}}

\newcommand{\SimplicialSets}{\mathsf{SSet}}
\newcommand{\Spectra}{\mathscr{S}}
\newcommand{\mSpectra}{\underline{\Spectra}}
\newcommand{\Complexes}{\mathsf{Kom}}
\newcommand{\mComplexes}{\underline{\Complexes}}
\newcommand{\Total}[1]{\mathsf{Tot}(#1)}

\newcommand{\Sets}{\mathsf{Sets}}
\newcommand{\PermuCat}{\mathsf{Permu}}

\newcommand{\SphereS}{\mathbb{S}} 

\newcommand{\Cob}{\mathsf{Cob}}
\newcommand{\CobE}{\mathsf{Cob}_e}
\newcommand{\CobD}{\mathsf{Cob}_d}

\newcommand{\HKKa}{\mathit{HKK}}

\DeclareMathOperator{\hocolim}{hocolim}

\newcommand{\KTfunc}{\mc{C}_{\mathit{Kh}}}
\newcommand{\KTalg}[1]{\mathcal{H}^{#1}}
\newcommand{\FlatTangles}[2]{\lsub{#1}\hat{\mathsf{B}}_#2}
\newcommand{\Tangles}[2]{\lsub{#1}\mathsf{C}_#2} 
\newcommand{\TangleDiags}[2]{\lsub{#1}\mathsf{D}_#2} 
\newcommand{\Crossingless}[1]{{\mathsf{B}}_{#1}}

\newcommand{\mHshape}[1]{\mathcal{S}_{#1}}
\newcommand{\mTshape}[2]{{}_{#1}\mskip-0.75\thinmuskip\mathcal{T}_{#2}}
\newcommand{\mTshapeStct}[2]{{}_{#1}^{\vphantom{0}}\mskip-0.75\thinmuskip\mathcal{T}_{#2}}
\newcommand{\mGlue}[3]{\mathcal{U}_{#1,#2,#3}}

\newcommand{\mCobD}{\wt{\underline{\mathsf{Cob}_d}}}
\newcommand{\mBurnside}{\underline{\BurnsideCat}}
\newcommand{\ttimes}{\tilde{\times}}

\newcommand{\GrCat}{\mathscr{G}}

\newcommand{\MultiCat}{\Cat}
\newcommand{\MultiDat}{\Dat}

\newcommand{\mV}{\underline{V}_{\HKKa}}

\newcommand{\mHfunc}[1]{\underline{\mathsf{MC}}_{#1}}
\newcommand{\mTfunc}[1]{\underline{\mathsf{MC}}^{\flat}_{#1}}
\newcommand{\mTfuncNF}[1]{{\underline{\mathsf{MC}}}_{#1}}
\newcommand{\mHinv}[1]{\underline{\mathsf{MB}}_{#1}}
\newcommand{\mTinvNF}[1]{{\underline{\mathsf{MB}}}_{#1}}
\newcommand{\mGlFunc}{\underline{G}}

\newcommand{\KTSpecCat}[1]{\mathscr{H}^{#1}}
\newcommand{\KTSpecRing}[1]{\mathscr{H}^{#1}_{\scriptscriptstyle\mathit{ring}}}
\newcommand{\KTSpecBim}[1]{\mathscr{X}(#1)}
\newcommand{\KTSpecBimRing}[1]{\mathscr{X}_{\scriptscriptstyle\mathit{module}}(#1)}

\newcommand{\DTP}{\otimes^{\mathbb{L}}}

\newcommand{\strictify}[1]{({#1})^0}
\newcommand{\npstrictify}[1]{#1^0}
\newcommand{\thicken}[1]{\wt{#1}}

\newcommand{\LL}{\mathbb{L}} 
\newcommand{\std}{\mathrm{std}}

\newcommand{\SModule}{\mathscr{M}}
\newcommand{\SNodule}{\mathscr{N}}
\newcommand{\cM}{\mathcal{M}}
\newcommand{\cN}{\mathcal{N}}
\newcommand{\symGrp}{\mathfrak{S}}

\newcommand{\aTree}{\ensuremath{\tikz[baseline=(char.base)]{
            \draw (0,0) -- (1ex,1.2ex) -- (1ex,2ex);
            \draw (1.3ex,0) -- (1ex,1.2ex) -- (2ex,0);
            \draw (0.7ex,0) -- (0.5ex,0.6ex);
          \node[inner sep=0pt,outer sep=0pt] (char) at (0,0) {};}}}

\newcommand{\sint}{\begingroup\textstyle \coprod\!\endgroup}
\newcommand{\GmHinv}[1]{\underline{\mathsf{MB}}^\bullet_{#1}}
\newcommand{\GmTinvNF}[1]{\underline{\mathsf{MB}}^\bullet_{#1}}
\newcommand{\GG}{G^\bullet}
\newcommand{\GKTSpecCat}[1]{\KTSpecCat{#1}}
\newcommand{\GKTSpecBim}[1]{\KTSpecBim{#1}}

\newcommand{\Roz}[1]{\mathrm{H}^{\mathrm{st}}(\mathbb{S}^2\times\mathbb{S}^1,#1)}
\DeclareMathOperator{\THH}{THH}

\begin{document}


\title{Khovanov spectra for tangles}

\author{Tyler Lawson}
\thanks{\texttt{TL was supported by NSF Grant DMS-1206008 and NSF FRG Grant DMS-1560699.}}
\email{\href{mailto:tlawson@math.umn.edu}{tlawson@math.umn.edu}}
\address{Department of Mathematics, University of Minnesota, Minneapolis, MN 55455}

\author{Robert Lipshitz}
\thanks{\texttt{RL was supported by NSF Grant DMS-1149800 and NSF FRG Grant DMS-1560783.}}
\email{\href{mailto:lipshitz@uoregon.edu}{lipshitz@uoregon.edu}}
\address{Department of Mathematics, University of Oregon, Eugene, OR 97403}

\author{Sucharit Sarkar}
\thanks{\texttt{SS was supported by NSF Grant DMS-1643401 and NSF FRG Grant DMS-1563615.}}
\email{\href{mailto:sucharit@math.ucla.edu}{sucharit@math.ucla.edu}}
\address{Department of Mathematics, University of California, Los Angeles, CA 90095}

\subjclass{Primary: \href{https://mathscinet.ams.org/mathscinet/msc/msc2020.html?t=&s=57K18}{57K18}, \href{https://mathscinet.ams.org/mathscinet/msc/msc2020.html?t=&s=55P43}{55P43}. Secondary: \href{https://mathscinet.ams.org/mathscinet/msc/msc2020.html?t=&s=57K16}{57K16}, \href{https://mathscinet.ams.org/mathscinet/msc/msc2020.html?t=&s=55P42}{55P42}}


\date{August 13, 2021}

\begin{abstract}
  We define stable homotopy refinements of Khovanov's arc algebras and tangle invariants.
\end{abstract}
\maketitle
\vspace{-1cm}


\tableofcontents


\section{Introduction}\label{sec:intro}

\subsection{Context}
Quantum topology began in the 1980s with the Jones
polynomial~\cite{Jones-knot-poly}, and Witten's reinterpretation of it
via Yang-Mills theory~\cite{Witten-knot-Jones}. Witten's work was at a
physical level of rigor, but Atiyah~\cite{Atiyah-knot-book},
Reshetikhin-Turaev~\cite{RT-knot-quantum}, and others introduced
mathematically rigorous definitions of topological field theories and
related them to both the Jones polynomial and deep questions in
representation theory.

Around the same time, topological field theories also began to appear
in dimension $4$, in the work of Donaldson~\cite{Donaldson-gauge-poly},
Floer~\cite{Floer-gauge-instanton}, and others. Unlike the Jones polynomial, these
4-dimensional invariants all required partial differential equations
to define. (Curiously, while Donaldson's and Floer's invariants were
archetypal examples for what Witten called topological field
theories~\cite{Witten-gauge-TQFT}, they do not satisfy the axioms
mathematicians came to insist on for topological field theories.) The
connection between these invariants and representation theory was also
less apparent.

In the 1990s, Crane-Frenkel proposed that the Jones polynomial and its
siblings might be extended to 4-dimensional topological field theories
via ``a categorical version of a Hopf
algebra''~\cite{CF-kh-categorify}. Inspired
by this suggestion, Khovanov categorified the Jones
polynomial~\cite{Kho-kh-categorification}. Rasmussen showed that
this categorification could be used to study smooth knot concordance
and even to deduce the existence of exotic smooth structures on $\RR^4$
without recourse to gauge theory~\cite{Ras-kh-slice}.

Answering a question of Khovanov's, Jacobsson proved that Khovanov
homology extends to a (3+1)-dimensional topological field
theory~\cite{Jac-kh-cobordisms}. His proof, which involved explicitly
checking the myriad \emph{movie moves} relating different movie
presentations of a surface, was long and intricate.
Khovanov~\cite{Kho-kh-tangles,Kho-kh-cobordism} and, independently,
Bar-Natan~\cite{Bar-kh-tangle-cob} gave simpler proofs of
functoriality of Khovanov homology, by extending it downwards, to
tangles (as Reshetikhin-Turaev had done for the Jones
polynomial). (Their tangle invariants are different, and since then
several more Khovanov homology invariants of tangles have also been
given~\cite{APS-kh-tangle,CK-kh-tangle,BS-kh-tangle,Roberts-kh-tangle}.)
These tangle invariants also led to categorifications of quantum
groups~\cite{Lauda-kh-sl2,KL-kh-cat-gp,Rouquier-Kh-cat} and tensor
products of representations~\cite{Webster-kh-tensor}, and many other
interesting advances.

Returning to gauge theory and related invariants,
in the 1990s, Cohen-Jones-Segal proposed a program to give stable
homotopy refinements of Floer homology groups, in certain
cases~\cite{CJS-gauge-floerhomotopy}. This program has yet to
be carried out rigorously, but using other techniques, stable homotopy
refinements have been given for certain Floer
homologies~\cite{Man-gauge-swspectrum,KM-gauge-swspectrum,Cohen10:Floer-htpy-cotangent,Kragh:transfer-spectra,KLS-gauge-unfolded}. The
Cohen-Jones-Segal program is in two steps: first they use the Floer
data to build a \emph{framed flow category}, and then they use the
framed flow category to build a space; it is the first step for which
technical difficulties have not yet been resolved.

In a previous paper, we built a framed flow category combinatorially
and then used the second step of the Cohen-Jones-Segal program to
define a Khovanov stable homotopy
type~\cite{RS-khovanov}. Hu-Kriz-Kriz gave another construction of a
Khovanov stable homotopy type, using the Elmendorf-Mandell infinite
loop space machine~\cite{HKK-Kh-htpy}. In another previous paper we
were able to show that these two constructions give equivalent
invariants~\cite{LLS-khovanov-product}. Hu-Kriz-Kriz's construction
factors through the embedded cobordism category of
$(\RR^2,[0,1]\times\RR^2)$, a point that will be important in our
construction of tangle invariants below.

Computations show that this Khovanov stable homotopy type is strictly stronger than Khovanov
homology~\cite{RS-steenrod,Seed-Kh-square} and can be used to give
additional concordance
information~\cite{RS-s-invariant,LLS-khovanov-product}. (A homotopy-theoretic
lift of Khovanov homology which does not have more information than
Khovanov homology was given by
Everitt-Turner~\cite{ET-kh-spectrum,ELST-trivial}.)

We would like to use the Khovanov homotopy type to study smoothly
embedded surfaces in $\RR^4$. Following Khovanov and Bar-Natan, as a
step towards this goal, in this paper we construct an extension of the
Khovanov stable homotopy type to tangles.

\begin{remark}
  Hu-Kriz-Somberg have outlined a construction of a stable homotopy
  type refining $\mathfrak{sl}_n$ Khovanov-Rozansky
  homology~\cite{HKS-Kh-sln}. Their construction passes through
  \emph{oriented tangles}, i.e., tangles in $[0,1]\times \mathbb{D}^2$
  every strand of which runs from $\{0\}\times\mathbb{D}^2$ to
  $\{1\}\times\mathbb{D}^2$. At the time of writing, their
  construction is restricted to a homotopy type localized at a
  ``large'' prime $p$ (depending on $n$).
\end{remark}

\subsection{Statement of results}
In this paper, we give two extensions of the
Khovanov homotopy type to tangles. The first is combinatorial, and has
the form of a multifunctor $\mTinvNF{T}$ from a particular
multicategory to the Burnside category. The functor $\mTinvNF{T}$ is
well-defined up to a notion of stable equivalence
(Theorem~\ref{thm:comb-inv}). (For the special case of knots, this
essentially reduces to the combinatorial invariant described in a
previous paper~\cite{LLS-cube}.) To summarize:

\begin{theorem}
  Given a $(2m,2n)$-tangle $T$ with $N$ crossings, there is an
  associated multifunctor
  \[
    \mTinvNF{T}\co \CCat{N}\ttimes\mTshape{m}{n}\to \mBurnside.
  \]
  Up to stable equivalence, $\mTinvNF{T}$ is an invariant of the isotopy
  class of $T$. The composition of $\mTinvNF{T}$ with the forgetful map
  $\mBurnside\to\mAbelianGroups$ is identified with Khovanov's tangle
  invariant~\cite{Kho-kh-tangles}.
\end{theorem}
(This is restated and proved as Lemma~\ref{lem:is-Kh} and Theorem~\ref{thm:comb-inv}, below.)

Next, we use the Elmendorf-Mandell machine to define a spectral
category (category enriched over spectra) $\KTSpecCat{m}$ so that the
homology of $\KTSpecCat{m}$ is the Khovanov arc algebra
$H^m$. (After this introduction we denote the algebra $H^m$ by
$\KTalg{m}$, to avoid conflicting with the notation for singular cohomology.) We
then turn $\mTinvNF{T}$ into a (spectral) bimodule $\KTSpecBim{T}$ over
$\KTSpecCat{m}$ and $\KTSpecCat{n}$, so that the singular chain
complex of $\KTSpecBim{T}$ is quasi-isomorphic, as a complex of
$(H^m,H^n)$-bimodules, to the Khovanov tangle invariant of $T$. We
then prove:
\begin{theorem}
  Up to equivalence of
  $(\KTSpecCat{m},\KTSpecCat{n})$-bimodules, $\KTSpecBim{T}$ is an
  invariant of the isotopy class of $T$. Further, given a $(2n,2p)$-tangle $T'$,
  \[
    \KTSpecBim{T'\circ T}\simeq \KTSpecBim{T}\otimes_{\KTSpecCat{n}}^\LL\KTSpecBim{T'}
  \]
  (where tensor product denotes the tensor product of module spectra).
\end{theorem}
(This is restated and proved as Theorems~\ref{thm:spec-invariance} and~\ref{thm:gluing}, below.)

The outline of the construction is as follows:
\begin{enumerate}
\item We construct a multicategory $\mCobD$, enriched in groupoids, of
  \emph{divided cobordisms}, so that:
  \begin{enumerate}
  \item there is at most one $2$-morphism between any pair of
    morphisms in $\mCobD$;
  \item the Khovanov-Burnside functor $V_{\HKKa}$ from the embedded
    cobordism category to the Burnside category induces a functor $\mV$ from
    $\mCobD$ to the Burnside category; and
  \item the cobordisms involved in the Khovanov arc algebras and
    tangle invariants have (essentially canonical) representatives in
    $\mCobD$.
  \end{enumerate}
  (Sections~\ref{sec:CobD} and~\ref{sec:mCobD}.)
\item We define an \emph{arc-algebra shape multicategory}
  $\mHshape{n}^0$ and \emph{tangle shape multicategory}
  $\mTshapeStct{m}{n}^0$ so that the Khovanov arc algebras and
  tangle invariants are equivalent to multifunctors
  $\mHshape{n}^0\to \mAbelianGroups$ and
  $\mTshapeStct{m}{n}^0\to\mComplexes$. There are also groupoid-enriched
  versions of $\mHshape{n}$ and $\mTshape{m}{n}$, and projection maps
  $\mHshape{n}\to\mHshape{n}^0$, $\mTshape{m}{n}\to\mTshapeStct{m}{n}^0$.
  (Sections~\ref{sec:multi-vec}
  and~\ref{sec:shape}.)
\item The functor $\mHshape{n}^0\to \mAbelianGroups$ factors through a
  functor $\mHshape{n}\to\mCobD$. Similarly, the tangle invariant
  $\mTshapeStct{m}{n}^0\to\mComplexes$ factors through a functor
  $\CCat{N}\ttimes\mTshape{m}{n}\to \mCobD$ from (an appropriate kind
  of) product of $\mTshape{m}{n}$ and a cube. So, we can compose with
  $\mV$ to get functors $\mHinv{n}\co \mHshape{n}\to\mBurnside$
  and $\mTinvNF{T}\co \CCat{N}\ttimes\mTshape{m}{n}\to\mBurnside$.  We
  also digress to note that we can view $\mTinvNF{T}$ as a tangle
  invariant in an appropriate derived
  category. (Section~\ref{sec:tang-to-burn}.)
\item Using the Elmendorf-Mandell $K$-theory machine and rectification
  results, we can turn $\mHinv{n}$ and $\mTinvNF{T}$ into functors
  $\mHshape{n}^0\to\mSpectra$ and $\mTshapeStct{m}{n}^0\to\mSpectra$. We
  reinterpret these functors as a spectral category and spectral
  bimodule, respectively. Whitehead's theorem combined with familiar
  invariance arguments implies that the functor
  $\mTshapeStct{m}{n}^0\to\mSpectra$ is a tangle
  invariant. (Section~\ref{sec:comb-to-top}.)
\item The gluing theorem for tangles follows by considering a map from
  a larger multicategory to $\mCobD$; the corresponding result for the
  Khovanov bimodules; projectivity (sweetness) of the Khovanov
  bimodules; and, again, a version of Whitehead's
  theorem. (Section~\ref{sec:gluing}.)
\end{enumerate}
We precede these constructions with a review of Khovanov's tangle
invariants and some algebraic topology background
(Section~\ref{sec:alg-top}), and follow it with some modest structural
applications (Section~\ref{sec:appl}). We concentrate the discussion of quantum gradings in Section~\ref{sec:q-gr}.

The outline of the construction is summarized by
Figure~\ref{fig:big-outline}. The partial diagrams at the bottom of
the pages, starting on page~\pageref{first-page-num-diag}, track the
progress of our construction.
\begin{figure}
\centering
\begin{tikzpicture}[xscale=3,yscale=2, every node/.style={outer sep=0pt,inner sep=1pt}]


\node (cob) at (2,0) {$\Cob$};
\node (cobe) at (1,0) {$\CobE$};
\node (cobd) at (1,1.5) {$\CobD$};
\node (cobdmore) at (1,0.75) {$\wh{\CobD}$};

\node (ab) at (2,0.5) {$\AbelianGroups$};
\node (gab) at (4,0) {${\mGrAbelianGroups}$};
\node (kom) at (4.5,0.5) {${\mComplexes}$};
\node (burn) at (1.5,1) {$\BurnsideCat$};

\node (mcobd) at (2,1.5) {$\mCobD$};
\node (mburn) at (2.75,1) {$\mBurnside$};
\node (mab) at (3.5,0.5) {$\mAbelianGroups$};
\node (sym) at (4,1.5) {$\mSpectra$};

\node (mcobds) at (1.5,1.5) {$\underline{\CobD}$};


\node (mtR) at (2,3) {$\CCat{N}\ttimes{\mTshape{m}{n}}$};
\node (mhR) at (1.2,2.5) {$\mHshape{m}$};
\node (miR) at (2.8,2.5) {$\mHshape{n}$};


\node (mtsR) at (3.5,3) {$\strictify{\CCat{N}\ttimes\mTshape{m}{n}}$};
\node (mtcR) at (4.5,3) {$\mTshapeStct{m}{n}^0$};
\node (mhsR) at (3,1.5) {$\mHshape{m}^0$};
\node (misR) at (5,1.5) {$\mHshape{n}^0$};

\draw[->] (cob) -- (ab);
\draw[->,dashed] (cobe) -- (burn);
\draw[>->,green] (cobdmore) -- (cobd);
\draw[->,dashed] (cobd) -- (burn);
\draw[->] (burn) -- (ab);
\draw[left hook->,blue!50!white] (mab) -- (gab);
\draw[left hook->,blue!50!white] (mab) -- (kom);
\draw[->] (kom) -- (gab);
\draw[->] (mburn) -- (mab);
\draw[->] (cobdmore) -- (cobe);
\draw[->] (cobe) -- (cob);
\draw[->] (mcobd) -- (mburn);


\draw[->,brown,line width=1.5pt] (mtR) -- (mcobd);
\draw[->] (mhR) -- (mcobd);
\draw[->] (miR) -- (mcobd);
\draw[->] (mburn) -- (sym);
\draw[->] (sym) -- (gab);
\draw[->] (sym) -- (kom);

\draw[right hook->,blue!50!white] (mhR) -- (mtR);
\draw[left hook->,blue!50!white] (miR) -- (mtR);

\draw[right hook->,blue!50!white] (mhsR) -- (mtsR);
\draw[left hook->,blue!50!white] (misR) -- (mtsR);
\draw[right hook->,blue!50!white] (mhsR) -- (mtcR);
\draw[left hook->,blue!50!white] (misR) -- (mtcR);

\draw[-latex,red!70!white] (mab) -- (ab);
\draw[-latex,red!70!white] (mburn) -- (burn);
\draw[-latex,red!70!white] (mcobds) -- (cobd);

\draw[>->,green] (mcobd) -- (mcobds);

\draw[->,brown,line width=1.5pt] (mtsR) -- (sym);
\draw[->,brown,line width=1.5pt] (mtcR) -- (sym);
\draw[->] (mhsR) -- (sym);
\draw[->] (misR) -- (sym);


\draw[>->,green] (mtR) -- (mtsR);
\draw[>->,green] (mhR) -- (mhsR);
\draw[>->,green] (miR) -- (misR);

\draw[->,line width=1.5pt,brown] (mtsR) -- (mab);
\draw[->,line width=1.5pt,brown] (mtcR) -- (kom);
\draw[->] (mhsR) -- (mab);
\draw[->] (misR) -- (mab);

\draw[->,dotted,brown,line width=1.5pt] ($(mtsR)!0.3!(sym)$) -- ($(mtcR)!0.3!(sym)$);
\draw[->,dotted,brown,line width=1.5pt] ($(mtsR)!0.8!(mab)$) -- ($(mtcR)!0.8!(kom)$);
\end{tikzpicture}
\caption{\textbf{The outline of the construction.} We construct the
  above diagram starting with a $(2m,2n)$-tangle diagram
  $T$. {\color{blue!50!white}Hook-tailed} arrows are subcategory
  inclusions, {\color{green} split-tailed} arrows are strictifications
  from groupoid enriched multicategories to ordinary multicategories,
  and {\color{red!70!white}solid-headed} arrows convert a multicategory
  to an ordinary category by forgetting multimorphisms. Only the
  {\color{brown}thick} arrows depend on the tangle $T$. Solid arrows
  are strict, while the dashed arrows are lax. The two
  {\color{brown}dotted} arrows are functors between functor
  categories,
  $\mSpectra^{\strictify{\CCat{N}\ttimes\mTshape{m}{n}}}\to\mSpectra^{\mTshapeStct{m}{n}^0}$
  and
  $\mAbelianGroups^{\strictify{\CCat{N}\ttimes\mTshape{m}{n}}}\to{\mComplexes}^{\mTshapeStct{m}{n}^0}$
  (their only dependence on the tangle is in an overall grading
  shift). The diagram commutes, with the understanding that anything
  involving the strictification arrows only commutes up to (zigzags of)
  natural equivalences, and arrows to $\mComplexes$ only commute up to
  quasi-isomorphisms. The picture does not encompass the quantum
  gradings.}
\label{fig:big-outline}
\end{figure}
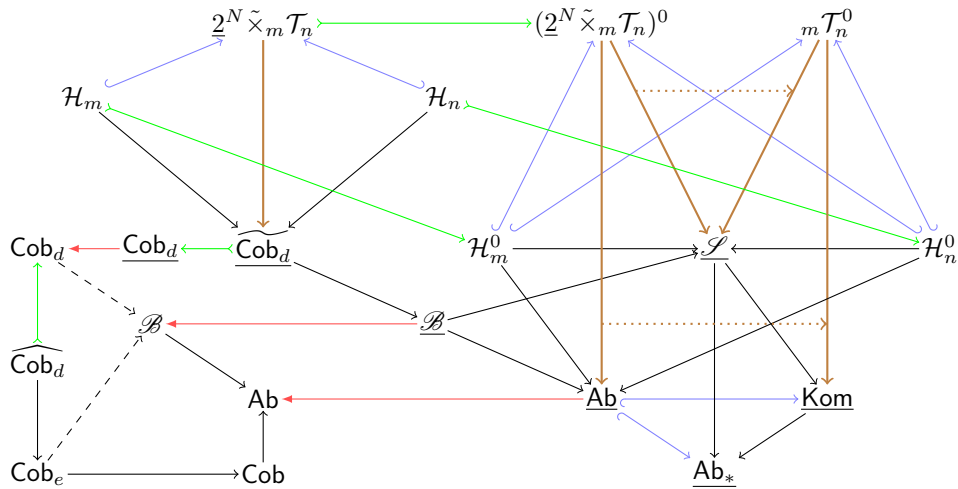

\begin{remark}
  To construct both the combinatorial and topological tangle invariants,
  we use the language of multicategories. There is another construction
  of a combinatorial invariant with at least as much information, using
  the language of enriched bicategories (cf.~\cite{GarnerShulman}); we
  may return to this point in a future paper.
\end{remark}

\emph{Acknowledgments.} We thank Finn Lawler for pointing us
to~\cite{GarnerShulman} and Andrew Blumberg and Aaron Royer for
helpful conversations. Finally, we thank the referee for more helpful
comments and corrections.

\section{Background}\label{sec:alg-top}

\subsection{Homological grading conventions}\label{sec:grading-convention}
In this paper, we will work with chain complexes. We view cochain complexes as chain complexes by negating the grading. In particular, the Khovanov complex was originally defined as a cochain complex~\cite{Kho-kh-categorification}, but we will view it as a chain complex. So, our homological gradings differ from Khovanov's by a sign.

\subsection{Multicategories}
\begin{definition}
  A \emph{multicategory} (or \emph{colored operad}) $\MultiCat$ consists of:
  \begin{enumerate}[label=(M-\arabic*)]
  \item A set or, more generally, class, $\Ob(\MultiCat)$ of \emph{objects};
  \item For each $n\geq 0$ and objects
    $x_1,\dots,x_n,y\in\Ob(\MultiCat)$, a set $\Hom(x_1,\dots,x_n;y)$
    of \emph{multimorphisms} from $(x_1,\dots,x_n)$ to $y$;
  \item a composition map 
    \[
      \Hom(y_1,\dots,y_n;z)\times \Hom(x_{1,1},\dots,x_{1,m_1};y_1)\times\cdots\times\Hom(x_{n,1},\dots,x_{n,m_n};y_n)\to 
      \Hom(x_{1,1},x_{1,2},\dots,x_{n,m_n};z);
    \]
    and 
  \item \label{item:unit} A distinguished element $\Id_x\in\Hom(x;x)$, called the \emph{identity} or \emph{unit},
  \end{enumerate}
  satisfying the following conditions:
  \begin{enumerate}[resume, label=(M-\arabic*)]
  \item Composition is associative, in the sense that the following diagram commutes:
    \[
      \xymatrix{
        {\begin{array}{l}\Hom(y_1,\dots,y_n;z)\\\times\prod_{i=1}^n\Hom(x_{i,1},\dots,x_{i,m_i};y_i)\\\times\prod_{i=1}^n\prod_{j=1}^{m_i}\Hom(w_{i,j,1},\dots,w_{i,j,k_{i,j}};x_{i,j})
          \end{array}}\ar[r]\ar[d]
        &
        {\begin{array}{l}\Hom(x_{1,1},\dots,x_{n,m_n};z)\\\times\prod_{i=1}^n\prod_{j=1}^{m_i}\Hom(w_{i,j,1},\dots,w_{i,j,k_{i,j}};x_{i,j})\end{array}}\ar[d]\\
        {\begin{array}{l}\Hom(y_1,\dots,y_n;z)\\\times\prod_{i=1}^n\Hom(w_{i,1,1},\dots,w_{i,m_i,k_{i,m_i}};y_i)\end{array}}
        \ar[r]
        &
        \Hom(w_{1,1,1},\dots,w_{n,m_n,k_{n,m_n}};z).
      }
    \]
    (Here, all of the maps are composition maps.) 
  \item\label{item:right-id} The identity elements are right identities for composition, in the
    sense that the following diagram commutes:
    \[
      \xymatrix{
        \Hom(x_1,\dots,x_n;y)\ar[r]^=\ar[d]_{\Id\times \prod \Id_{x_i}} & \Hom(x_1,\dots,x_n;y)\\
        \Hom(x_1,\dots,x_n;y)\times\prod_{i=1}^n\Hom(x_i,x_i).\ar[ur]_\circ
      }
    \]
  \item\label{item:left-id} The identity elements are left identities for composition, in the
    sense that the following diagram commutes:
    \[
      \xymatrix{
        \Hom(x_1,\dots,x_n;y)\ar[r]^=\ar[d]_{\Id_{y}\times \Id} & \Hom(x_1,\dots,x_n;y)\\
        \Hom(y,y)\times \Hom(x_1,\dots,x_n;y).\ar[ur]_\circ
      }
    \]
  \end{enumerate}

  Given multicategories $\MultiCat$ and $\MultiDat$, a
  \emph{multifunctor} $F\co\MultiCat\to\MultiDat$ is a map
  $F\co \Ob(\MultiCat)\to\Ob(\MultiDat)$ and, for each
  $x_1,\dots,x_n,y\in\Ob(\MultiCat)$, a map
  $\Hom_{\MultiCat}(x_1,\dots,x_n;y)\to\Hom_{\MultiDat}(F(x_1),\dots,F(x_n);F(y))$
  which respects multi-composition and identity elements.
\end{definition}

Multicategories, which model the notion of multilinear maps, are a
common generalization of a category (a multicategory in which only
multimorphism sets with one input are nonempty) and an operad (a
multicategory with one object). Multicategories were introduced by
Lambek~\cite{Lambek-other-multicat} and
Boardman-Vogt~\cite{BV-other-book}. In Boardman-Vogt's work and most
modern algebraic topology, the multimorphism sets in multicategories
are equipped with actions of the symmetric group; the definition we
have given would be called a non-symmetric multicategory. Some of our
multicategories (notably $\mBurnside$, $\Sets/X$, and $\Spectra$)
are, in fact, symmetric multicategories. In particular, the
multicategories $\Sets/X$ to which we apply Elmendorf-Mandell's
$K$-theory are symmetric multicategories.

A monoidal category $(\Cat,\otimes)$
produces a multicategory, which we will
denote $\underline{\Cat}$, as follows. The objects of
$\underline{\Cat}$ are the same as the objects of $\Cat$, and the
multimorphism sets are given by
\[
\Hom_{\underline{\Cat}}(x_1,\dots,x_n;y)=\Hom_{\Cat}(x_1\otimes\dots\otimes x_n;y)
\]
(for any choice of how to parenthesize the tensor product).  If the
monoidal category happened to be a symmetric monoidal category, as in
the case of abelian groups $\AbelianGroups$, graded abelian groups
$\GrAbelianGroups$, or chain complexes $\Complexes$,
then the corresponding multicategory is a symmetric
multicategory. (These are examples of Hu-Kriz-Kriz's
\emph{$\star$-categories}~\cite{HKK-Kh-htpy}.)
\setcounter{OutlineCompletion}{10}\label{first-page-num-diag}

Many of our multicategories will be enriched in groupoids. That is,
the multimorphism sets will be groupoids (i.e., categories in which
all the morphisms are invertible) and the composition maps are maps of
groupoids (i.e., functors).

Most of our non-enriched multicategories will be rather
simple, in a sense we make precise:

\begin{definition}\label{def:shape-multicat-set}
  Given a finite set $X$, the \emph{shape multicategory of $X$} has
  objects $X\times X$, and the multimorphism set
  $\Hom((a_1,b_1),(a_2,b_2),\dots,(a_n,b_n);(b_0,a_{n+1}))$ consists
  of a single element if $b_i=a_{i+1}$ for all $0\leq i\leq n$, and
  all other multimorphism sets empty. We allow the special case
  $n=0$ which produces a unique zero-input multimorphism in
  $\Hom(\emptyset;(a,a))$ for each $a\in X$.
\end{definition}

Generalizing Definition~\ref{def:shape-multicat-set}, we have the
following variant.
\begin{definition}\label{def:shape-multicat-set-seq}
  Given a finite sequence of finite sets $X^1,\dots,X^k$, the
  \emph{shape multicategory of $(X^1,\dots,X^k)$} has objects
  $\coprod_{i\leq j} X^i\times X^j$ and
  $\Hom((a_1,b_1),(a_2,b_2),\dots,(a_n,b_n);(b_0,a_{n+1}))$ consists
  of a single element if $b_i=a_{i+1}$ for all $0\leq i\leq n$, and
  all other multimorphism sets empty. Once again, we allow the special
  case $n=0$ which produces a unique zero-input multimorphism in
  $\Hom(\emptyset;(a,a))$ for each $a\in \coprod_i X^i$.
\end{definition}

\subsection{Linear categories and multifunctors to abelian groups}
\label{sec:multi-vec}

Many of the algebras that we will encounter in this paper will come
equipped with an extra structure, which we abstract below.
\begin{definition}\label{def:idempotented}
  \emph{An algebra equipped with an orthogonal set of idempotents} is an algebra $A$ and a finite subset $I\subset A$, so that 
  \begin{itemize}
  \item $\iota^2=\iota$ for all $\iota\in I$,
  \item $\iota\iota'=\iota'\iota=0$ for all distinct $\iota,\iota'\in I$, and
  \item $\sum_{\iota\in I}\iota=1$.
  \end{itemize}
\end{definition}
The following three notions are equivalent.
\begin{enumerate}
\item A ring $A$ (algebra over $\ZZ$) equipped with an orthogonal set of idempotents $X$.
\item A linear category (category enriched over abelian groups
  $\AbelianGroups$) with objects a finite set $X$.
\item A multifunctor from the shape multicategory $M$ of a finite set $X$ to the multicategory $\mAbelianGroups$ of abelian groups. 
\end{enumerate}
(A similar statement holds for algebras over any ring $R$; the
corresponding linear category has to be enriched over $R$-modules, and
the corresponding multifunctor should map to the multicategory of
$R$-modules.)

To see the equivalence, given a multifunctor $F\co
M\to\mAbelianGroups$ there is a corresponding linear category with
objects $X$, $\Hom(x,y)=F((x,y))$, composition
$\Hom(y,z)\otimes\Hom(x,y)\to\Hom(x,z)$ is the image of the unique
morphism $(x,y),(y,z)\to(x,z)$, and the identity $\Id_x\in\Hom(x,x)$
is the image of $1$ under the maps $\ZZ\to\Hom(x,x)$, which is the
image under $F$ of the unique morphism $\emptyset\to (x,x)$. Given
a linear category $\Cat$ with finitely many objects, we can form a ring
$A_\Cat=\bigoplus_{x,y\in\Ob(\Cat)}\Hom_\Cat(x,y)$ with multiplication
given by composition (i.e., $a\cdot b \coloneqq b\circ a$) when defined and $0$ otherwise; the ring $A_\Cat$
is equipped with the orthogonal set of idempotents $\{\Id_x\mid
x\in\Ob(\Cat)\}$. From a ring $A$ equipped with an orthogonal set of
idempotents $I$, we obtain a map $F\co
M\to\mAbelianGroups$ by setting $F((x,y))=xAy$ and declaring that $F$
sends the unique morphism $(x,y),(y,z)\to(x,z)$ to the multiplication
map $xAy\otimes yAz\to xAz$ and that $F$ respects composition and
identity maps.

In a similar fashion, given linear categories $\Cat$ and $\Dat$ with finitely many objects, the
following are equivalent notions for bimodules.
\begin{enumerate}
\item A left-$A_\Cat$ right-$A_\Dat$ bimodule $B$.
\item An enriched functor
  $F_A\from\Cat^\op\times\Dat\to\AbelianGroups$; an enriched functor
  between linear categories is one for which the map on morphisms
  $\Hom_{\Cat^{\op}\times\Dat}((c,d),(c',d'))\to\Hom_{\AbelianGroups}(F_A(c,d),F_A(c',d'))$
  is linear, or equivalently,
  $\Hom_{\Cat^{\op}\times\Dat}((c,d),(c',d'))\times F_A(c,d)\to
  F_A(c',d')$ is bilinear.
\item A multifunctor from the shape multicategory $M(\Cat,\Dat)$ of
  $(\Ob(\Cat),\Ob(\Dat))$ to $\mAbelianGroups$, which restricts to the
  multifunctors corresponding to $\Cat$, respectively $\Dat$, (as
  defined above) on the subcategory of $M(\Cat,\Dat)$ which is the
  shape multicategory of $\Ob(\Cat)$, respectively $\Ob(\Dat)$.
\end{enumerate}
Recall from Definition~\ref{def:shape-multicat-set-seq} that $M(\Cat,\Dat)$ consists of the following data.
\begin{itemize}
\item Three kinds of objects:
  \begin{itemize}
  \item Pairs $(x_1,x_2)\in\Ob(\Cat)\times\Ob(\Cat)$.
  \item Pairs $(y_1,y_2)\in\Ob(\Dat)\times\Ob(\Dat)$.
  \item Pairs $(x,y)$ where $x\in\Ob(\Cat)$ and $y\in\Ob(\Dat)$. For
    notational clarity, we will write $(x,y)$ instead as $(x,[B],y)$.
  \end{itemize}
\item A single multimorphism in each of the following cases:
  \begin{itemize}
  \item $(x_1,x_2),(x_2,x_3),\dots,(x_{m-1},x_m)\to(x_1,x_m)$ where $x_1,\dots,x_m\in\Ob(\Cat)$.
  \item $(y_1,y_2),(y_2,y_3),\dots,(y_{n-1},y_n)\to(y_1,y_n)$ where $y_1,\dots,y_n\in\Ob(\Dat)$.
  \item
    $(x_1,x_2),\dots,(x_{m-1},x_m),(x_m,[B],y_1),(y_1,y_2),\dots,(y_{n-1},y_n)\to
    (x_1,[B],y_n)$ where $x_1,\dots,x_m\in\Ob(\Cat)$ and
    $y_1,\dots,y_n\in\Ob(\Dat)$.
  \end{itemize}
\end{itemize}
The bimodule $B$ defines a multifunctor $F_B\co
M(\Cat,\Dat)\to\mAbelianGroups$ as follows:
\begin{itemize}
\item On objects, for $x_1,x_2\in\Ob(\Cat)$ and $y_1,y_2\in\Ob(\Dat)$,
  $F_B(x_1,x_2)=\Hom_\Cat(x_1,x_2)=\Id_{x_1}A_\Cat\Id_{x_2}$,
  $F_B(y_1,y_2)=\Hom_\Dat(y_1,y_2)=\Id_{y_1}A_\Dat\Id_{y_2}$, and
  $F_B(x_1,[B],y_1)=\Id_{x_1}B\Id_{y_1}$.
\item On the first and second types of multimorphisms, $F_B$ is simply
  composition. For the third type, the map $F_B$ sends the
  multimorphism
  \[
  (x_1,x_2),\dots,(x_{m-1},x_m),(x_m,[B],y_1),(y_1,y_2),\dots,(y_{n-1},y_n)\to
  (x_1,[B],y_n)
  \]
  to the product 
  \[
  \Id_{x_1}R_\Cat\Id_{x_2}\otimes\cdots\otimes \Id_{x_{m-1}}R_\Cat\Id_{x_m}
  \otimes \Id_{x_m}B\Id_{y_1}\otimes
  \Id_{y_1}R_\Dat\Id_{y_2}\otimes\cdots\otimes \Id_{y_{n-1}}R_\Dat\Id_{y_n}\to \Id_{x_1}B\Id_{y_n}.
  \]
\end{itemize}
Conversely, every multifunctor $M(\Cat,\Dat)\to\mAbelianGroups$ of the
given form arises as $F_B$ for
the bimodule $B=\bigoplus_{x\in\Ob(\Cat),y\in\Ob(\Dat)}F_B(x,[B],y)$.

Similarly, given a multifunctor $F_B\from
M(\Cat,\Dat)\to\mAbelianGroups$, we can construct an enriched functor
$F_A\from \Cat^\op\times\Dat\to\AbelianGroups$ as follows:
\begin{itemize}
\item On objects, $F_A(c,d)=F_B(c,[B],d)$. 
\item On morphisms, $\Hom_{\Cat^\op\times\Dat}((c,d),(c',d'))\otimes F_A(c,d)\to F_A(c',d')$ is the composition
\[
\Hom_{\Cat^\op\times\Dat}((c,d),(c',d'))\otimes F_A(c,d)
=F_B(c',c)\otimes F_B(c,[B],d)\otimes F_B(d,d')\to F_B(c',[B],d').
\]
\end{itemize}

There are similar equivalences for the notions of differential
$(A_\Cat,A_\Dat)$-bimodules, enriched functors
$\Cat^\op\times\Dat\to\Complexes$, and multifunctors
$M(\Cat,\Dat)\to\mComplexes$.

\subsection{Trees and canonical groupoid enrichments}
To define some enriched multicategories, we will first need some
terminology about trees.

A \emph{planar, rooted tree} is a tree $\aTree$ with some number
$n\geq 1$ of leaves, which has been embedded in $\RR\times[0,1]$ so
that $k\leq n-1$ of the leaves are embedded in $\RR\times\{0\}$, one
leaf is embedded in $\RR\times\{1\}$, and no other edges or vertices
are mapped to $\RR\times\{0,1\}$. The vertices mapped to
$\RR\times\{0\}$ are called \emph{inputs} of $\aTree$ and the vertex
mapped to $\RR\times\{1\}$ is the \emph{output} or \emph{root} of
$\aTree$. We call the remaining vertices of $\aTree$
\emph{internal}. We view planar, rooted trees as directed graphs, in
which edges point away from the inputs and towards the output. In
particular, given a valence $m$ internal vertex $p$ of $\aTree$,
$(m-1)$ of the edges adjacent to $p$ are \emph{input edges} to $p$ and
one edge is the \emph{output edge} of $p$, and the inputs of $p$ are
ordered. We allow the case $m=1$, and call such $0$-input $1$-output
internal vertices \emph{stump leaves}. We view two planar, rooted
trees as equivalent if there is an orientation-preserving
self-homeomorphism of $\RR\times[0,1]$ which preserves
$\RR\times\{0\}$ and $\RR\times\{1\}$ and takes one tree to the other.

Given a tree $\aTree$, the \emph{collapse} of $\aTree$ is the result of
collapsing all internal edges of $\aTree$, to obtain a tree with one
internal vertex (i.e., a \emph{corolla}).

\subsubsection{Canonical groupoid enrichments}\label{sec:enrich}
First, given a non-enriched multicategory $\Cat$ we can enrich $\Cat$
over groupoids \emph{trivially} as follows. Given elements
$f,g\in\Hom_{\Cat}(x_1,\dots,x_n;y)$ define $\Hom(f,g)$ to be empty if
$f\neq g$ and to consist of a single element, the identity map, if
$f=g$.

Next we give a different way of enriching multicategories over groupoids, which
provides a tool for turning lax multifunctors into strict ones
(from a different source), though we will avoid ever actually defining or using
the notion of a lax multifunctor or multicategory.  Suppose $\Cat$ is an unenriched
multicategory. The \emph{canonical thickening}
$\thicken{\Cat}$ is the multicategory enriched in groupoids
defined as follows. The objects of $\thicken{\Cat}$ are the same as
the objects of $\Cat$.  Informally, an object in
$\Hom_{\thicken{\Cat}}(x_1,\dots,x_n;y)$ is a sequence of composable
multimorphisms starting at $x_1,\dots,x_n$ and ending at $y$. The $2$-morphisms record whether two sequences compose to the same multimorphism.

More precisely, an object of $\Hom_{\thicken{\Cat}}(x_1,\dots,x_n;y)$
is a tree $\aTree$ with $n$ inputs, together with a labeling of each edge of
$\aTree$ by an object of $\Cat$ and each internal vertex of $\aTree$ by a
multimorphism of $\Cat$, subject to the following conditions:
\begin{enumerate}
\item The input edges of $\aTree$ are labeled by $x_1,\dots,x_n$ (in that order).
\item The output edge of $\aTree$ is labeled by $y$.
\item At a vertex $v$, if the input edges to $v$ are labeled
  $w_1,\dots,w_k$ and the output edge is labeled $z$ then the vertex
  $v$ is labeled by an element of $\Hom_{\Cat}(w_1,\dots,w_k;z)$. In
  particular, stump leaves of $\aTree$ are labeled by multimorphisms
  in $\Hom_{\Cat}(\emptyset;z)$, i.e., by $0$-input multimorphisms.
\end{enumerate}

Given a morphism $f\in \Hom_{\thicken{\Cat}}(x_1,\dots,x_n;y)$, we can
compose the multimorphisms labeling the vertices according to the
tree to obtain a morphism
$f^\circ\in \Hom_{\Cat}(x_1,\dots,x_n;y)$. Given morphisms
$f,g\in \Hom_{\thicken{\Cat}}(x_1,\dots,x_n;y)$, define
$
  \Hom_{\thicken{\Cat}}(f,g)
$
to have one element if $f^\circ=g^\circ$ and to be empty otherwise.
The unit in $\Hom_{\thicken{\Cat}}(x;x)$ is the tree with one input,
one output, no internal vertices, and edge labeled $x$.  This
completes the definition of the multimorphism groupoids in
$\thicken{\Cat}$.

Composition of multimorphisms is simply gluing of
trees. (When gluing together the external vertices, they disappear
rather than creating new internal vertices.)
\begin{lemma}
  This definition of composition extends uniquely to morphisms in the
  multimorphism groupoids.
\end{lemma}
\begin{proof}
  This is immediate from the definitions.
\end{proof}

\begin{lemma}
  These definitions make $\thicken{\Cat}$ into a multicategory enriched in groupoids.
\end{lemma}
\begin{proof}
  At the level of objects of the multimorphism groupoids,
  associativity follows from associativity of composition of trees. At
  the level of morphisms of the multimorphism groupoids, associativity
  trivially holds. The unit axioms follow
  from the fact that gluing on a tree with no internal vertices has no
  effect.
\end{proof}

We will call a multimorphism in $\thicken{\Cat}$ \emph{basic} if the
underlying tree has only one internal vertex.  Every object
in the multimorphism groupoid
$\Hom_{\thicken{\Cat}}(x_1,\dots,x_n;y)$ is a composition of basic
multimorphisms.
 
\begin{figure}
  \begin{tikzpicture}[xscale=2.5,yscale=3,every node/.style={inner sep=1pt,outer sep=0}]
    \foreach \i in {1,...,4}{
      \begin{scope}[xshift=1.1*\i cm]
        
        \coordinate (a1) at (0.2,0.2);
        \coordinate (a2) at (0.5,0.2);
        \coordinate (a3) at (0.8,0.2);
        \coordinate (a4) at (0.5,0.8);

        \draw (0.5,0.5) -- (a1) node[pos=1,anchor=north] {\small $(x,y)$};
        \draw (0.5,0.5) -- (a2) node[pos=1,anchor=north] {\small $(y,z)$};
        \draw (0.5,0.5) -- (a3) node[pos=1,anchor=north] {\small $(z,w)$};
        \draw (0.5,0.5) -- (a4) node[pos=1,anchor=south] {\small $(x,w)$};
        
        \path (0.5,0.5) -- (a\i) node[pos=0,circle,fill=black,minimum width=3pt] {} node[pos=0.5,circle,fill=black,minimum width=2pt] {};

      \end{scope}}
    \foreach \i in {1,...,4}{
      \begin{scope}[xshift=1.1*\i cm,yshift=-1cm]
        
        \coordinate (a1) at (0.2,0.2);
        \coordinate (a2) at (0.5,0.2);
        \coordinate (a3) at (0.8,0.2);
        \coordinate (a4) at (0.5,0.8);

        \draw (0.5,0.5) -- (a1) node[pos=1,anchor=north] {\small $(x,y)$};
        \draw (0.5,0.5) -- (a2) node[pos=1,anchor=north] {\small $(y,z)$};
        \draw (0.5,0.5) -- (a3) node[pos=1,anchor=north] {\small $(z,w)$};
        \draw (0.5,0.5) -- (a4) node[pos=1,anchor=south] {\small $(x,w)$};
        
        \coordinate (b1) at (0.3,0.4);
        \coordinate (b2) at (0.4,0.3);
        \coordinate (b3) at (0.6,0.3);
        \coordinate (b4) at (0.7,0.4);

        \draw (0.5,0.5) -- (b\i) node[pos=0,circle,fill=black,minimum width=3pt] {} node[pos=1,circle,fill=black,minimum width=2pt] {};

      \end{scope}}

    \foreach \i in {1,3}{
      \begin{scope}[xshift=5.5 cm,yshift=0.5cm-0.5*\i cm]
        \coordinate (a1) at (0.2,0.2);
        \coordinate (a2) at (0.5,0.2);
        \coordinate (a3) at (0.8,0.2);
        \coordinate (a4) at (0.5,0.8);

        \draw (0.5,0.5) -- (a1) node[pos=1,anchor=north] {\small $(x,y)$};
        \draw (0.5,0.5) -- (a3) node[pos=1,anchor=north] {\small $(z,w)$};
        \draw (0.5,0.5) -- (a4) node[pos=1,anchor=south] {\small $(x,w)$};
        
        \path (0.5,0.5) -- (a\i) node[pos=0,circle,fill=black,minimum width=3pt] {} node[pos=0.5,circle,fill=black,minimum width=2pt] (a) {};

        \draw (a) -- (a2) node[pos=1,anchor=north] {\small $(y,z)$};

      \end{scope}}
  \end{tikzpicture}
  \caption{\textbf{Some of the multimorphisms in
      $\Hom((x,y),(y,z),(z,w);(x,w))$ from
      Example~\ref{exam:enrichment}.} The edges are labeled by the
    objects and the internal vertices are labeled by multimorphisms in
    the original multicategory (i.e., basic multimorphisms). Edges
    ending in a node are stumps. The original multicategory being the
    shape multicategory of a set, the vertex labels and the internal
    edge labels are forced, and are not
    shown.}\label{fig:basic-enrichment}
\end{figure}
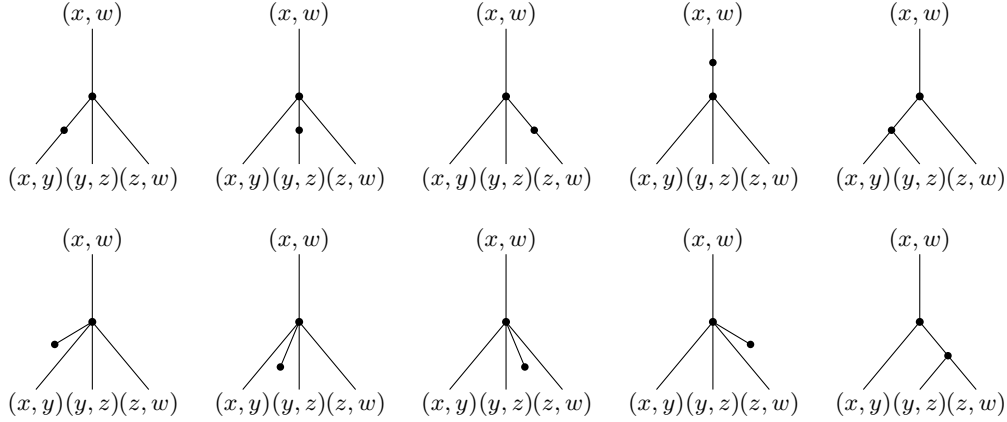

\begin{example}\label{exam:enrichment}
  Consider the canonical groupoid enrichment of the shape
  multicategory of some set $X$
  (cf.~Definition~\ref{def:shape-multicat-set}). For any $x,y,z,w\in
  X$, the multimorphism set $\Hom((x,y),(y,z),(z,w);(x,w))$ consists
  of infinitely many elements since the underlying tree could contain an
  arbitrary number of internal vertices. However, there is exactly one
  multimorphism when the underlying tree has exactly one internal
  vertex, exactly ten when the underlying tree has exactly two
  interval vertices (shown in Figure~\ref{fig:basic-enrichment}), and so on.
\end{example}

There is a canonical projection multifunctor $\thicken{\Cat}\to\Cat$
which is the identity on objects and composes the multimorphisms
associated to the vertices of a tree according to the edges. (Here, we
view $\Cat$ as trivially enriched in groupoids.)
\begin{lemma}\label{lem:unenrich}
  The projection map $\thicken{\Cat}\to\Cat$ is a weak equivalence.
\end{lemma}
(See~\cite[Definition 12.1]{EM-top-machine} for the definition of a
weak equivalence.)
\begin{proof}
  We must check that projection induces an equivalence on the
  categories of components and that for each $x_1,\dots,x_n,y$ the
  projection map gives a weak equivalence of simplicial nerves
  \begin{equation}\label{eq:nerve-proj}
    \mathscr{N}\Hom_{\thicken{\Cat}}(x_1,\dots,x_n;y)\to
    \mathscr{N}\Hom_{\Cat}(x_1,\dots,x_n;y).
  \end{equation}
  The first statement follows from the fact that
  the components of the
  groupoid $\Hom_{\thicken{\Cat}}(x_1,\dots,x_n;y)$ correspond, under
  the projection, to the elements of $\Hom_{{\Cat}}(x_1,\dots,x_n;y)$.
  The second statement follows from the fact that in each component of
  the multimorphism groupoid
  $\Hom_{\thicken{\Cat}}(x_1,\dots,x_n;y)$, every object is initial
  (and terminal), so
  $\mathscr{N}\Hom_{\thicken{\Cat}}(x_1,\dots,x_n;y)$ is contractible.
\end{proof}

A related construction is strictification: 
\begin{definition}\label{def:strictify}
  Given a multicategory $\Cat$ enriched in groupoids there is a
  \emph{strictification} $\Cat^0$ of $\Cat$, which is an
  ordinary multicategory, with objects
  $\Ob(\Cat^0)=\Ob(\Cat)$ and multimorphism sets
  $\Hom_{\Cat^0}(x_1,\dots,x_n;y)$ the set of isomorphism
  classes (path components) in the groupoid
  $\Hom_{\Cat}(x_1,\dots,x_n;y)$. If we view $\Cat^0$ as
  trivially enriched in groupoids then there is a projection
  multifunctor $\Cat\to\Cat^0$.
\end{definition}

Strictification is a left inverse to thickening, i.e., for
any non-enriched multicategory $\Cat$,
\[
  \strictify{\thicken{\Cat}}\cong\Cat.
\]

A more general notion than a multicategory enriched in groupoids is a
\emph{simplicial multicategory}, i.e., a multicategory enriched in
simplicial sets. Given a multicategory enriched in groupoids $\Cat$,
replacing each $\Hom$ groupoid $\Hom_{\Cat}(x,y)$ by its nerve gives a
simplicial multicategory. One can also \emph{strictify} a
simplicial multicategory $\Dat$ by replacing each $\Hom$ simplicial set by
its set of path components. If $\Dat$ came from a multicategory
$\Cat$ enriched in groupoids by taking nerves then the strictification
$\Cat^0$ of $\Cat$ and the strictification $\Dat^0$ of $\Dat$ are are naturally
equivalent. Our main reason for introducing simplicial multicategories
is that some of the background results we use are stated in that more
general language. For instance, spectra form a simplicial multicategory.

\subsection{Homotopy colimits}\label{sec:hocolim}

In this section we will discuss homotopy colimits in the categories of
simplicial sets and chain complexes.

Given an index category $I$ and a functor $F$ from $I$ to the category
$\SimplicialSets_*$ of based simplicial sets, there is a based homotopy
colimit denoted by $\hocolim_I F$: it is a quotient of the space
\[
\coprod_{p \geq 0} \coprod_{i_0 \to i_1 \to \dots \to i_p} F(i_0)
\smas (\Delta^p)_+
\]
by an equivalence relation induced by simplicial face and
degeneracy operations \cite[XII.2]{BK-top-book}. Similarly, if instead
we are given a functor $F$ from $I$ to the category $\Complexes$ of
complexes, there is a homotopy colimit $\hocolim_I F$ (denoted
$\coprod_* F$ in \cite{BK-top-book}): it is a quotient of the complex
\[
\bigoplus_{p \geq 0} \bigoplus_{i_0 \to i_1 \to \dots \to i_p} F(i_0)
\otimes C_*(\Delta^p),
\]
where $C_*$ is the normalized chain functor on simplicial sets. (More
explicit chain-level descriptions can be given.) In particular, the
natural commutative and associative Eilenberg-Zilber shuffle pairing
$\widetilde C_*(X) \otimes \widetilde C_*(Y) \to \widetilde C_*(X
\smas Y)$, applied to the above constructions,
gives rise to a natural transformation
$\hocolim (\widetilde C_* \circ F) \to \widetilde C_* (\hocolim F)$.

In the following, we use the shorthand \emph{equivalence} to denote both
a weak equivalence of simplicial sets and a quasi-isomorphism of
chain complexes.
\begin{proposition}\label{prop:hocolim-props}
  Homotopy colimits satisfy the following properties.
  \begin{itemize}
  \item Homotopy colimits are functorial: a natural transformation $F
    \to F'$ induces a map $\hocolim F \to \hocolim F'$ that makes hocolim
    functorial in $F$, and a map of diagrams $j\co I \to J$ induces a
    natural transformation $\hocolim (F \circ j) \to \hocolim F$ that 
    makes hocolim functorial in $I$.
  \item Homotopy colimits preserve equivalences: any natural
    transformation $F \to F'$ of functors such that $F(i) \to F'(i)$
    is an equivalence for all $i$ induces an equivalence $\hocolim F
    \to \hocolim F'$.
  \item For a diagram $F$ indexed by $I \times J$, there is a natural
    transformation
    \[
    \hocolim_{i \in I} (\hocolim_{j \in J} F(i \times j)) \to
    \hocolim_{I \times J} F
    \]
    coming from the (non-commutative) Alexander-Whitney pairing (not the
    commutative Eilenberg-Zilber shuffle pairing).  This is an
    isomorphism for a homotopy colimit in simplicial sets, and a
    quasi-isomorphism for a homotopy colimit in complexes. This is
    associative in $I$ and $J$, but not commutative.
  \item The reduced chain functor $\widetilde C_*$ preserves homotopy
    colimits: given a functor $F\co I \to \SimplicialSets_*$, the natural
    map $\hocolim (\widetilde C_* \circ F) \to
    \widetilde C_* (\hocolim F)$ is a quasi-isomorphism.
  \item The smash product $\smas$ and tensor product $\otimes$
    preserve homotopy colimits in each variable, and this is
    compatible with the Eilenberg-Zilber shuffle pairing.
  \end{itemize}
\end{proposition}
In particular, these combine to give a natural quasi-isomorphism
\[
(\hocolim_I F) \otimes (\hocolim_J G) \to \hocolim_{I \times J} (F
\otimes G)
\]
which is compatible with associativity (but not commutativity) of the
tensor product.

Homotopy colimits  in the category $\Complexes$ are closely related to
left derived functors. In the following, we view $\AbelianGroups$ as a
subcategory of $\Complexes$, given by the chain complexes
concentrated in degree zero.
\begin{proposition}\label{prop:hocolim-sseq-1}
  Homotopy colimits of complexes satisfy the following properties.
  \begin{itemize}
  \item Write $\AbelianGroups^I$ for the category of functors $I \to
    \AbelianGroups$ and $\colim_I$ for the colimit functor
    $\AbelianGroups^I \to \AbelianGroups$. Then there is a natural
    isomorphism between the left derived functor $\LL_p \colim_I(F)$
    and the homology group $H_p (\hocolim F)$, for each $p \geq 0$ \cite[XII.5]{BK-top-book}.
  \item For a functor $F\co I \to \Complexes$, there is a convergent
    spectral sequence
    \[
    \LL_p \colim_I(H_q \circ F) \Rightarrow H_{p+q}(\hocolim_I F).
    \]
  \item For a functor $F\co I \to \SimplicialSets_*$, there is a
    convergent spectral sequence
    \[
    \LL_p \colim_I(\widetilde H_q \circ F) \Rightarrow \widetilde
    H_{p+q}(\hocolim_I F)
    \]
    for the homology groups of a homotopy colimit \cite[XII.5.7]{BK-top-book}.
  \end{itemize}
\end{proposition}

\begin{proposition}[{\cite[XII.5.6]{BK-top-book}}]
  Suppose $\Delta$ denotes the category of finite ordinals and
  order-preserving maps, and $A\co \Delta^{\op} \to \Complexes$
  represents a simplicial chain complex $A_\bullet$. Then the chain
  complex $\hocolim_{\Delta^\op} A$ 
  is
  quasi-isomorphic to the total complex of the double complex
  \[
  \cdots \to A_2 \to A_1 \to A_0 \to 0,
  \]
  where the ``horizontal'' boundary maps are given by the standard
  alternating sum of the face maps of $A_\bullet$.
\end{proposition}
\begin{proposition}
  If $A$ is an abelian group, represented by a functor $F\co I \to
  \AbelianGroups$ from the trivial category with one object, then the
  complex $\hocolim_I F$ 
  is the complex
  with $A$ in degree $0$ and $0$ in all other degrees.
\end{proposition}
\begin{proposition}
  Suppose $I$ is a category and we have a natural transformation
  $\phi\co F \to G$ of functors $I \to \Complexes$. Let $P$ denote
  the category $\{* \leftarrow 0 \rightarrow 1\}$, and define a
  functor $C\phi\co P \times I \to \Complexes$ on objects by
  \[
  C\phi(x,y) = \begin{cases}
    0 &\text{if }x = *,\\
    F(y) &\text{if }x = 0,\\
    G(y) &\text{if }x = 1
  \end{cases}
  \]
  with morphisms determined by $F$, $G$, and $\phi$. Then the chain
  complex $\hocolim_{P\times I} (C\phi)$ 
  is quasi-isomorphic to the standard mapping cone of the
  map of chain complexes $\hocolim_I F \to \hocolim_I G$ induced by $\phi$.
\end{proposition}
Using the previous two propositions to iterate a mapping cone
construction gives the following result for cube-shaped diagrams.
\begin{corollary}\label{cor:totalization}
  Let $P$ denote the category $\{* \leftarrow 0 \to 1\}$ and $\CCat{}$
  denote the subcategory $\{0 \to 1\}$. Given a functor $F\co \CCat{n}
  \to \AbelianGroups$, its totalization is defined to be the
  chain complex
  \begin{equation}\label{eq:totalization}
  \bigoplus_{\substack{v\in\CCat{n}\\|v|=0}}F(v)\to\bigoplus_{\substack{v\in\CCat{n}\\|v|=1}}F(v)\to \dots\to\bigoplus_{\substack{v\in\CCat{n}\\|v|=n}}F(v),
  \end{equation}
  graded so that $\bigoplus_{\substack{v\in\CCat{n}\\|v|=i}}F(v)$ is
  in grading $n-i$ (where $|v|$ denotes the number of $1$'s in $v$),
  and the differential counts the sum of the edge maps of $F$ with
  standard signs.  Let $\widetilde F\from P^n\to\AbelianGroups$ be the
  extended functor given by
  \[
  \widetilde F(v) = \begin{cases}
    F(v)&\text{if }v \in \CCat{n},\\
    0&\text{otherwise.}
  \end{cases}
  \]
  Then the complex $\hocolim_{P^n} \tilde F$
  is quasi-isomorphic to the totalization of $F$.
\end{corollary}

\subsection{Classical spectra}\label{sec:spectra}

In this section we will review some of the models for the category of
spectra and some of the properties we will need.

For us, a \emph{classical spectrum} $X$ (sometimes called a \emph{sequential
spectrum}) is a sequence of based simplicial sets $X_n$, together with
structure maps $\sigma_n\co X_n \smas S^1 \to X_{n+1}$. A map $X \to
Y$ is a sequence of based maps $f_n\co X_n \to Y_n$ such that the
diagrams
\begin{equation}\label{eq:smash-square}
\mathcenter{\xymatrix{
X_n \smas S^1 \ar[d]_{f_n \smas \Id} \ar[r]^-{\sigma_n^X} &
X_{n+1} \ar[d]^{f_{n+1}} \\
Y_n \smas S^1 \ar[r]_-{\sigma_n^Y} & Y_{n+1}
}}
\end{equation}
all commute. The structure maps produce natural homomorphisms on homotopy groups
$\pi_k(X_n) \to \pi_{k+1}(X_{n+1})$ and (reduced) homology groups $\widetilde
H_k(X_n) \to \widetilde H_{k+1}(X_{n+1})$, allowing us to define
homotopy and homology groups
\begin{align*}
\pi_k(X) &= \colim_n \pi_{k+n} X_n  & H_k(X) &= \colim_n \widetilde H_{k+n} X_n
\end{align*}
for all $k \in \ZZ$ that are functorial in $X$. A map of classical
spectra $X \to Y$ is defined to be a \emph{weak equivalence} if it induces an
isomorphism $\pi_* X \to \pi_* Y$, and the \emph{stable homotopy category} is
obtained from the category of classical spectra by inverting the weak
equivalences. The functors $\pi_*$ and $H_*$ both factor through the
stable homotopy category. (This description is due to Bousfield and
Friedlander \cite{BF-top-spectra}, and they show that it gives a
stable homotopy category equivalent to the one defined by Adams
\cite{Ada-top-stablehomotopy}. It has the advantage that maps of
spectra are easier to describe, but the disadvantage that maps $X \to
Y$ in the stable homotopy category are not defined as
homotopy classes of maps $X \to Y$.)

Classical spectra $X$ and $Y$ have a \emph{handicrafted smash product}
given by
\[
(X \smas Y)_n =
\begin{cases}
  X_k \smas Y_k &\text{if }n=2k,\\
  (X_k \smas Y_k)\smas S^1 &\text{if }n=2k+1.
\end{cases}
\]
The structure map $(X \smas Y)_n \smas S^1 \to (X \smas Y)_{n+1}$ is
the canonical isomorphism when $n$ is even and is obtained from the
structure maps of $X$ and $Y$ when $n$ is odd. This smash product is
not associative or unital, but it induces a smash product functor that
makes the stable homotopy category symmetric monoidal. There is a
K\"unneth formula for homology: there is a multiplication pairing
$H_p(X) \otimes H_q(Y) \to H_{p+q}(X \smas Y)$ that is part of a
natural exact sequence
\[
0 \to \bigoplus_{p+q=n} H_p(X) \otimes H_q(Y) \to H_n(X \smas Y)
\to \bigoplus_{p+q=n-1} \Tor_1^{\ZZ}(H_p(X), H_q(Y)) \to 0
\]
that can be obtained by applying colimits to the ordinary K\"unneth
formula. In particular, this multiplication pairing is an isomorphism
if the groups $H_*(X)$ or $H_*(Y)$ are all flat over $\ZZ$.

Given a functor
$F$ from $I$ to the category of classical spectra, there is a homotopy
colimit $\hocolim_I F$ obtained by applying homotopy colimits
levelwise. Homotopy colimits preserve weak equivalences, and the
handicrafted smash product preserves homotopy colimits in each
variable. There is also a derived functor spectral sequence
\[
\LL_p \colim_I(H_q \circ F) \Rightarrow H_{p+q}(\hocolim_I F)
\]
for calculating the homology of a homotopy colimit. (In fact, this
spectral sequence exists for stable homotopy groups $\pi_*$ as well.)

The Hurewicz theorem for spaces translates into a Hurewicz theorem for
spectra:
\begin{definition}
  For an integer $n$, an object $X$ in the stable homotopy category is
  \emph{$n$-connected} if $\pi_k X = 0$ for $k \leq n$. If $n = -1$, we
  simply say that $X$ is \emph{connective}.
\end{definition}
\begin{citethm}
  There is a natural Hurewicz map $\pi_n(X) \to H_n(X)$, which is an
  isomorphism if $X$ is $(n-1)$-connected.
\end{citethm}
This induces a homology Whitehead theorem:
\begin{citethm}\label{citethm:homologywhitehead}
  If $f\co X \to Y$ is a map of spectra that induces an isomorphism
  $H_* (X) \to H_* (Y)$ and both $X$ and $Y$ are $n$-connected for some
  $n$, then $f$ is an equivalence.
\end{citethm}

Spectra have suspensions and desuspensions:
\begin{definition}\label{def:suspension}
  For a spectrum $X$, there are \emph{suspension} and \emph{loop functors}, as well
  as formal \emph{shift functors}, as follows:
  \begin{align*}
    (S^1 \smas X)_n &= S^1 \smas (X_n) &
    (\Omega X)_n &= \Omega(X_n)\\
    \sh(X)_n &= X_{n+1} &
    \sh^{-1}(X)_n &= \begin{cases}
      X_{n-1}&\text{if }n > 0\\
      \ast&\text{if }n=0
    \end{cases}
  \end{align*}
\end{definition}

\begin{proposition}\label{prop:sh-adjoints}
  The pairs $(S^1 \smas (-), \Omega)$ and $(\sh^{-1}, \sh)$ are adjoint
  pairs, and all unit and counit maps are weak equivalences.

  In the stable homotopy category, there are isomorphisms
  \begin{align*}
    S^1 \smas X &\cong \sh(X) & \Omega X &\cong \sh^{-1}(X)
  \end{align*}
  In particular, the suspension functor and desuspension (i.e., loop)
  functor are inverse to each other.
\end{proposition}

Although it may look like there are natural maps $S^1 \smas X \to \sh(X)$ and $\Omega X\to \sh^{-1}(X)$
that implement these equivalences, there are not: the apparent maps do
not make Diagram~\eqref{eq:smash-square} commute.

\subsection{Symmetric spectra}
Many of our constructions make use of Elmendorf-Mandell's paper~\cite{EM-top-machine},
which uses Hovey-Shipley-Smith's more structured category of symmetric
spectra~\cite{HSS-top-symmetric}. In this section we review some
details about symmetric spectra and their relationship to classical
spectra.

A \emph{symmetric spectrum} (which, in this paper, we may simply call a
\emph{spectrum}) is a sequence of based simplicial sets $X_n$, together with
actions of the symmetric group $\symGrp_n$ on $X_n$, and
structure maps $\sigma_n\co X_n \smas S^1 \to X_{n+1}$. These are
required to satisfy the following additional constraint. For any $n$
and $m$, the iterated structure map
\[
X_n \smas S^m \cong X_n \smas (S^1 \smas S^1 \smas \dots \smas S^1)
\to X_{n+m}
\]
has actions of $\symGrp_n \times \symGrp_m$ on the source and target:
via the actions on the two factors for the source, and via the
standard inclusion $\symGrp_n \times \symGrp_m \to \symGrp_{n+m}$ in
the target. The structure maps are required to intertwine these two
actions. A map of symmetric spectra consists of a sequence of based,
$\symGrp_n$-equivariant maps $f_n\co X_n \to Y_n$ commuting with the structure
maps. We write $\Spectra$ for the category of symmetric spectra.

A symmetric spectrum can also be described as the following equivalent
data. To a finite set $S$, a symmetric spectrum assigns a simplicial
set $X(S)$, and this is functorial in isomorphisms of finite sets. To
a pair of finite sets $S$ and $T$, there is a structure map $X(S) \smas
\left(\bigwedge_{t \in T}S^1\right) \to X(S \coprod T)$, and this is compatible with
isomorphisms in $S$ and $T$ as well as satisfying an associativity
axiom in $T$. We recover the original definition by setting $X_n =
X(\{1,2,\dots,n\})$.

Symmetric spectra also have a more rigid monoidal structure $\smas$,
characterized by the property that a map $X \smas Y \to Z$ is
equivalent to a natural family of maps $X(S) \smas Y(T) \to Z(S
\coprod T)$ compatible with the structure maps in both variables. This
makes the category of symmetric spectra symmetric monoidal
closed.

Again, the constructions of homotopy colimits are compatible enough
that they extend to symmetric spectra. Given a functor $F$ from $I$ to
the category of symmetric spectra, there is a homotopy colimit
$\hocolim_I F$ obtained by applying homotopy colimits
levelwise. Homotopy colimits preserve weak equivalences.  The smash
product also behaves well with respect to homotopy colimits, as
follows.
\begin{proposition}
  The smash product of symmetric spectra preserves homotopy colimits
  in each variable.
\end{proposition}

The category of symmetric spectra has an internal
notion of weak equivalence, and a homotopy category of symmetric
spectra. Both symmetric spectra and classical spectra have model
structures \cite{HSS-top-symmetric,BF-top-spectra}, and we have the
following results.
\begin{citethm}[{\cite[4.2.5]{HSS-top-symmetric}}]
  The forgetful functor $U$ from symmetric spectra to classical
  spectra has a left adjoint $V$, and this pair of adjoint functors is
  a Quillen equivalence between these model categories.
\end{citethm}
\begin{corollary}
  The homotopy category of symmetric spectra is equivalent to the
  stable homotopy category.
\end{corollary}
\begin{corollary}
  The equivalence between symmetric spectra and classical spectra
  preserves homotopy colimits.
\end{corollary}
Note that the forgetful functor $U$ does
not preserve weak equivalences except between certain symmetric
spectra, the so-called \emph{semistable} ones \cite[Section
5.6]{HSS-top-symmetric}. Any fibrant symmetric spectrum is semistable,
and any symmetric spectrum is weakly equivalent to a semistable one.
\begin{citethm}[{\cite[0.3]{MMSS-diagram-spectra}}]
  The equivalence between the homotopy category of symmetric spectra
  and the stable homotopy category preserves smash products.
\end{citethm}

\begin{remark}\label{remark:cofib-smash}
  In order for $X\smas Y$ to have the correct homotopy type, $X$ and $Y$
  should both be cofibrant symmetric spectra.
\end{remark}

These results allow us to define homotopy and homology groups for a
symmetric spectrum $X$ as a composite: take the image of $X$ in the
homotopy category of symmetric spectra; apply the (right) derived
functor of $U$ to get an element in the homotopy category of classical
spectra; and then apply homotopy or homology
groups. The homology groups of
symmetric spectra therefore inherit the following properties from
classical spectra.
\begin{proposition}
  For symmetric spectra $X$ and $Y$, there is a natural K\"unneth
  exact sequence
\[
0 \to \bigoplus_{p+q=n} H_p(X) \otimes H_q(Y) \to H_n(X \smas Y)
\to \bigoplus_{p+q=n-1} \Tor_1^{\ZZ}(H_p(X), H_q(Y)) \to 0.
\]
\end{proposition}
\begin{proposition}
  For a diagram $F\co I \to \Spectra$ of symmetric spectra, there is a
  convergent derived functor spectral sequence
\[
\LL_p \colim_I(H_q \circ F) \Rightarrow H_{p+q}(\hocolim_I F).
\]
\end{proposition}
It will be convenient for us to have a lift of these homology groups
to a chain functor. Let $L$ denote the reduced chain complex
$\widetilde C_*(S^1)$ of the simplicial set $S^1$. This is a complex
with value $\ZZ$ in degree $1$ and zero elsewhere. For complexes $C$
and $D$, let $\HomComplex(C,D)$ be the function complex. 
\begin{definition}
  Fix a symmetric spectrum $X$.  For an inclusion of finite sets
  $T \subset U$, there is a natural map
  \[
    \HomComplex(L^{\otimes T}, \widetilde C_*X(T)) \stackrel{\sim}{\too}
    \HomComplex(L^{\otimes T} \otimes L^{\otimes U\setminus T},
    \widetilde C_* X(T) \otimes
    L^{\otimes U \setminus T}) \to \HomComplex(L^{\otimes U}, \widetilde C_*X(U)).
  \]
  
  Now, given any set $S$ (infinite or not), these maps make the
  complexes $\HomComplex(L^{\otimes T}, \widetilde C_*X(T))$ into a
  directed system indexed by finite subsets $T\subset S$.  Define the
  chain complex
  \[
    \widehat C_k(X)_S = \colim_{T \subset S\text{ finite}} \HomComplex(L^{\otimes
      T}, \widetilde C_kX(T)).
  \]
\end{definition}
If $S$ is finite of size $n$, $\widehat C_k(X)_S$ is isomorphic to the
shift $\wt{C}_k X_n[-n]$.  More generally, these structure maps
naturally make the system of chain groups and homology groups
$\{H_{n+k}(X_n)\}$ into a functor from the category of finite sets and
injections to the category of abelian groups (i.e., an $FI$-module in the
language of~\cite{CEF-top-FImodules}).
 
There is a natural pairing
\[
\widehat C_*(X)_S \otimes \widehat C_*(Y)_T \to \widehat C_*(X \smas
Y)_{S \coprod T}.
\]
The construction of $\widehat C_*$ is also natural in injections $S
\to S'$.

\begin{definition}
  Let $\mathcal{M}$ be the category whose objects are the
  sets
  $\coprod^k \NN$ for $k \geq 1$, and whose morphisms
  are monomorphisms of sets. For a symmetric spectrum $X$, we define
  \[
    C_*(X) = \hocolim_{S \in \mathcal{M}} (\widehat C_*(X)_{S}).
  \]
\end{definition}
Let $M$ be the monoid of monomorphisms $\NN \to \NN$. Since all
objects in the category $\mathcal{M}$ are isomorphic to $\NN$, this
homotopy colimit is quasi-isomorphic to the homotopy colimit over this
one-object subcategory, which can be re-expressed as the derived tensor
product $\ZZ \otimes^{\LL}_{\ZZ[M]} \widehat
C_*(X)_\NN$. See \cite{Sch-top-htpygrp} and \cite[Exercise
E.II.13]{S-top-symmetricspectra} for a discussion of this functor.

\begin{proposition}
  The chain functor $C_*\co \Spectra \to \Complexes$ satisfies the
  following properties.
  \begin{itemize}
  \item The homology groups of $C_* X$ are the classical homology
    groups of the image of $X$ in the stable homotopy category.
  \item The associative disjoint union operation $\mathcal{M} \times
    \mathcal{M} \to \mathcal{M}$ gives rise to a natural
    quasi-isomorphism $\bigotimes C_*(X_i) \to C_*(\bigwedge X_i)$,
    which respects the associativity isomorphisms for $\smas$ and
    $\otimes$.
  \item The functor $C_*$ preserves homotopy colimits: for a diagram
    $F\co I \to \Spectra$, there is a natural quasi-isomorphism
    $\hocolim (C_* \circ F) \to C_*(\hocolim F)$.
  \end{itemize}
\end{proposition}

Therefore, if $\mSpectra$ denotes the associated multicategory of
symmetric spectra, $C_*$ induces a multifunctor
$\mSpectra\to\mComplexes$. To a multimorphism in symmetric spectra realized by a map 
$X_1\smas\cdots\smas X_n\to Y$, $C_*$ associates the chain map $C_*(X_1)\otimes\dots\otimes
C_*(X_n)\to C_*(X_1\smas\cdots\smas X_n)\to C_*(Y)$. This definition of $C_*$ respects
multi-composition. (The multifunctor $C_*$ is not compatible with the symmetries interchanging factors, if we regard $\mSpectra$ and $\mComplexes$ as symmetric multicategories.)\setcounter{OutlineCompletion}{20}

If we defined homotopy and homology groups
\begin{align*}
  \widehat\pi_k(X) &= \colim \pi_{k+n}(X_n) & \widehat H_k(X) &=
  \colim H_{k+n}(X_n)
\end{align*}
using the same formula as for classical spectra, we obtain ``na\"ive''
homotopy and homology groups of a symmetric spectrum $X$ which are not
preserved under weak equivalence. If we tensor with the sign
representation of $\symGrp_n$ and take $\colim H_{n+k}(X_n) \otimes
sgn$, the result is isomorphic to $\widehat H_k(X)_{\NN}$ with its
action of the monoid $M$ of injections $\NN \to
\NN$~\cite{Sch-top-htpygrp}. The natural map $\widehat H_k(X)
\to H_k(X)$ to the true homology groups factors through the
quotient by $M$. A similar action and factorization hold
relating the na\"ive homotopy groups $\widehat\pi_k(X)$ to the true
homotopy groups $\pi_k(X)$.

A similar warning holds for homotopy colimits. If $F$ is a diagram of
symmetric spectra, it is not the case that $U (\hocolim F) \simeq
\hocolim(U \circ F)$ unless $F$ is a diagram of semistable symmetric
spectra. However, it is always possible to replace $F$ with a weakly
equivalent diagram $F'$ of semistable symmetric spectra so that
$\hocolim F \simeq \hocolim F'$, and then $U (\hocolim F') \simeq
\hocolim(U \circ F')$.

Symmetric spectra have suspension and desuspension (i.e., loop) functors.
\begin{definition}
  For a symmetric spectrum $X$, there are suspension and loop
  functors, as well as formal shift functors, as follows:
  \begin{align*}
    (S^1 \smas X)_n &= S^1 \smas (X_n) &
    (\Omega X)_n &= \Omega(X_n)\\
    \sh(X)_n &= X_{1+n} &
    \sh^{-1}(X)_n &= \begin{cases}
      (\symGrp_{1+m})_+ \smas_{\symGrp_m} X_m&\text{if }n = 1+m\\
      \ast&\text{if }n=0
    \end{cases}
  \end{align*}
\end{definition}
The notation $1+n$ in the shift functor $sh$ indicates that the
$\symGrp_n$-action on $X_{1+n}$ is via the inclusion $\symGrp_1
\times \symGrp_{n} \to \symGrp_{1+n}$.

\begin{proposition}
  The pairs $(S^1 \smas (-), \Omega)$ and $(\sh^{-1}, \sh)$ are adjoint
  pairs, and all unit and counit maps are weak equivalences. 

  There are natural weak equivalences of symmetric spectra $S^1 \smas
  X \to \sh(X)$ and $\sh^{-1} X \to \Omega X$. These become equivalent
  to the standard shift functors in the stable homotopy category.
\end{proposition}
For example, the map $S^1 \smas X \to \sh(X)$ is the composite
\[
S^1 \smas X_n \cong X_n \smas S^1 \to X_{n+1} \stackrel{\sigma}{\to} X_{1+n},
\]
where the final map $\sigma$ is a block permutation in
$\symGrp_{n+1}$: this is necessary to ensure that this commutes
with the structure maps.

\begin{proposition}\label{prop:shift-fun}
  The suspension functor $S^1 \smas (-)$ and the formal shift functors
  preserve homotopy colimits. They also preserve smash products:
  there are natural isomorphisms
  \begin{align*}
    \sh(X) \smas Y &\to \sh(X \smas Y)
    & X \smas \sh(Y) &\to X \smas \sh(Y)\\
    (S^1 \smas X) \smas Y &\to S^1 \smas (X \smas Y) & 
    X \smas (S^1 \smas Y)&\to S^1 \smas (X \smas Y).
  \end{align*}
\end{proposition}
As with chain complexes, order matters in these identities. For
example, the two isomorphisms for $(S^1 \smas X) \smas (S^1 \smas Y)$
do not commute with each other, but differ by a transposition of $(S^1
\smas S^1)$; the two isomorphisms of $(\sh X) \smas (\sh Y)$ with
$\sh(\sh(X \smas Y))$ differ by a transposition in $\symGrp_{2+n}$.
\begin{proposition}
  There are natural isomorphisms $\HomComplex(L,C_*(\sh(X)) \to
  C_*(X)$ and $C_*(\sh^{-1}(X)) \to \HomComplex(L,C_*(X))$, as well as
  natural quasi-isomorphisms $C_*(S^1 \smas X) \to C_*(\sh(X))$.
\end{proposition}
In more standard notation, this implies that $C_*(\sh(X)) \cong
C_*(X)[1]$ and $C_*(\sh^{-1}(X)) \cong C_*(X)[-1]$. The isomorphism for
$\sh(X)$ is true before taking homotopy colimits for $\mathcal{M}$,
but the isomorphism for $\sh^{-1}$ is not.

\subsection{The Elmendorf-Mandell machine}\label{sec:EM-background}
A permutative category is a category $\Cat$ together with a
$0$-object, a strictly associative operation
$\oplus\co\Cat\times\Cat\to\Cat$, and a natural isomorphism $\gamma\co
a\oplus b\to b\oplus a$ satisfying certain coherence conditions
(see~\cite[Definition 3.1]{EM-top-machine}).  An example is the
category $\Sets/X$ of finite sets over $X$, with:
\begin{itemize}
\item Objects pairs $(Y,f\co Y\to X)$ of a finite set $Y$ and a map from $Y$ to $X$,
\item Morphisms $\Hom((Y,f),(Z,g))=\{h\co Y\to Z\mid f=g\circ h\}$,
\item Zero object the pair $(\emptyset, \iota)$ (where $\iota$ is the
  unique map $\emptyset\to X$), and
\item Sum $\oplus$ given by disjoint union.
\end{itemize}
The category $\Sets/X$ can be made small by requiring that all sets
$Y$ are elements of some chosen, large set. For instance, the objects
of $\Sets/X$ could be pairs $(n,S)$ where $n\in\NN$ and $S$ is a
finite subset of $\RR^n$ mapping to $X$. Moreover, the disjoint union
operation may be made strictly associative by declaring objects to be
finite sequences of such pairs $(n_1,S_1),\dots,(n_k,S_k)$ (which
morally represents their disjoint union), the actual sum $\oplus$
given by concatenation of sequences, and the morphisms from
$(n_1,S_1),\dots,(n_k,S_k)$ to $(m_1,T_1),\dots,(m_\ell,T_\ell)$ are
given by maps $\amalg_{i=1}^k S_i\to\amalg_{i=1}^\ell T_i$ (for any
standard definition of disjoint union), respecting the maps to $X$. We
will elide these points, but see~\cite{Isbell-other-coherent} for a
more detailed account.

Given a finite correspondence $A\co X\to Y$, i.e., a finite set $A$
and a map $(\pi_X,\pi_Y)\co A\to X\times Y$, there is a corresponding
functor of permutative categories
\begin{align*}
  F_A&\co \Sets/X\to \Sets/Y\\
  F_A&(Z,f)=(A\times_X Z,\pi_Y)=\bigl(\{(a,z)\in A\times Z\mid \pi_X(a)=f(z)\},(a,z)\mapsto \pi_Y(a)\bigr).
\end{align*}

The collection of all (small) permutative categories forms a
simplicial multicategory $\PermuCat$
\cite[Definition 3.2]{EM-top-machine}. The
category $\mSpectra$ of symmetric spectra also forms a simplicial multicategory,
and Elmendorf-Mandell construct an enriched multifunctor, \emph{$K$-theory}, 
\[
  K\co \PermuCat\to\mSpectra.
\]
Their
functor $K$ takes the category $\Sets/X$ to $\bigvee_{x\in X}\SphereS$, a
wedge of copies of the sphere spectrum. Further, given a
correspondence $A$ from $X$ to $Y$, the induced map $K(A)\co K(X)\to
K(Y)$ sends $\SphereS_x$ to $\SphereS_y$ (for $x\in X$, $y\in Y$) by a
map of degree $\#\bigl(\pi_X^{-1}(x)\cap \pi_Y^{-1}(y)\bigr)$. (This
special case can be understood concretely, using the Pontrjagin-Thom
construction; see, for example,~\cite[Section
5]{LLS-khovanov-product}.)

We note that that $K$ is invariant under equivalence in the following
sense. Because $K$ respects the enrichments of $\PermuCat$ and
$\mSpectra$ in simplicial sets, it takes natural isomorphisms
between functors of permutative categories to homotopies between maps
of $K$-theory spectra. Therefore, equivalent permutative categories give
homotopy equivalent answers.

This concludes our general introduction to Elmendorf-Mandell's $K$-theory machine.
In the rest of this section, we discuss a precise sense in
which multifunctors from different multicategories can be
equivalent. This will be used in Section~\ref{sec:build-spec-bim} to
replace multifunctors from floppy multicategories (enriched in
groupoids) with multifunctors from more rigid (unenriched)
multicategories.

\begin{definition}[cf. {\cite[2.0.0.1]{Lurie-top-HA}}]
  Suppose $I$ is a multicategory. The \emph{associated monoidal
    category} $I^\otimes$ is the category defined as follows. An
  object of $I^\otimes$ is a (possibly empty) tuple $(i_1,\dots,i_n)$
  of objects of $I$. The maps $(i_1,\dots,i_n) \to (j_1,\dots,j_m)$
  are given by
  \[
  \coprod_{\alpha\co \{1,\dots,n\} \to \{1,\dots,m\}}
  \prod_{k=1}^m\Hom_I(\alpha^{-1}(j_k); j_k).
  \]
  The monoidal structure on $I^\otimes$ is given by concatenation of
  tuples, with unit given by the empty tuple.
\end{definition}

\begin{definition}
  Given multicategories $I$ and $J$ and multifunctors $f\co I\to J$ and
  $G\co I\to \mSpectra$ there is a map $f_*G\co J\to\mSpectra$, the
  \emph{left Kan extension of $G$}, defined on objects by
  \begin{equation}\label{eq:colimit}
    (f_*G)(j)=\colim_{[(f(i_1),\dots,f(i_n)\to j]\in I^\otimes\downarrow j}G(i_1)\wedge \cdots\wedge G(i_n).
  \end{equation}
  (Here $I^\otimes\downarrow j$ denotes the overcategory of $j$.)
  Left Kan extension is functorial in $G$, i.e., gives a functor of
  diagram categories $f_*\co \mSpectra^I\to\mSpectra^J$.

  There is also a restriction map $f^*\co \mSpectra^J\to\mSpectra^I$, and
  $f_*$ is left adjoint to $f^*$.
\end{definition}

Following Elmendorf-Mandell~\cite[Definition 12.1]{EM-top-machine}, a
map $f\co\cM\to\cN$ between simplicial multicategories
is a \emph{(weak)
  equivalence} if the induced map on the strictifications
$f^0\co\cM^0\to \cN^0$ is an equivalence of (ordinary) categories and
for any $x_1,\dots,x_n,y\in\Ob(\cM)$, the map
$\Hom_{\cM}(x_1,\dots,x_n;y)\to\Hom_{\cN}(f(x_1),\dots,f(x_n);f(y))$
is a weak equivalence of simplicial sets.

A key technical result of Elmendorf-Mandell's is:
\begin{citethm}[{\cite[Theorems 1.3 and 1.4]{EM-top-machine}}]\label{citethm:EM-top-machine34}
  Let $\cM$ be a simplicial multicategory. Then the functor
  categories $\mSpectra^\cM$ and $\mSpectra^{\cM^0}$ are simplicial
  model categories with weak equivalences (respectively fibrations)
  the maps which are objectwise weak equivalences (respectively
  fibrations).

  Further, suppose $\cN$ is another simplicial multicategory and
  \[
    f\co \cM\to \cN
  \]
  is an equivalence. Then there are Quillen equivalences
  \[
    \xymatrix{
      \mSpectra^{\cM}\ar@/^/[r]^{f_*} & \mSpectra^{\cN},\ar@/^/[l]^{f^*}
    }
  \]
  where $f_*$ is left Kan extension and $f^*$ is restriction.
\end{citethm}
For instance, in Theorem~\ref{citethm:EM-top-machine34}, $\cM$ might
be (the nerve of) a multicategory enriched in groupoids whose every
component is contractible, and $\cN$ might be (the nerve of) its strictification
$\cM^0$.

We will need some additional cofibrancy for the rectification results
we apply (see Section~\ref{sec:rect-background}). In particular,
Elmendorf-Mandell also show that $\mSpectra^{\cM}$ is cofibrantly
generated \cite[Section 11]{EM-top-machine} and it is combinatorial in
the sense of \cite[Definition A.2.6.1]{Lurie-top-htt}.  Using a small
object argument, Chorny~\cite{Chorny-top-small} constructs functorial
cofibrant factorizations that apply, in particular, to combinatorial
model categories such as $\mSpectra$. So, his construction gives a cofibrant
replacement functor
\[
  Q^\cM\co \mSpectra^{\cM}\to\mSpectra^{\cM}.
\]
His construction satisfies the following property:
\begin{proposition}\label{prop:Chorny}
  Suppose $j\co \cN\hookrightarrow \cM$ is a full subcategory such
  that $\Hom(m_1,\dots,m_k;n)=\emptyset$ if $n\in \cN$ and
  $m_i\not\in \cN$ for some $i$. (That is, there are no arrows into
  $\cN$; we call such a full subcategory $\cN$ \emph{blockaded}.) Then
  the small object argument is preserved by restriction: there
  is a natural isomorphism
  \[
    j^*Q^{\cM}\stackrel{\simeq}{\longrightarrow}Q^{\cN}j^*.
  \]
\end{proposition}

We note that various operations preserve cofibrancy.

\begin{lemma}\label{lem:cofib-behaves}
  If $X$ is a cofibrant symmetric spectrum then $\sh(X)$ and
  $\sh^{-1}(X)$ are also cofibrant. Further, if $F\co I \to \Spectra$
  is a diagram of symmetric spectra which is pointwise cofibrant
  (i.e., $F(x)$ is cofibrant for all $x\in\Ob(I)$) then $\hocolim F$
  is cofibrant.
\end{lemma}
\begin{proof}
  This is mechanical to verify from the definitions in
  \cite[Section 3.4]{HSS-top-symmetric}, because shifts of the
  generating cofibrations are cofibrations.
\end{proof}

\begin{lemma}\label{lem:cofib-implies-levelwise}
  If $\cM$ is a multicategory and $F\co \cM\to\mSpectra$ is cofibrant
  then for each object $x\in\Ob(\cM)$, $F(x)$ is a cofibrant spectrum.
\end{lemma}
\begin{proof}
  The functor $ev_x\co \mSpectra^\cM \to \mSpectra$, given by $F
  \mapsto F(x)$, has a right adjoint given by right Kan
  extension. Given a symmetric spectrum $X$, the value of this right
  Kan extension on an object $y$ is
  \[
    \prod_{n \geq 0} X^{\Hom_{\cM}(x,x,\dots,x;y)}.
  \]
  In particular, any fibration $X \to Y$ becomes a fibration on
  applying right Kan extension. Therefore, $ev_x$ is a left Quillen
  functor and so preserves cofibrations and cofibrant objects.
\end{proof}

\subsection{Rectification}\label{sec:rect-background}
In the process of defining the arc algebras and tangle invariants, we
will construct a number of cobordisms which are not equal but are
canonically isotopic. The lax nature of the construction will be
encoded by defining multifunctors from multicategories in which the
$\Hom$ sets are groupoids in which each component is contractible: the
objects in the groupoids are mapped to the cobordisms while the
morphisms in the groupoids are mapped to the isotopies, and
contractibility of the groupoids encodes the fact that these isotopies
are canonical. We then use the Khovanov-Burnside functor and the
Elmendorf-Mandell machine to produce functors from these
multicategories to spectra. At that point, we want to collapse the
enriched multicategories to ordinary multicategories, to obtain
simpler invariants. This collapsing is called rectification,
and is accomplished as follows.

\begin{definition}\label{def:rectification}
  Let $\cM$ be a simplicial multicategory (e.g., the nerve of a
  multicategory enriched in groupoids), $\cM^0$ the strictified
  (discrete) multicategory, and $f\co\cM\to\cM^0$ the
  projection. Given a functor $G\co \cM\to \mSpectra$, the
  \emph{rectification of $G$} is the composite
  \[
    f_*Q^\cM G\co \cM^0\to\mSpectra.
  \]
\end{definition}

\begin{lemma}\label{lem:rect-equiv}
  If the projection map $\cM\to\cM^0$ is an equivalence then
  rectification is part of a Quillen equivalence. In particular, if the
  projection is an equivalence then for any $G\co \cM\to\mSpectra$, the
  functors $G$ and $f^*f_*Q^\cM G\co \cM\to\mSpectra$ are naturally
  equivalent.
\end{lemma}
\begin{proof}
  By definition of cofibrant replacement, the natural transformation
  $Q^{\cM} G \to G$ is an equivalence of diagrams: for every object in
  $x \in \cM$ the map $(Q^{\cM} G)(x) \to G(x)$ is an
  equivalence. Thus it suffices to show that the unit map from
  $Q^{\cM} G$ to $f^* f_* Q^\cM G$ is an equivalence.
  
  By Theorem~\ref{citethm:EM-top-machine34}, the adjoint pair $f^*$
  and $f_*$ form a Quillen equivalence. This implies that
  for any fibrant replacement $f_* Q^{\cM} G \to (f_* Q^{\cM} G)_{\mathit{fib}}$ in
  $\mSpectra^{\cM^0}$, the composite
  \[
    Q^{\cM} G \to f^* f_* Q^{\cM} G \to f^* (f_* Q^{\cM} G)_{\mathit{fib}}
  \]
  is an equivalence. For every object $x \in \cM$ the composite
  \[
    (Q^{\cM} G)(x) \to (f_* Q^{\cM} G)(f(x)) \to (f_* Q^{\cM} G)_{\mathit{fib}}(f(x))
  \]
  is therefore an equivalence. However, by definition of fibrant
  replacement the map
  \[(f_* Q^{\cM} G)(y) \to (f_* Q^{\cM} G)_{\mathit{fib}}(y)\]
  is an equivalence
  for any $y \in \cM^0$, and hence $Q^{\cM}(G) \to f^* f_* Q^{\cM} G$
  is also an equivalence by the 2-out-of-3 property. 
\end{proof}

\begin{lemma}\label{lem:rectify-restrict}
  Suppose that $j\co \cN\into\cM$ is a blockaded subcategory and let
  $j^0\co\cN^0\to\cM^0$ denote the strictification.  For any functor
  $G\co \cM\to\mSpectra$, there is a natural isomorphism of
  rectifications
  \[
    f^\cN_*Q^\cN j^*G\cong (j^0)^*f^{\cM}_*Q^\cM G.
  \]
\end{lemma}
\begin{proof}
  There is a natural transformation
  $f^\cN_*j^*G \to (j^0)^*f^\cM_*G$, the \emph{mate}.  Note that if
  $K\subset I$ is blockaded and $j\in K$ then the colimit in
  Equation~\eqref{eq:colimit} only sees the objects of $K$. Thus, the
  mate is a natural isomorphism
  \[
    f^\cN_*j^*G \cong (j^0)^*f^\cM_*G
  \]
  (i.e., satisfies the Beck-Chevalley condition). So, the result
  follows from Proposition~\ref{prop:Chorny}.
\end{proof}

\subsection{Khovanov invariants of tangles}\label{sec:review}
\begin{convention}\label{convention:vertical}
  All embedded cobordisms will be assumed to be the same as the product
  cobordism in some neighborhood of the boundary.
\end{convention}
\begin{definition}
  Let $\Diff^1$ denote the group of orientation-preserving
  diffeomorphisms $\phi\co [0,1]\to[0,1]$ so that there is some
  $\epsilon=\epsilon(\phi)>0$ so that
  $\phi|_{[0,\epsilon)\cup(1-\epsilon,1]}=\Id$. This restriction that
  $\phi$ be the identity near the boundary is similar to
  Convention~\ref{convention:vertical}.
\end{definition}
\begin{definition}
Let $\Diff^2$ denote the group of orientation-preserving
diffeomorphisms $\phi\from [0,1]^2\to[0,1]^2$ so that there is some
$\epsilon=\epsilon(\phi)>0$ and some $\psi_0,\psi_1\in\Diff^1$ so that
$\phi|_{[0,1]\times([0,\epsilon)\cup(1-\epsilon,1])}=\Id$, and
$\phi(p,q)=(p,\psi_0(q))$ for all $p\in[0,\epsilon)$, and
$\phi(p,q)=(p,\psi_1(q))$ for all $p\in(1-\epsilon,1]$.
\end{definition}
See Figure~\ref{fig:diff-act} for examples of the actions of elements
in $\Diff^1$ and $\Diff^2$.

\begin{figure}
  \centering
  \includegraphics{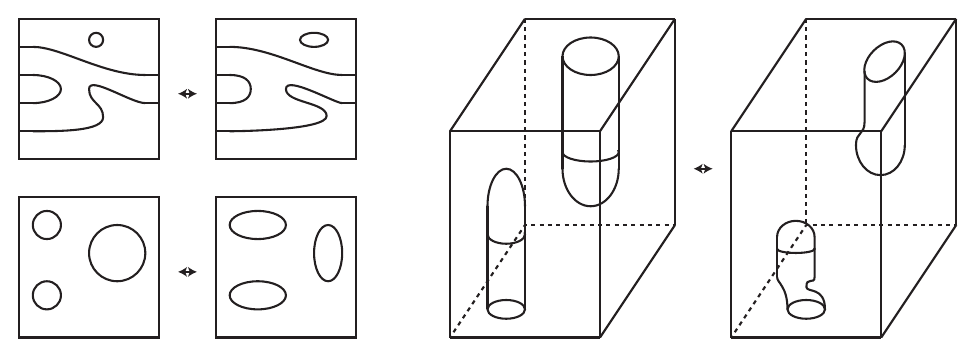}
  \caption{\textbf{The actions of $\Diff^1$ and $\Diff^2$.} Left: two
    flat $(4,2)$-tangles related by the action of $\Diff^1\times\Id$ and two
    flat $(0,0)$-tangles related by the action of $\Diff^1\times\Id$. Right:
    two flat $(0,0)$-tangle cobordisms related by the action of
    $\Diff^2\times\Id$.}
  \label{fig:diff-act}
\end{figure}

By the \emph{$2n$ standard points in $(0,1)$} we mean
$[2n]_\std=\{1/(2n+1),\dots,2n/(2n+1)\}$.  A \emph{flat
  $(2m,2n)$-tangle} is an embedded cobordism in $[0,1]\times(0,1)$
from $\{0\}\times[2m]_{\std}$ to $\{1\}\times[2n]_{\std}$. More
generally, a \emph{$(2m,2n)$-tangle} is an embedded cobordism in
$\RR\times[0,1]\times(0,1)$ from $\{0\}\times\{0\}\times[2m]_{\std}$
to $\{0\}\times\{1\}\times[2n]_{\std}$. We call flat tangles $T$ and
$T'$ \emph{equivalent} if there is a $\phi\in \Diff^1$ so that
$T'=(\phi\times\Id_{(0,1)})(T)$. Similarly, tangles $T$ and $T'$ are
\emph{equivalent} if there is a $\phi\in \Diff^1$ so that
$T'=(\Id_{\RR}\times\phi\times\Id_{(0,1)})(T)$.

\begin{convention}
  From now on, by \emph{tangle} (respectively \emph{flat tangle)} we
  mean an equivalence class of tangles (respectively flat tangles).
\end{convention}

\begin{remark}
  We are writing tangles horizontally, while
  Khovanov~\cite{Kho-kh-tangles} (and many others) writes tangles
  vertically.
\end{remark}

Khovanov~\cite{Kho-kh-tangles} associated an algebra $H^n$ to each
integer $n$; an $(H^m,H^n)$-bimodule $\KTfunc(T)$ to a
flat $(2m,2n)$-tangle T; and more generally a chain complex of
$(H^m,H^n)$-bimodules to any $(2m,2n)$-tangle.
We will review Khovanov's construction briefly. Because we reserve
$H^n$ for singular cohomology, we will use the notation $\KTalg{n}$
for Khovanov's algebra $H^n$.

The constructions start from Khovanov's Frobenius algebra
$V=H^*(S^2)=\ZZ[X]/(X^2)$ with comultiplication $1\mapsto 1\otimes X +
X\otimes 1,X\mapsto X\otimes X$ and counit $1\mapsto 0,X\mapsto
1$.

Let $\FlatTangles{m}{n}$ denote the collection of flat
$(2m,2n)$-tangles. Composition of flat tangles, followed by scaling
$[0,2]\times(0,1)\to[0,1]\times(0,1)$, is a map
$\FlatTangles{m}{n}\times\FlatTangles{n}{p}\to\FlatTangles{m}{p}$,
which we will write $(a,b)\mapsto ab$. (This map is associative and
has strict identities because we quotiented by $\Diff^1$.)
Reflection is a map $\FlatTangles{m}{n}\to \FlatTangles{n}{m}$, which
we will write $a\mapsto\Wmirror{a}$.

The isotopy classes of $\FlatTangles{0}{n}$ with no closed components
are called \emph{crossingless matchings}. For each crossingless
matching $a$, we choose a namesake representative
$a\subset [0,1]\times(0,1)$ in $\FlatTangles{0}{n}$ so that the
projection $a\to[0,1]$ to the $x$-coordinate is Morse with exactly $n$
critical points with distinct critical values; therefore, we may view
the set of crossingless matchings, $\Crossingless{n}$, as a subset of
$\FlatTangles{0}{n}$.

\begin{figure}
  \centering
  \includegraphics{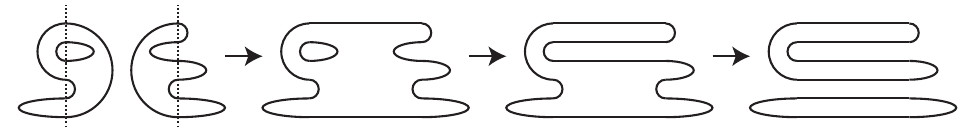}
  \caption{\textbf{Flat tangles and the multiplication on $\KTalg{n}$.}}
  \label{fig:flatTangles}
\end{figure}

Given a collection of disjoint, embedded circles $Z$ in the plane, let
$V(Z)=\bigotimes_{C\in\pi_0(Z)}V$.  As a $\ZZ$-module, the ring
$\KTalg{n}$ is given by
\[
\KTalg{n}=\bigoplus_{a,b\in \Crossingless{n}}V(a\Wmirror{b}).
\]
The product on $\KTalg{n}$ satisfies $xy=0$ if $x\in V(a\Wmirror{b})$
and $y\in V(c\Wmirror{d})$ with $b\neq c$. To define the product
$V(a\Wmirror{b})\otimes V(b\Wmirror{c})\to V(a\Wmirror{c})$, consider
the representative $b\subset [0,1]\times(0,1)$ and let
$\mu_1,\dots,\mu_n$ be the critical points of the projection
$b\to[0,1]$, ordered according to the critical values. Define a
sequence of $(2n,2n)$-tangles $\gamma_i$, $i=0,\dots, n$, inductively
by setting $\gamma_0=\Wmirror{b}b$ and obtaining $\gamma_{i+1}$ by
performing embedded surgery on $\gamma_i$ along an arc connecting
$\Wmirror{\mu_{i+1}}$ and $\mu_{i+1}$. (See Figure~\ref{fig:flatTangles}.)
Observe that $\gamma_n$ is canonically isotopic to the identity tangle
on $2n$ strands.  The Frobenius structure on $V$ induces a map
$V(a\gamma_i\Wmirror{c})\to V(a\gamma_{i+1}\Wmirror{c})$; define the
product $V(a\Wmirror{b})\otimes V(b\Wmirror{c})\to V(a\Wmirror{c})$ to
be the composition
\[
V(a\Wmirror{b})\otimes V(b\Wmirror{c})\cong V(a\gamma_0\Wmirror{c})\to V(a\gamma_1\Wmirror{c})\to\cdots\to V(a\gamma_n\Wmirror{c})\cong V(a\Wmirror{c}).
\]

\begin{lemma}[{\cite[Proposition 1]{Kho-kh-tangles}}]\label{lem:Kh-alg-defined}
  The multiplication just defined is associative and unital, and is
  independent of the choice of the representative in
  $\FlatTangles{0}{n}$ of the $b\in\Crossingless{n}$.
\end{lemma}
\begin{proof}[Sketch of proof]
  The key point is that a Frobenius algebra is the same as a
  $(1+1)$-dimensional topological field theory. Multiplication is
  induced by certain collections of saddle cobordisms, described more
  explicitly and called multi-merge cobordisms
  in Section~\ref{sec:cabinet}.
  Up to homeomorphism these cobordisms are independent of the
  choices of ordering of the saddles, and a composition of these multi-merge
  cobordisms is another multi-merge cobordism. (Units are also induced by canonical cup cobordisms.)
\end{proof}

Given a flat $(2m,2n)$-tangle $T\in\FlatTangles{m}{n}$, the bimodule
$\KTfunc(T)$ is given additively by
\[
\KTfunc(T)=\bigoplus_{(a,b)\in\Crossingless{m}\times\Crossingless{n}}V(aT\Wmirror{b}).
\]
The left action of $\KTalg{m}$ (respectively, the right action of
$\KTalg{n}$) is defined similarly to the multiplication on
$\KTalg{n}$: multiplication sends $V(a\Wmirror{b})\otimes
V(cT\Wmirror{d})$ to $0$ unless $b=c$ (respectively, sends
$V(cT\Wmirror{d})\otimes V(e\Wmirror{f})$ to $0$ unless $d=e$), and
the product $V(a\Wmirror{b})\otimes V(bT\Wmirror{c})\to
V(aT\Wmirror{c})$ (respectively, $V(bT\Wmirror{c})\otimes
V(c\Wmirror{d})\to V(bT\Wmirror{d})$) is defined by a sequence of
merge and split maps, turning the tangle $\Wmirror{b}b$
(respectively, $\Wmirror{c}c$) into the identity tangle.

\begin{lemma}[{\cite[Section 2.7]{Kho-kh-tangles}}]
  The bimodule structure on $\KTfunc(T)$ is independent of the
  choices in its construction and defines an associative, unital action.
\end{lemma}
\begin{proof}[Sketch of proof]
  Like Lemma~\ref{lem:Kh-alg-defined}, this follows from the fact that
  these operations are induced by cobordisms which, up to
  homeomorphism, themselves satisfy the associativity and unitality
  axioms.
\end{proof}

\begin{figure}
  \centering
  \begin{overpic}[scale=.7, tics=10]{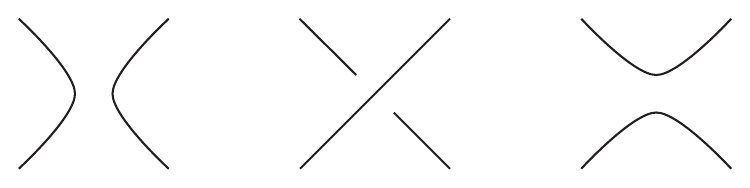}
    \put(11.5,0) {$0$}
    \put(86,0) {$1$}
  \end{overpic}
  \caption{\textbf{Resolutions of a crossing.}}
  \label{fig:resolutions}
\end{figure}

Now let $\Tangles{m}{n}$ denote the collection of all
$(2m,2n)$-tangles in $\RR\times[0,1]\times(0,1)$, with each component
oriented. Call such a tangle \emph{generic} if its
projection to $[0,1]\times(0,1)$ has no cusps, tangencies, or
triple points. A \emph{tangle diagram} is a generic tangle along with
a total ordering of its crossings (double points of the projection to
$[0,1]\times(0,1)$). Let $\TangleDiags{m}{n}$ be the set of
all $(2m,2n)$-tangle diagrams. (Forgetting the ordering of the
crossings, followed by an inclusion, gives a map
$\TangleDiags{m}{n}\to\Tangles{m}{n}$.)

Given a $(2m,2n)$-tangle diagram $T\in\TangleDiags{m}{n}$ with $N$
(totally ordered) crossings, and any crossingless matchings
$a\in\Crossingless{m}$ and $b\in\Crossingless{n}$, there is a
corresponding link $aT\Wmirror{b}$, which has an associated Khovanov
complex $\KhCx(aT\Wmirror{b})$. Additively, $\KhCx(aT\Wmirror{b})$ is
a direct sum over the complete resolutions $T_v$, $v\in\{0,1\}^N$,
of $V(aT_v\Wmirror{b})$. (Our conventions for resolutions are
shown in Figure~\ref{fig:resolutions}.) Thus,
\[
\KTfunc(T)\coloneqq \bigoplus_{\substack{(a,b)\in \Crossingless{m}\times\Crossingless{n}\\v\in\{0,1\}^N}}V(aT_v\Wmirror{b})
\]
inherits the structure of a chain complex, as a direct sum over the
$a$ and $b$ of $\KhCx(aT\Wmirror{b})$, and of a bimodule over $\KTalg{m}$ and
$\KTalg{n}$, as a direct sum over $v$ of $\KTfunc(T_v)$.

\begin{lemma}[{\cite[Section 3.4]{Kho-kh-tangles}}]
  The differential and bimodule structures on $\KTfunc(T)$ commute,
  making $\KTfunc(T)$ into a chain complex of bimodules.
\end{lemma}
\begin{proof}[Sketch of proof]
  Again, this follows from the fact that both the differential and
  multiplication are induced by Khovanov's TQFT, and the cobordisms
  inducing the differential and the multiplication commute up to
  homeomorphism. Indeed, this is a kind of far-commutation: the
  non-identity portions of the cobordisms inducing multiplication and
  differentials are supported over different regions of the diagram.
\end{proof}

These chain complexes of bimodules have the following TQFT property:
\begin{proposition}[{\cite[Proposition 13]{Kho-kh-tangles}}]\label{prop:Kh-pairing}
  If $T_1\in\TangleDiags{m}{n}$ is a $(2m,2n)$-tangle diagram and
  $T_2\in\TangleDiags{n}{p}$ is a $(2n,2p)$-tangle diagram, then the
  complexes of $(\KTalg{m},\KTalg{p})$-bimodules $\KTfunc(T_1T_2)$ and
  $\KTfunc(T_1)\otimes_{\KTalg{n}}\KTfunc(T_2)$ are isomorphic.
\end{proposition}
\begin{proof}[Sketch of proof]
  Suppose $T_1$ has $N_1$ crossings and $T_2$ has $N_2$ crossings. Then the isomorphism 
  \[
    \KTfunc(T_1)\otimes_{\KTalg{n}}\KTfunc(T_2)\stackrel{\cong}{\longrightarrow}
    \KTfunc(T_1T_2)
  \]
  identifies the summand of
  $\KTfunc(T_1)\otimes_{\KTalg{n}}\KTfunc(T_2)$ over the vertices
  $v\in \{0,1\}^{N_1}$ and $w\in\{0,1\}^{N_2}$ with the summand of
  $\KTfunc(T_1T_2)$ over $(v,w)\in\{0,1\}^{N_1+N_2}$. For these flat
  tangles $T_{1,v}$, $T_{2,w}$, and $(T_1T_2)_{(v,w)}$, the gluing map
  \[
    \KTfunc(T_{1,v})\otimes_{\KTalg{n}}\KTfunc(T_{2,w})\to \KTfunc((T_1T_2)_{(v,w)})
  \]
  is induced by the multi-saddle cobordism (cf.~Section~\ref{sec:cabinet}) map
  \[
    \KTfunc(aT_{1,v}\Wmirror{b})\otimes_{\ZZ}\KTfunc(bT_{2,w}\Wmirror{c})\to \KTfunc(a(T_1T_2)_{(v,w)}\Wmirror{c})
  \]
  \cite[Theorem 1]{Kho-kh-tangles}. 
\end{proof}

\begin{proposition}[{\cite[Theorem 2]{Kho-kh-tangles}}]\label{prop:Kh-invariance}
  For any tangle diagram $T\in\TangleDiags{m}{n}$, the chain homotopy
  type of the chain complex of bimodules $\KTfunc(T)$ is an invariant
  of the isotopy class of $T$ viewed as a tangle in $\Tangles{m}{n}$.
\end{proposition}

For comparison with our constructions later, note that each of the
$1$-manifolds $a\Wmirror{b}$ in the construction of $\KTalg{n}$ lies
in $(0,1)^2\subset[0,1]\times(0,1)$; and so does each of the
$1$-manifolds $aT\Wmirror{b}$ in the construction of $\KTfunc(T)$ for
a flat tangle $T$. There is a disjoint union operation on embedded
$1$-manifolds in $(0,1)^2$ induced by the map
\[
  (0,1)^2\amalg (0,1)^2\to(0,1)^2
\]
which identifies the first copy of $(0,1)^2$ with $(0,1/2)\times(0,1)$
and the second copy of $(0,1)^2$ with $(1/2,1)\times(0,1)$, by affine
transformations.  Since we have quotiented by the action of
$\Diff^1$ on the first $(0,1)$-factor, this disjoint union
operation is strictly associative.  Further, we can view the maps
inducing the multiplication on $\KTalg{n}$, the actions on
$\KTfunc(T)$, and the differential on $\KTfunc(T)$ when $T$ is
non-flat as induced by cobordisms embedded in
$[0,1]\times(0,1)^2$. For instance, the multiplication
$V(a\Wmirror{b})\otimes V(b\Wmirror{c})\to V(a\Wmirror{c})$ is induced
by a cobordism in $[0,1]\times(0,1)^2$ from
$\{0\}\times (a\Wmirror{b}\amalg b\Wmirror{c})$ to
$\{1\}\times(a\Wmirror{c})$. For this section, only the abstract (not
embedded) cobordisms are relevant; but for the stable homotopy refinement we
will need the embedded cobordisms.

\subsubsection{Gradings}
Khovanov homology has both a quantum (internal) and homological
grading. 

We start with the quantum grading. We grade $V$ so that $\intgr(1)=-1$
and $\intgr(X)=1$. Then the grading of $\KTalg{n}$ is obtained by
shifting the grading on each $V(a\Wmirror{b})$ up by $n$. In
particular, the elements of lowest degree in $\KTalg{n}$ are the
idempotents in $V(a\Wmirror{a})$, in which each of the $n$ circles is
labeled by $1$, and these generators lie in quantum grading $0$. All
homogeneous, non-idempotent elements lie in positive quantum
grading. Similarly, for the invariants of flat tangles, if
$T\in\FlatTangles{m}{n}$ then the quantum grading on
$V(aT\Wmirror{b})$ is shifted up by $n$. Given a tangle diagram $T$
with $N$ crossings and a vertex $v\in\{0,1\}^N$, we shift the grading
of $\KTfunc(T_v)$, the part of $\KTfunc(T)$ lying over the vertex $v$,
down by an additional $|v|$. (Here, $|v|$ denotes the number of $1$'s
in $v$.) The grading on the whole cube is then shifted down by
$N_+-2N_-$, where $N_+$, respectively $N_-$, is the number of
positive, respectively negative, crossings in $T$; this is where the
orientation of $T$ is used. In other words, for $T$ a $(2m,2n)$-tangle
diagram, the quantum grading on $V(aT_v\Wmirror{b})\subset \KTfunc(T)$
is shifted up by $n-|v|-N_++2N_-$.

For the homological gradings, all of $\KTalg{n}$ lies in grading
$0$. The homological grading on $\KTfunc(T_v)\subset \KTfunc(T)$ is
given by $N_--|v|$.  The differential on $\KTfunc(T)$ preserves the
quantum grading and decreases the homological grading by $1$. The
isomorphism of Proposition~\ref{prop:Kh-pairing} and the chain
homotopy equivalences of Proposition~\ref{prop:Kh-invariance} respect
both gradings.

\begin{remark}\label{rem:grading-convention}
  Khovanov's first paper on $\mathfrak{sl}_2$ knot
  homology~\cite{Kho-kh-categorification} and his paper on its
  extension to tangles~\cite{Kho-kh-tangles} use different conventions
  for the quantum grading: in the first paper, $\intgr(X)=\intgr(1)-2$
  while in the second $\intgr(X)=\intgr(1)+2$.
  Our first papers on Khovanov homotopy
  type~\cite{RS-khovanov,LLS-khovanov-product} follow
  Khovanov's original convention
  from~\cite{Kho-kh-categorification}. In this
  paper we switch to Khovanov's newer quantum grading convention
  of~\cite{Kho-kh-tangles}.

  Khovanov's homological grading conventions are the same in all of
  his papers, but our homological gradings also differ from his by a
  sign. This is because we treat the Khovanov complex as a chain
  complex, not a cochain complex; see our conventions from
  Section~\ref{sec:grading-convention}.
\end{remark}

\subsection{The Khovanov-Burnside 2-functor}\label{sec:CobE-to-Burn}
\begin{definition}\label{def:burnsidecat}
  Informally, the \emph{Burnside category} $\BurnsideCat$ is the
  bicategory with objects finite sets $X$, $\Hom(X,Y)$ the class of
  finite correspondences $A\co X\to Y$, i.e., diagrams of sets
  \[
    \xymatrix{
      & A\ar[dl]_s\ar[dr]^t & \\
      X & & Y,
    }
  \]
  and $\twoHom(A,B)$ the set of isomorphisms of correspondences from $A$
  to $B$, i.e., commutative diagrams
  \[
    \xymatrix{
      & A\ar[dl]\ar[drr] \ar[r]^-\cong& B\ar[dll]\ar[dr] & \\
      X & & & Y.
    }
  \]
  Composition of correspondences is fiber product: given $A\co X\to Y$
  and $B\co Y\to Z$, $B\circ A=A\times_YB$.  Note that one can think
  of a correspondence $A\co X\to Y$ as a $(Y\times X)$-matrix of
  sets, i.e., for each $(y,x)\in Y\times X$ a set
  $A_{y,x}=s^{-1}(x)\cap t^{-1}(y)$. Composition of correspondences
  then corresponds to multiplication of matrices, using the Cartesian
  product and disjoint union to multiply and add sets.

  Note that, with this definition, composition is not strictly
  associative since $(A\times_Y B)\times_Z C$ is in canonical
  bijection with, but not equal to, $A\times_Y(B\times_Z
  C)$. Composition also
  lacks strict identities since $A\times_X X$ is in canonical
  bijection with, but not equal to, $A$. There are many ways to
  rectify this; here is one.

  Instead of correspondences, let $\Hom(X,Y)$ denote the set of pairs
  of an integer $n$ and a $(Y\times X)$-matrix
  $(A_{y,x})_{x\in X,\ y\in Y}$ of finite subsets $A_{y,x}$ of
  $\RR^n$, with the following property:
  \begin{itemize}
  \item[(D)] $A_{y,x}\cap A_{y',x}=\emptyset$ if $y\neq y'$ and
    $A_{y,x}\cap A_{y,x'}=\emptyset$ if $x\neq x'$.
  \end{itemize}
  (A $(Y\times X)$-matrix of subsets of $\RR^n$ is a function
  $Y\times X\to 2^{\RR^n}$.) Given subsets $A\subset\RR^n$ and
  $B\subset\RR^m$, $A\times B$ is a subset of $\RR^{n+m}$. Composition
  is defined by
  \[
    (A_{z,y})_{y\in Y,\ z\in Z}\circ 
    (A_{y,x})_{x\in X,\ y\in Y}=\left(\bigcup_{y\in Y}A_{z,y}\times A_{y,x}\right)_{x\in X,\ z\in Z}.
  \]
  The condition that $A_{y,x}\cap A_{y',x}=\emptyset$ whenever
  $y\neq y'$ implies that the sets in the union are disjoint. Given
  $x\neq x'$, $(A_{z,y}\times A_{y,x})\cap (A_{z,y'}\times A_{y',x'})$
  is empty unless $y=y'$ (by looking at the first factor), and thus
  is empty unless $x=x'$ (by looking at the second factor). Similarly,
  $(A_{z,y}\times A_{y,x})\cap (A_{z',y'}\times A_{y',x})=\emptyset$
  if $z\neq z'$. Thus, the composition has Property~(D). Composition
  is clearly strictly associative. The (strict) identity element of
  $X$ is the $(X\times X$)-diagonal matrix with diagonal entries the
  $1$-element subset of $\RR^0$. A $2$-morphism of correspondences
  $\phi\co (A_{y,x})_{x\in X,\ y\in Y}\to (B_{y,x})_{x\in X,\ y\in Y}$
  is a collection of bijections
  $(\phi_{y,x}\co A_{y,x}\stackrel{\cong}{\longrightarrow}
  B_{y,x})_{x\in X,\ y\in Y}$; note that $2$-morphisms ignore the
  embedding information.

  Throughout, when we talk about the Burnside category we mean this
  latter, strict version of the category. Typically, however, the
  embedding data can be chosen arbitrarily, and in those cases we will
  not specify it.
\end{definition}

The free abelian group construction gives a functor
$\BurnsideCat\to \AbelianGroups$, by
\begin{align*}
  \Ob(\BurnsideCat)\ni X&\mapsto \bigoplus_{x\in X}\ZZ\\
  (A_{y,x})_{x\in X,\ y\in Y}&\mapsto (|A_{y,x}|)_{x\in X,\ y\in Y}
\end{align*}
where $|A_{y,x}|$ denotes the number of elements of $A_{y,x}$; the
right-hand side is a $(Y\times X)$-matrix of non-negative integers,
specifying a homomorphism
$\ZZ\langle X\rangle\to\ZZ\langle
Y\rangle$.\setcounter{OutlineCompletion}{26}

\begin{definition}
  The \emph{embedded cobordism category} of $1$-manifolds in $(0,1)^2$,
  $\CobE=\CobE^{1+1}((0,1)^2)$, has:
  \begin{itemize}
  \item Objects equivalence classes of smooth, closed, one-dimensional
    submanifolds $Z\subset (0,1)^2$ (i.e., finite collections of
    disjoint, embedded circles in the open square). Here, we view $Z$
    and $Z'$ as equivalent if there is a diffeomorphism
    $\phi\in\Diff^1$ so that $(\phi\times\Id_{(0,1)})(Z)=Z'$.
  \item Morphisms $\Hom(Z,W)$ equivalence classes of proper
    cobordisms embedded in $[0,1]\times(0,1)^2$ from $\{0\}\times Z$
    to $\{1\}\times W$, which intersect $[0,\epsilon]\times(0,1)^2$
    and $[1-\epsilon,1]\times(0,1)^2$ as $[0,\epsilon]\times Z$ and
    $[1-\epsilon,1]\times W$, respectively, for some $\epsilon>0$
    (which may depend on the cobordism; compare
    Convention~\ref{convention:vertical}), and so that each component
    of the cobordism intersects $\{1\}\times(0,1)^2$. Here, we view
    two cobordisms $\Sigma$, $\Sigma'$ as equivalent if there is a
    diffeomorphism $\phi\in\Diff^2$ so that
    $(\phi\times\Id_{(0,1)})(\Sigma)=\Sigma'$.
  \item Two-morphisms the set of isotopy classes of isotopies of
    cobordisms.
  \end{itemize}
  Note the morphisms are well-defined, because if an embedded
  one-manifold $Z$, respectively $W$, is equivalent (related by
  $\Diff^1$) to $Z'$, respectively $W'$, and if $\Sigma$ is any
  embedded cobordism from $Z$ to $W$, then there is an embedded
  cobordism $\Sigma'$ from $Z'$ to $W'$ which is equivalent (related
  by $\Diff^2$) to $\Sigma$. Note that composition maps and identity
  maps are strict, because we quotiented by the action of
  diffeomorphisms of $[0,1]$ (the first factor in
  $[0,1]\times(0,1)^2$). There is also a disjoint union operation on
  objects and morphisms induced by $(0,1)\coprod(0,1)\to
  (0,1/2)\coprod(1/2,1)\hookrightarrow (0,1)$, where $(0,1)$ is the
  first factor in $(0,1)^2$. This operation is strictly associative
  because we quotiented by the action of diffeomorphisms on this
  factor. Finally note that we have explicitly disallowed closed
  surfaces in morphisms; see
  Remark~\ref{rem:no-closed-surface-in-cobe}.
\end{definition}

There is a forgetful map from the embedded cobordism
category $\CobE=\CobE^{1+1}((0,1)^2)$ to the abstract
$(1+1)$-dimensional cobordism category $\Cob^{1+1}$. So, any Frobenius
algebra induces a functor $\CobE\to\AbelianGroups$ by composing the
corresponding abstract $(1+1)$-dimensional TQFT with the forgetful
functor. (Here, we view the monoidal category $\AbelianGroups$ of
abelian groups as a monoidal bicategory with only identity
$2$-morphisms.) In particular, the Khovanov Frobenius algebra
$V=H^*(S^2)$ induces such a functor.\setcounter{OutlineCompletion}{28}

Hu-Kriz-Kriz~\cite{HKK-Kh-htpy} observed that the Khovanov functor
$V\co \Cob\to \AbelianGroups$ lifts to a lax $2$-functor
$V_{\HKKa}\co \CobE\to \BurnsideCat$:
\begin{equation}\label{eq:forget-HKK}
  \vcenter{\hbox{\xymatrix{
    \CobE\ar[d]\ar[r]^{V_{\HKKa}} & \BurnsideCat\ar[d]\\
    \Cob\ar[r]_{V} & \AbelianGroups.
  }}}
\end{equation}
In this section, we will describe this functor
$V_{\HKKa}\co \CobE\to \BurnsideCat$, following the treatment in our
earlier paper~\cite[Section 8.1]{LLS-khovanov-product}.

\begin{remark}\label{rem:hom-not-cohom}
  The functor $\CobE\to\BurnsideCat$
  from~\cite{HKK-Kh-htpy,LLS-khovanov-product} actually did not lift
  the Khovanov functor $V$, but rather its opposite.
  That ensured that the
  \emph{cohomology} of the space constructed
  in~\cite{RS-khovanov,HKK-Kh-htpy,LLS-khovanov-product} was
  isomorphic to the Khovanov homology.

  However, in this paper we wish to construct a stable homotopy refinement
  of Khovanov's arc algebras (among other things). If we stick to
  cohomology, we would either have to construct a co-ring spectrum
  whose cohomology is the Khovanov arc algebra, or define a Khovanov
  arc co-algebra first, and then construct a ring spectrum whose
  cohomology is the newly defined Khovanov arc co-algebra. Not fancying
  either route, in this paper instead construct stable homotopy
  refinements whose \emph{homologies} are Khovanov homology; that is, their
  cohomology is the Khovanov homology of the mirror knot
  (cf.~\cite[Proposition~32]{Kho-kh-categorification}).
  Therefore, below we define a functor
  ${V_{\HKKa}}\from\CobE\to\BurnsideCat$ that actually lifts the
  Khovanov functor $V\from\Cob\to\AbelianGroups$, and not its
  opposite; in particular, it is not the functor described
  in~\cite{HKK-Kh-htpy,LLS-khovanov-product}, but rather, its
  opposite.
\end{remark}

\begin{remark}\label{rem:no-closed-surface-in-cobe}
  In \cite{HKK-Kh-htpy,LLS-khovanov-product}, the functor to
  $\BurnsideCat$ was actually constructed from a larger category,
  where the additional restriction that each component of the
  cobordism intersects $\{1\}\times(0,1)^2$ was not imposed. However,
  in this paper we wish to make the embedded cobordism category
  strictly monoidal and strictly associative, and therefore we have
  quotiented out the objects and morphisms by $\Diff^1$ and $\Diff^2$,
  respectively. Unfortunately, $\Diff^2$ can interchange some closed
  components of a cobordism, and therefore, we work with the
  subcategory where each component of the cobordism must intersect
  $\{1\}\times(0,1)^2$, ruling out closed components.
\end{remark}

On objects, for $C\in\Ob(\CobE)$ a disjoint union of circles,
$V_{\HKKa}(C)$ is the set of labelings of the circles in $C$ by $1$ or
$X$, i.e., functions $\pi_0(C)\to \{1,X\}$. Note that $\Diff^1$ cannot
interchange the components of $C$, so $C$, despite being a
$\Diff^1$-equivalence class, still has a notion of components.

To define $V_{\HKKa}$ on morphisms, fix an embedded cobordism $\Sigma$
from $C_0$ to $C_1$. Fix also a checkerboard coloring (2-coloring) of
the complement of $\Sigma$; for definiteness, choose the coloring in
which the region at $\infty$ (the region whose closure in
$[0,1]\times(0,1)^2$ is non-compact) is colored white.

The value of $V_{\HKKa}(\Sigma)$ is the product over the components
$\Sigma'$ of $\Sigma$ of $V_{\HKKa}(\Sigma')$ (with respect to the
checkerboard coloring of the complement of $\Sigma'$ that is induced
from the checkerboard coloring of the complement of $\Sigma$ by
declaring that the two colorings agree in a neighborhood of
$\Sigma'$), and the source and target maps respect this
decomposition. (Once again, since $\Sigma$ has no closed components,
$\Diff^2$ cannot interchange components, and so the notion of
components descends to equivalence classes.)

So, to define $V_{\HKKa}(\Sigma)$ we may assume $\Sigma$ is connected,
but the checkerboard coloring is now arbitrary (that is, the region at
$\infty$ need not be colored white). Fix $x\in V_{\HKKa}(C_0)$ and
$y\in V_{\HKKa}(C_1)$. If $\Sigma$ has genus $>1$ then
$V_{\HKKa}(\Sigma)=\emptyset$. If $\Sigma$ has genus $0$ then we
declare that $s^{-1}(x)\cap t^{-1}(y)\subset V_{\HKKa}(\Sigma)$ has
$0$ or $1$ elements, and so $V_{\HKKa}(\Sigma)$ is determined by
Formula~\eqref{eq:forget-HKK}. If $\Sigma$ has genus $1$ then
$s^{-1}(x)\cap t^{-1}(y)\subset V_{\HKKa}(\Sigma)$ is empty unless $x$
labels each circle in $C_0$ by $1$ and $y$ labels each circle in $C_1$
by $X$.

In the remaining genus $1$ case, $V_{\HKKa}(\Sigma)$ has two elements,
which we describe as follows. Let $S^2$ denote the one-point
compactification of $(0,1)^2$. Let $B(([0,1]\times S^2)\setminus
\Sigma)$ denote the black region in the checkerboard coloring
(possibly extended to the new points at infinity).  Let
$B((\{0,1\}\times S^2)\setminus \Sigma)= (\{0,1\}\times S^2)\cap
B(([0,1]\times S^2)\setminus \Sigma)$.  Then $V_{\HKKa}(\Sigma)$ is
the set of generators of
\[
  H_1(B(([0,1]\times S^2)\setminus \Sigma))/H_1(B((\{0,1\}\times S^2)\setminus \Sigma))\cong\ZZ.
\]

To define $V_{\HKKa}$ on $2$-morphisms, note that the definitions above
are natural with respect to isotopies of the surface $\Sigma$.

The composition $2$-isomorphism is obvious except when composing two
genus $0$ components $\Sigma_0$, $\Sigma_1$ to obtain a genus $1$
component $\Sigma$. In this non-obvious case, it again suffices to
assume $\Sigma$ is connected. For any curve $C$ on $\Sigma$, let $C_b$
and $C_w$ be its push-offs into $B((\{1/2\}\times S^2)\setminus
\Sigma)$ (the black region) and $(\{1/2\}\times S^2\setminus
\Sigma)\setminus B((\{1/2\}\times S^2)\setminus \Sigma)$ (the white
region), respectively. Now consider the a unique component $C$ of
$(\bdy\Sigma_0)\cap(\bdy\Sigma_1)$ that is non-separating in $\Sigma$
and is labeled $1$, and orient it as the boundary of $B((\{1/2\}\times
S^2)\setminus \Sigma)$. One of the two push-offs $C_b$ and $C_w$ is a
generator of $H_1(([0,1]\times S^2)\setminus\Sigma)/H_1((\{0,1\}\times
S^2)\setminus\bdy\Sigma)\cong \ZZ^2$ and the other one is zero.  If
$C_b$ is the generator, label $\Sigma$ by $[C]$.  If $C_w$ is the
generator, let $D$ be a curve on $\Sigma$, oriented so that the
algebraic intersection number $D\cdot C=1$ (with $\Sigma$ being
oriented as the boundary of the black region); and label $\Sigma$ by
$[D_b]$.\setcounter{OutlineCompletion}{30}

\section{Combinatorial tangle invariants}

\subsection{A decoration with divides}\label{sec:CobD}
The embedded cobordism category $\CobE$ has $2$-morphisms which give
nontrivial endomorphisms of $V_{\HKKa}(\Sigma)$. For example, if
$\Sigma$ consists of the connected sum of a cylinder and a torus then
rotating the torus by $\pi$ around the connect sum point exchanges the
two elements of $V_{\HKKa}(\Sigma)$. To define the tangle invariants,
it is more convenient to be able to work with a multicategory where
each $2$-morphism space is empty or has a single element, so
$V_{\HKKa}$ takes each $2$-endomorphism to the identity map: this will
save use from having to check many compatibility conditions.

So, let $\CobD$ be the following $2$-category.
\begin{enumerate}
\item An object of $\CobD$ is an equivalence class of the following data:
  \begin{itemize}
  \item A smooth, closed $1$-manifold $Z$
    embedded in $(0,1)^2$.
  \item A compact $1$-dimensional submanifold-with-boundary $A\subset Z$ satisfying the following:  If
    $I$ denotes the closure of $Z\setminus A$, then each of $A$ and
    $I$ is a disjoint union of closed intervals.
    We call components of $A$ \emph{active arcs} and components of $I$
    \emph{inactive arcs}.
  \end{itemize}
  We call $(Z,A)$ a \emph{divided $1$-manifold}. Two divided
  $1$-manifolds $(Z,A)$ and $(Z',A)$ are equivalent if
  there is an orientation-preserving diffeomorphism
  $\phi\in\Diff^1$ so that
  $(\phi\times \Id_{(0,1)})(Z,A)=(Z',A')$.

  We may sometimes suppress $A$ from the notation.
  
  See Figure~\ref{fig:flat-active} for some examples of divided
  $1$-manifolds.  
\item A morphism from $(Z,A)$ to $(Z',A')$ is an equivalence class of
  pairs $(\Sigma,\Gamma)$ where
  \begin{itemize}
  \item $\Sigma$ is a smoothly embedded cobordism in $[0,1]\times
    (0,1)^2$ from $Z$ to $Z'$ (satisfying
    Convention~\ref{convention:vertical}).
  \item $\Gamma\subset \Sigma$ is a collection of properly embedded
    arcs in $\Sigma$ (also satisfying
    Convention~\ref{convention:vertical}), with $(\bdy A\cup \bdy A')=
    \bdy\Gamma$, and so that every component of $\Sigma\setminus
    \Gamma$ has one of the following forms:
    \begin{enumerate}[label=(\Roman*)]
    \item\label{item:type1} A rectangle, with two sides components of $\Gamma$
      and two sides components of $A\cup A'$. 
    \item\label{item:type2} A $(2n+2)$-gon, with $(n+1)$ sides components of $\Gamma$,
      one side a component of $I'$, and the other $n$ sides components of
      $I$. (The integer $n$ is allowed to be zero.) 
    \end{enumerate}
    We call the components of $\Gamma$ \emph{divides}.
  \end{itemize}
  The pairs $(\Sigma,\Gamma)$ and $(\Sigma',\Gamma')$ are equivalent
  if there is a diffeomorphism
  $\phi\in\Diff^2$ so that
  $(\phi\times\Id_{(0,1)})(\Sigma)=\Sigma'$ and
  $(\phi\times\Id_{(0,1)})(\Gamma)=\Gamma'$.

  We will call a morphism in $\CobD$ a \emph{divided
    cobordism}. Again, we will sometimes suppress $\Gamma$ from the
  notation.

  See Figure~\ref{fig:divided-cob} for some examples of divided cobordisms.
\item There is a unique $2$-morphism from $(\Sigma,\Gamma)$ to
  $(\Sigma',\Gamma')$ whenever (some representative of the equivalence
  class of) $(\Sigma,\Gamma)$ is isotopic to (some representative of the
  equivalence class of) $(\Sigma',\Gamma')$ rel boundary.
\item Composition of divided cobordisms is defined as follows. Given
  $(\Sigma,\Gamma)\co (Z,A)\to(Z',A')$ and $(\Sigma',\Gamma')\co
  (Z',A')\to(Z'',A'')$, choose a representative of the equivalence
  class of $(Z',A')$ and representatives of the equivalence classes
  $(\Sigma,\Gamma)$ and $(\Sigma',\Gamma')$ which end / start at this
  representative of $(Z',A')$.  Define
  $(\Sigma',\Gamma')\circ(\Sigma,\Gamma)$ to be
  $(\Sigma'\circ\Sigma,\Gamma'\circ\Gamma)$.
  
  It is not too hard to check that composition of divided cobordisms
  is indeed is a divided cobordism. To wit, Type~\ref{item:type2}
  regions compose to produce Type~\ref{item:type2} regions; in
  particular, since each divide has a Type~\ref{item:type2} region on
  one side, we do not get any closed components in the set of divides
  after composing. While composing Type~\ref{item:type1} rectangles,
  we glue them along their active boundaries to get new
  Type~\ref{item:type1} rectangles. We do not get any annuli by gluing
  together such rectangles since that would produce closed divides.

  It is also clear that the composition map extends uniquely to
  $2$-morphisms.
\end{enumerate}\setcounter{OutlineCompletion}{34}

Forgetting the divides does not immediately give a functor from the
$2$-category $\CobD$ to the $2$-category $\CobE$. While we do get
maps on the objects and the $1$-morphisms, there are no immediate
maps on the $2$-morphisms. However:

\begin{lemma}\label{lem:div-cob-loop}
  If $(\Sigma_t,\Gamma_t)$ is a loop of divided cobordisms (rel
  boundary) then the induced map $\Sigma_0\to
  \Sigma_1=\Sigma_0$ is isotopic to the identity map.
\end{lemma}
\begin{proof}
Since the loop is constant on the boundary, the induced map
$\Sigma_0\to \Sigma_0$ must take each connected component $C$ of
$\Sigma_0\setminus \Gamma$ to itself. The map fixes $\bdy\Sigma_0$
pointwise and the divides $\Gamma$ setwise; but since there are no
closed divides, it is isotopic to a map that fixes $\Gamma$
pointwise. However, since $C$ is a planar region (for both
Types~\ref{item:type1} and~\ref{item:type2}), the mapping class group
of $C$ fixing the boundary is trivial.
\end{proof}

\begin{proposition}\label{prop:CobD-to-Burnside}
  The lax $2$-functor $V_{\HKKa}\co\CobE\to \BurnsideCat$ induces a
  lax $2$-functor $\CobD\to \BurnsideCat$.
\end{proposition}

More precisely, there is an analogue $\wh{\CobD}$ of $\CobD$ in which
the set of $2$-morphisms from $\Sigma_0$ to $\Sigma_1$ is the set of
isotopy classes of isotopies of divided cobordisms from $\Sigma_0$ to
$\Sigma_1$. There are forgetful maps
$\Pi_{\CobD}\co \wh{\CobD}\to \CobD$ (collapsing the $2$-morphism
sets) and $\Pi_{\CobE}\co \wh{\CobD}\to \CobE$ (forgetting the
divides). Proposition~\ref{prop:CobD-to-Burnside} asserts that the map
$V_{\HKKa}\circ\Pi_{\CobE}$ descends to a functor
$\CobD\to \BurnsideCat$, so that the following diagram
commutes:
\[
  \xymatrix{
    \wh{\CobD}\ar[r]^{\Pi_{\CobE}}\ar[d]_{\Pi_{\CobE}} & \CobE\ar[d]^{V_{\HKKa}}\\
    \CobD\ar@{-->}[r] & \BurnsideCat.
  }
\]

\begin{proof}[Proof of Proposition~\ref{prop:CobD-to-Burnside}]
  We must check that if $\phi$ is an isotopy from $(\Sigma,\Gamma)$ to
  itself then the induced map $V_{\HKKa}(\Sigma)\to V_{\HKKa}(\Sigma)$ is the
  identity map. The only interesting case, of course, is a genus $1$
  component of $\Sigma$. By Lemma~\ref{lem:div-cob-loop}, a loop induces
  the identity map on $H_1(\Sigma)$. The Mayer-Vietoris theorem implies
  that the map $H_1(\Sigma)\to H_1(\overline{B(([0,1]\times S^2)\setminus
  \Sigma)})\cong H_1(B(([0,1]\times S^2)\setminus \Sigma))$ is surjective, so the map
  on $H_1(B(([0,1]\times S^2)\setminus \Sigma))$ induced by $\phi$ is also the identity map.
\end{proof}
By a slight abuse of notation, we will let $V_{\HKKa}$ denote the
induced functor $\CobD\to\BurnsideCat$ as well.\setcounter{OutlineCompletion}{40}

\begin{remark}
  It is interesting to compare $\CobD$ with Zarev's \emph{divided surfaces}~\cite[Definition 3.1]{Zarev-hf-BS}.
\end{remark}
\subsection{A meeting of multicategories}

\subsubsection{The Burnside multicategory}\label{sec:mBurnside}

We may treat the Burnside category $\BurnsideCat$ as a monoidal category
with Cartesian product as the monoidal operation on objects. However,
this operation is not strictly associative. We can make the monoidal structure strict by
embedding the objects of $\BurnsideCat$ in standard Euclidean spaces,
similarly to what we did for morphisms in
Definition~\ref{def:burnsidecat}, and then define a multicategory
$\mBurnside$ induced from the monoidal structure.

More directly, define $\mBurnside$ as the multicategory enriched in
groupoids with:
\begin{itemize}
\item Objects pairs $(k,X)$ where $k\in\NN$ and $X$ is a finite subset of $\RR^k$. We will always suppress $k$ from the notation.
\item $\Hom_{\mBurnside}(X_1,\dots,X_n;Y)=\Hom_{\BurnsideCat}(X_1\times\dots\times X_n,Y)$, 
  the groupoid of maps in the Burnside category from
  $X_1\times\dots\times X_n$ to $Y$. (Note that since each $X_i$ is a
  subset of $\RR^{k_i}$,
  $(X_i\times X_{i+1})\times X_{i+2}=X_{i}\times(X_{i+1}\times
  X_{i+2})$ identically.)
\end{itemize}
Multi-composition is defined in the obvious way.  The special case
$n=0$ of the multimorphism sets seems worth spelling out. Let
$1=(0,\{0\})$ be the object in $\mBurnside$ consisting of a single
point embedded in $\RR^0$. Note that for any object $X$ in
$\mBurnside$, $1\times X=X$. We declare that the empty product in the
Burnside category is the object $1$. So, for any object
$X\in\Ob(\mBurnside)$, $\Hom_{\mBurnside}(\emptyset;X)=\Hom_{\BurnsideCat}(1,X).$ In particular,
an element of the set $X$ gives a multimorphism $\emptyset\to X$.\setcounter{OutlineCompletion}{50}

Recall that we have a multicategory of abelian groups
$\mAbelianGroups$ by defining $\Hom(V_1,V_2,\dots,V_n;V)$ to be the
set of multilinear maps $V_1,\dots,V_n\to V$ (or equivalently, the set
of maps $V_1\otimes\cdots\otimes V_n\to V$).  We can view
$\mAbelianGroups$ as trivially enriched in groupoids. The forgetful functor
$\BurnsideCat\to\AbelianGroups$ from Section~\ref{sec:CobE-to-Burn}
respects the monoidal structure on both $\BurnsideCat$ and
$\AbelianGroups$, and therefore induces a forgetful functor
$\Forget\from\mBurnside\to\mAbelianGroups$.

\subsubsection{Shape multicategories}\label{sec:shape}
Recall from Section~\ref{sec:review} that $\Crossingless{n}$ denotes the
set of crossingless matchings on $2n$ points. Define $\mHshape{n}^0$
to be the shape multicategory associated to $\Crossingless{n}$
(Definition~\ref{def:shape-multicat-set}).  Specifically, the
multicategory $\mHshape{n}^0$ has one object for each pair $(a,b)$ of
crossingless matchings of $2n$ points, and a unique multimorphism
\[
(a_1,a_2),(a_2,a_3),\dots,(a_{k-1},a_k)\to (a_1,a_k).
\]
We will sometimes denote the unique morphism in
$\Hom((a_1,a_2),(a_2,a_3),\dots,(a_{k-1},a_k);(a_1,a_k))$ by
$f_{a_1,\dots,a_k}$. In particular, the special case $k=1$ of the
zero-input multimorphism $\emptyset\to (a_1,a_1)$ is denoted
$f_{a_1}$.

Similarly, define $\mTshapeStct{m}{n}^0$ to be the shape multicategory
associated to the sequence of sets
$(\Crossingless{m},\Crossingless{n})$
(Definition~\ref{def:shape-multicat-set-seq}). Specifically, the
multicategory $\mTshapeStct{m}{n}^0$ has three kinds of objects:
\begin{enumerate}
\item\label{item:ob-1} Objects $(a,b)$ where $a,b$ are crossingless matchings on $2m$ points,
\item\label{item:ob-2} Objects $(a,b)$ where $a,b$ are crossingless matchings on $2n$
  points, and
\item\label{item:ob-3} Objects $(a,b)$ where $a$ is a crossingless matching of $2m$
  points and $b$ is a crossingless matching of $2n$ points. For
  clarity we will write such objects instead as $(a,T,b)$ where $T$ is
  just a notational placeholder.
\end{enumerate}
There is a unique multimorphism
\[
(a_1,a_2),(a_2,a_3),\dots,(a_{k-1},a_k)\to(a_1,a_k)
\]
if $a_1,\dots,a_k$ are crossingless matchings on $2m$ points.  There
is a unique multimorphism
\[
(b_1,b_2),(b_2,b_3),\dots,(b_{\ell-1},b_\ell)\to(b_1,b_\ell)
\]
if $b_1,\dots,b_\ell$ are crossingless matchings on $2n$ points. There is
a unique multimorphism
\[
(a_1,a_2),\dots,(a_{k-1},a_k),(a_k,T,b_1),(b_1,b_2),\dots,(b_{\ell-1},b_\ell)\to (a_1,T,b_\ell)
\]
if $a_1,\dots,a_k$ are crossingless matchings on $2m$ points and
$b_1,\dots,b_\ell$ are crossingless matchings on $2n$ points. (The
special cases $k=1$ and $\ell=1$ are allowed.)

Note that $\mHshape{m}^0$ and $\mHshape{n}^0$ are full
sub-multicategories of $\mTshapeStct{m}{n}^0$. Extending the notation $f_{a_1,\dots,a_k}$ from $\mHshape{m}^0$, we will sometimes denote the unique morphism in 
\[
  \Hom_{\mTshapeStct{m}{n}^0}\bigl((a_1,a_2),\dots,(a_{k-1},a_k),(a_k,T,b_1),(b_1,b_2),\dots,(b_{\ell-1},b_\ell);(a_1,T,b_\ell)\bigr)
\]
by $f_{a_1,\dots,a_k,T,b_1,\dots,b_\ell}$.

Let $\mHshape{n}$ (respectively $\mTshape{m}{n}$) be the canonical
groupoid enrichment of $\mHshape{n}^0$ (respectively
$\mTshapeStct{m}{n}^0$) from Section~\ref{sec:enrich}. See in particular
Example~\ref{exam:enrichment} for some of the multimorphisms that
appear in the groupoid enriched categories.\setcounter{OutlineCompletion}{55}

Now recall, from Section~\ref{sec:review}, Khovanov's arc algebra
$\KTalg{n}$, and Khovanov's tangle invariant $\KTfunc(T)$, which is
a \dg $(\KTalg{m},\KTalg{n})$-bimodule. The
algebra $\KTalg{n}$ is equipped with an orthogonal set of idempotents
(Definition~\ref{def:idempotented}), one for each crossingless
matching $a\in\Crossingless{n}$, with the idempotent corresponding to
$a$ being the element of $V(a\Wmirror{a})$ that labels each of the $n$
circles by $1\in V$. Therefore, via the equivalences from
Section~\ref{sec:multi-vec}, we have the following:

\begin{principle}\label{prin:multi-is-bim}
  The Khovanov arc algebra $\KTalg{n}$ may be viewed as a multifunctor
  $\mHshape{n}^0\to\mAbelianGroups$. Composing with the inclusion
  $\mAbelianGroups\to\mComplexes$ (which views an abelian group as a
  chain complex concentrated in grading $0$), we can also view the
  Khovanov arc algebra as a multifunctor from $\mHshape{n}^0$ to chain
  complexes. Similarly, Khovanov's tangle invariant $\KTfunc(T)$ may
  be viewed as a multifunctor $\mTshapeStct{m}{n}^0\to\mComplexes$ which
  restricts to $\mHshape{m}^0$ and $\mHshape{n}^0$ as the arc algebra
  multifunctors.
\end{principle}
\setcounter{OutlineCompletion}{60}

\subsubsection{The divided cobordism multicategory}\label{sec:mCobD}
Next we turn to the multicategory $\mCobD$ of divided cobordisms.  The
divided cobordism category $\CobD$ from Section~\ref{sec:CobD} can be
endowed with a disjoint union bifunctor $\amalg$ induced by
concatenation in the first $(0,1)$-factor. Disjoint union is a
strictly associative (non-symmetric) monoidal structure on $\CobD$,
since we have quotiented out objects by $\Diff^1$ and morphisms by
$\Diff^2$. Therefore, we get an associated multicategory
$\underline{\CobD}$. The groupoid enriched multicategory $\mCobD$ is
the canonical groupoid enrichment of $\underline{\CobD}$.

Fleshing out the definition, the objects of $\mCobD$ are the same as
the objects of $\CobD$, i.e., $\Diff^1$-equivalence classes of
smooth, closed, embedded $1$-manifolds in $(0,1)^2$ which are
decomposed as unions of \emph{active arcs} and \emph{inactive arcs}.

A \emph{basic multimorphism from $(Z_1,\dots,Z_n)$ to $Z$} is an
element of $\Hom_{\CobD}(Z_1\amalg \dots\amalg Z_n,Z)$. Now, an object
of $\Hom_{\mCobD}(Z_1,\dots,Z_n;Z)$ consists of:
\begin{itemize}
\item a tree $\aTree$;
\item a labeling of each edge of $\aTree$ by an object of $\CobD$, so that
  the input edges are labeled $Z_1,\dots,Z_n$ and the output edge is
  labeled $Z$; and
\item a labeling of each internal vertex $v$ of $\aTree$ with input edges
  labeled $Z'_1,\dots,Z'_k$ and output edge labeled $Z'$ by a basic
  multimorphism from $(Z'_1,\dots,Z'_k)$ to $Z'$ (i.e., an object in
  $\Hom_{\CobD}(Z'_1\amalg\dots\amalg Z'_k,Z')$).
\end{itemize}
Composition of multimorphisms is induced by composition of trees;
being a canonical thickening, this is automatically strictly
associative and has strict units (the $0$ internal vertex trees).

Given a multimorphism $f$ in $\Hom_{\mCobD}(Z_1,\dots,Z_n;Z)$, the
\emph{collapsing} $f^\circ$ of $f$ is the result of composing the
cobordisms associated to the vertices of the tree according to the
edges of the tree, in some order compatible with the
tree. Associativity of composition in $\mCobD$ implies that the
collapsing $f^\circ$ of $f$ is well-defined, i.e., independent of the
order that one composes vertices in the tree.
Given multimorphisms $f,g\in\Hom_{\mCobD}(Z_1,\dots,Z_n;Z)$ there is a
unique morphism from $f$ to $g$ if and only if $f^\circ$ is isotopic
to $g^\circ$.  It is clear that if $f\circ(g_1,\dots,g_n)$ is defined
and there is a morphism from $f$ to $f'$ and from $g_i$ to $g'_i$ for
$i=1,\dots,n$ then there is a morphism from $f\circ(g_1,\dots,g_n)$ to
$f'\circ(g'_1,\dots,g'_n)$.

Putting these observations together, we have proved:
\begin{lemma}
  These definitions make $\mCobD$ into a multicategory.
\end{lemma}
\setcounter{OutlineCompletion}{70}

\subsubsection{Cubes}\label{sec:multi-cube}
To a non-flat tangle we will associate a cube of flat tangles, and
hence, roughly, a cube of multifunctors between groupoid-enriched
multicategories. In this section we make sense of this notion in
enough generality for our applications.

\begin{definition}
  Let $\CCat{N}_0$, the \emph{cube category}, be the category with
  objects $\{0,1\}^N$ and a unique morphism
  $v=(v_1,\dots,v_N)\to w=(w_1,\dots,w_N)$ whenever
  $v_i\leq w_i$ for all $1\leq i\leq N$. 
\end{definition}

\begin{remark}
  In our previous papers, we defined cube categories to be the
  opposite category of the above. However, since in this paper we are
  taking homology instead of cohomology (cf.
  Remark~\ref{rem:hom-not-cohom}) we need the morphisms in the cube to
  go from $0$ to $1$.
\end{remark}

We will define a groupoid-enriched multicategory
$\CCat{N}\ttimes\mTshape{m}{n}$, a kind of product of the cube $\CCat{N}$
and $\mTshape{m}{n}$. We first define its strictification
$\strictify{\CCat{N}\ttimes\mTshape{m}{n}}$
(Definition~\ref{def:strictify}).
\begin{itemize}
\item Objects of $\strictify{\CCat{N}\ttimes\mTshape{m}{n}}$ are pairs
  $(a,b)\in \Ob(\mHshape{m})\cup\Ob(\mHshape{n})$ or quadruples
  $(v,a,T,b)$ where $v\in \{0,1\}^N$ and
  $(a,T,b)\in \Ob(\mTshape{m}{n})$. 
\item For any objects $a_i\in\Ob(\mHshape{m})$,
  $b_j\in\Ob(\mHshape{n})$, and morphism $v\to w$ in $\CCat{n}$, there
  are unique multimorphisms
  \begin{align*}
    (a_1,a_2),\dots,(a_{k-1},a_k)&\to(a_1,a_k)\\
    (a_1,a_2),\dots,(a_{k-1},a_k),(v,a_k,T,b_1),(b_1,b_2),\dots,(b_{\ell-1},b_\ell)&\to (w,a_1,T,b_\ell)\\
    (b_1,b_2),\dots,(b_{\ell-1},b_\ell)&\to (b_1,b_\ell)
  \end{align*}
  in $\strictify{\CCat{N}\ttimes\mTshape{m}{n}}$, and no other
  multimorphisms.
\end{itemize}\setcounter{OutlineCompletion}{75}

Next define the \emph{thick $N$-cube category of $\mTshapeStct{m}{n}^0$},
$\CCat{N}\ttimes\mTshape{m}{n}$, as the following multicategory enriched in
groupoids:
\begin{itemize}
\item The objects are he same as $\Ob(\strictify{\CCat{N}\ttimes\mTshape{m}{n}})$.
\item A \emph{basic multimorphism} is one of:
  \begin{itemize}
  \item A multimorphism in $\mHshape{m}$ or $\mHshape{n}$, or
  \item A multimorphism of the form
    \[
      (a_1,a_2),\dots,(a_{k-1},a_k),(v,a_k,T,b_1),(b_1,b_2),\dots,(b_{\ell-1},b_\ell)\to (v,a_1,T,b_\ell)
    \]
    in $\strictify{\CCat{N}\ttimes\mTshape{m}{n}}$, or
  \item A morphism of the form $(v,a,T,b)\to (w,a,T,b)$ in
    $\strictify{\CCat{N}\ttimes\mTshape{m}{n}}$.
  \end{itemize}
\item An object of a multimorphism groupoid in
  $\CCat{N}\ttimes\mTshape{m}{n}$ is a 
  tree with $p$ inputs, together with a labeling of:
  \begin{itemize}
  \item each edge by an object of $\CCat{N}\ttimes\mTshape{m}{n}$ and
  \item each vertex by a basic multimorphism from the inputs of the
    vertex to the output of the vertex.
  \end{itemize}
\item Given a multimorphism in $\CCat{N}\ttimes\mTshape{m}{n}$, there
  is a corresponding multimorphism in
  $\strictify{\CCat{N}\ttimes\mTshape{m}{n}}$ by composing the basic
  multimorphisms according to the tree. Define the multimorphism
  groupoid to have a unique morphism $\aTree\to \aTree'$ if the
  corresponding multimorphisms in
  $\strictify{\CCat{N}\ttimes\mTshape{m}{n}}$ are the
  same. Equivalently, there is a unique morphism $\aTree\to\aTree'$ if
  and only if $\aTree$ and $\aTree'$ have the same source and target.
\end{itemize}

The above definition ensures that
$\strictify{\CCat{N}\ttimes\mTshape{m}{n}}$ is indeed the
strictification of $\CCat{N}\ttimes\mTshape{m}{n}$.
\begin{lemma}\label{lem:strictify-mTshape}
  The projection
  $\CCat{N}\ttimes\mTshape{m}{n}\to\strictify{\CCat{N}\ttimes\mTshape{m}{n}}$,
  which is the identity on objects and sends a tree with inputs
  $x_1,\dots,x_n$ and output $y$ to the unique multimorphism
  $x_1,\dots,x_n\to y$, is a weak equivalence.
\end{lemma}
\begin{proof}
  The proof is essentially the same as the proof of Lemma~\ref{lem:unenrich}.
\end{proof}
\setcounter{OutlineCompletion}{80}

The category $\CCat{N+1}\ttimes\mTshape{m}{n}$ has the category
$\CCat{N}\ttimes\mTshape{m}{n}$ as a full subcategory in two distinguished ways:
the full subcategory spanned by objects $(a,b)$ and
$(\{0\}\times v,a,T,b)$, which we denote
$\{0\}\times\CCat{N}\ttimes\mTshape{m}{n}$; and the full subcategory spanned by
objects $(a,b)$ and $(\{1\}\times v,a,T,b)$, which we denote
$\{1\}\times\CCat{N}\ttimes\mTshape{m}{n}$. The strictified product
$\strictify{\CCat{N+1}\ttimes\mTshape{m}{n}}$ has corresponding subcategories
$\strictify{\{0\}\times \CCat{N}\ttimes\mTshape{m}{n}}$ and
$\strictify{\{1\}\times\CCat{N}\ttimes\mTshape{m}{n}}$, both isomorphic to
$\strictify{\CCat{N}\ttimes\mTshape{m}{n}}$.

\begin{remark}
  The groupoid-enriched
  multicategory $\CCat{N}\ttimes\mTshape{m}{n}$ is related to a
  groupoid-enriched version of the Boardman-Vogt tensor
  product~\cite[Section II.3, Paragraph (2.15)]{BV-other-book}, the
  main difference being that we have not multiplied the objects
  of the form $(a,b)$ in $\mTshape{m}{n}$ by $\CCat{N}$.
\end{remark}

\subsection{A cabinet of cobordisms}\label{sec:cabinet}
In this section we enhance some of the topological objects used to
define the Khovanov arc algebras and modules so that they lie in the
category of divided cobordisms.

\begin{definition}\label{def:pox}
  Given a tangle diagram $T$, $\pi(T)\subset\RR^2$ is a planar,
  4-valent graph. The edges of $\pi(T)$ are the \emph{segments} of
  $T$.
  
  A \emph{poxed tangle} is a tangle diagram $T$ together with a
  collection of points (\emph{pox}) on the segments of
  $\pi(T)\subset \RR^2$ so that for each resolution $T_v$ of $T$,
  there is at least one pox on each closed component of $T_v$.

  A \emph{poxed link} is a poxed $(0,0)$-tangle.
\end{definition}

\begin{construction}\label{const:divided-crossingless}
  Given crossingless matchings $a,b\in\Crossingless{n}$, we make
  $a\Wmirror{b}$ into a divided $1$-manifold as follows. The inactive
  arcs are the connected components of a small neighborhood of
  $\bdy a\subset a\Wmirror{b}$ (so there are $2n$ inactive arcs), while
  the active arcs are the connected components of the complement of the
  inactive arcs (so there are also $2n$ active arcs). See
  Figure~\ref{fig:flat-active}.
\end{construction}

\begin{figure}
  \centering
  \includegraphics[scale=.75]{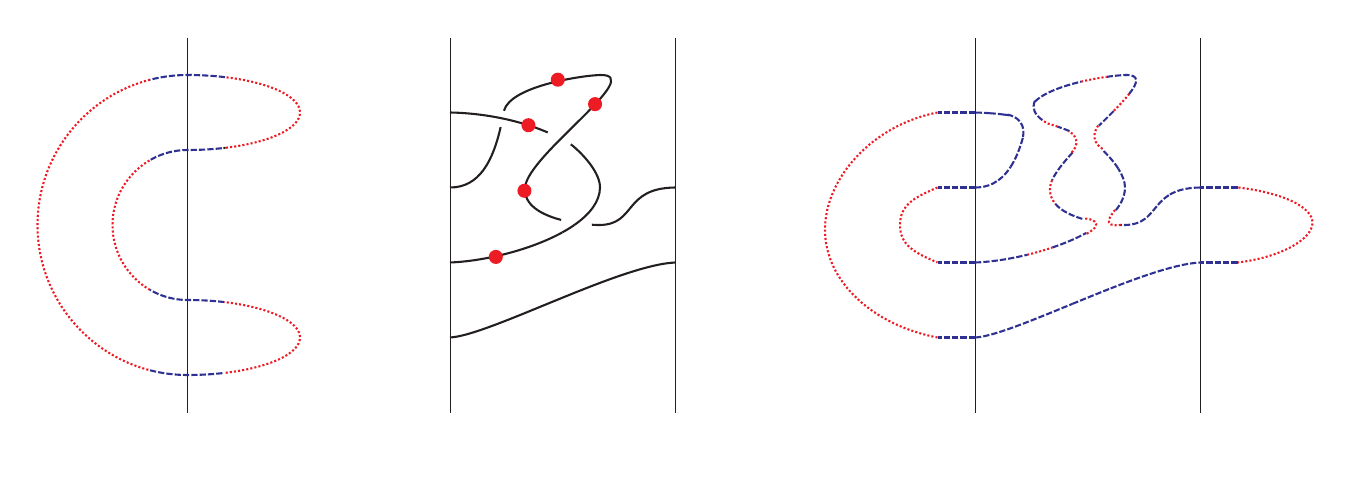}
  \caption{\textbf{Active arcs and flat tangles.} Left: $a\Wmirror{b}$ for
    two crossingless matchings $a,b\in\Crossingless{2}$ and the
    decomposition of $a\Wmirror{b}$ into active (\textcolor{red}{dotted}) and
    inactive (\textcolor{blue}{dashed}) arcs. Center: a poxed $(4,2)$-tangle
    $T$. Right: $aT_v\Wmirror{b}$ where $v=(1,0,0)$ and $a$ is the same
    crossingless matching as on the left.}
  \label{fig:flat-active}
\end{figure}

Given an oriented, poxed link
$K\in\TangleDiags{0}{0}$ with $N$ ordered crossings and a vector
$v\in\{0,1\}^N$, we make the resolution $K_v$ into a divided
$1$-manifold as follows. Let $\pi(K)$ denote the projection of $K$ to
$(0,1)^2$. For each $1\leq i\leq N$, choose a small disk $D_i$ around
the $i\th$ crossing of $\pi(K)$, so that $\bdy D_i$ intersects
$\pi(K)$ transversely in $4$ points, and a small disk $D'_p$ around
each pox $p$ of $K$. Choose the disks $D_i$ and
$D'_p$ small enough that they are all pairwise disjoint. Choose the
resolution $K_v$ so that $\pi(K)\cap \bigl((0,1)^2\setminus(\bigcup_i
D_i)\bigr)=K_v\cap \bigl((0,1)^2\setminus(\bigcup_i D_i)\bigr)$, i.e.,
so that $\pi(K)$ and $K_v$ agree outside the disks $D_i$. The
boundaries of the disks $D_i$ and $D'_p$ divide $K_v$ into arcs. 
Declare the arcs outside the disks $D_i$
and $D'_p$ to be inactive. Define the arcs inside $D_i$ to be active
if $v_i=0$ and inactive if $v_i=1$. Define the arcs inside the $D'_p$
to be active.

Combining the previous two cases:
\begin{construction}\label{const:divided-tangle}
  Given a poxed $(2m,2n)$-tangle $T\in\TangleDiags{m}{n}$ with $N$
  ordered crossings, $a\in\Crossingless{m}$, $b\in\Crossingless{n}$,
  and $v\in\{0,1\}^{N}$ we make $aT_v\Wmirror{b}$ into a divided
  $1$-manifold as follows. Again, choose small disks $D_i$ around the
  crossings of $\pi(T)$, so that outside the disks $D_i$, $T_v$ agrees
  with $\pi(T)$ and small disks $D'_p$ around the pox of $T$. Choose small neighborhoods of the endpoints of $a$
  and $b$.
  Here, small means that all these neighborhoods are disjoint. Then
  the active arcs of $aT_v\Wmirror{b}$ are:
  \begin{itemize}
  \item The arcs inside the $D_i$ with $v_i=0$,
  \item The arcs inside the $D'_p$, and
  \item The arcs in $a$ and $\Wmirror{b}$ in the complement of the
    neighborhoods of the endpoints.
  \end{itemize}
  The remaining arcs of $aT_v\Wmirror{b}$ are inactive. See
  Figure~\ref{fig:flat-active}.
\end{construction}
Next we turn to the divided cobordisms we will use as building
blocks. 

A \emph{trivial cobordism} is a cobordism of the form $[0,1]\times Z$
where $Z$ is a divided $1$-manifold. If $P$ is the set of endpoints of
the active arcs in $Z$ then the divides are given by
$\Gamma=[0,1]\times P$.

Next, fix a divided $1$-manifold $Z$ and a disk $D$ so that $D\cap Z$
consists of exactly two active arcs in $Z$. Call these four endpoints
$a,b,c,d$, so that the arcs join $a\leftrightarrow b$ and
$c\leftrightarrow d$, and $a$ and $d$ are consecutive around $\bdy D$.
Let $Z'$ be a divided $1$-manifold which agrees with $Z$ outside $D$
and consists of two arcs in $Z'\cap D$ connecting $a\leftrightarrow d$
and $b\leftrightarrow c$. Make $Z'$ into a divided $1$-manifold by
declaring that the arcs inside $D$ are inactive, and the other arcs of
$Z'$ are the same as the arcs of $Z$.  A \emph{saddle cobordism} is a
cobordism $\Sigma$ from $Z$ to $Z'$ so that:
\begin{itemize}
\item $\Sigma\cap \bigl([0,1]\times((0,1)^2\setminus D)\bigr)=[0,1]\times (Z\setminus D)$,
\item Inside $[0,1]\times D$, $\Sigma$ consists of a single embedded saddle, and
\item The dividing arcs $\Gamma$ for $\Sigma$ connect
  $a\leftrightarrow d$ and $b\leftrightarrow c$ inside the saddle, and
  agree with $[0,1]\times P$ outside the saddle, where $P$ is the
  collection of endpoints of active arcs of $Z'$.
\end{itemize}
(See Figure~\ref{fig:divided-cob} for the local form of $\Sigma$
in a neighborhood of $D$.) The cobordism $\Sigma$ is well-defined up to unique isomorphism
in $\CobD$. We call $D$ the \emph{support} of the saddle
cobordism. Note that a saddle cobordism $\Sigma\co Z_1\to Z_2$ is
determined by $Z_1$ and the support of $\Sigma$ (up to isotopy rel
$(\{0\}\times Z_1)\cup([0,1]\times (Z_1\setminus D))$).

\definecolor{darkgreen}{cmyk}{.9,.3,.9,.3}
\definecolor{darkblue}{cmyk}{84,53,0,0}
\begin{figure}
  \centering
  \includegraphics[width=\textwidth]{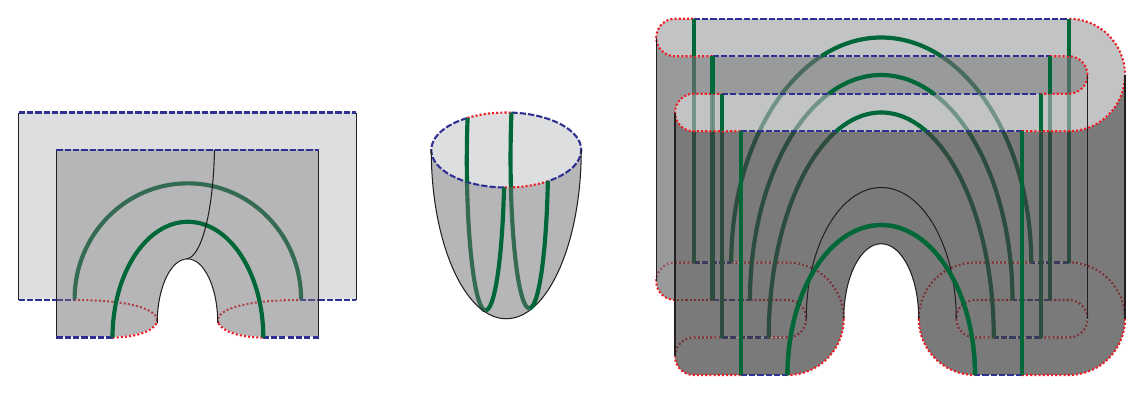}
  \caption{\textbf{Basic divided cobordisms.} Left: a saddle. Center:
    a cup. Curves $\Gamma$ are \textbf{\textcolor{darkgreen}{thick}}, arcs in $I\cup I'$ are \textcolor{darkblue}{dashed}, and arcs in $A\cup A'$ are \textcolor{red}{dotted}. Together with product cobordisms, these are the local
    pieces that the divided cobordisms of interest are built
    from. Right: the divided cobordism associated to a product on the
    Khovanov arc algebra.}
  \label{fig:divided-cob}
\end{figure}

More generally, given a divided $1$-manifold $Z$ and a collection of
disjoint disks $D_i$ so that each $D_i\cap Z$ consists of two active
arcs, a \emph{multi-saddle cobordism} is a divided cobordism $\Sigma$
from $Z$ so that $\Sigma\cap ([0,1]\times D_i)$ is a saddle for each
$i$ and
$\Sigma\setminus\bigl(\bigcup_i [0,1]\times D_i\bigr)
=\bigl([0,1]\times Z\bigr)\setminus\bigcup_i\bigl([0,1]\times D_i\bigr)$ is a product;
and where the dividing arcs on $\Sigma$
\begin{itemize}
\item connect the points in $P\cap \bdy D_i$ in pairs inside the
  saddles, as in Figure~\ref{fig:divided-cob} (i.e., so that points not
  connected in $Z\cap D_i$ are connected by arcs in $\Gamma$) and
\item are of the form $[0,1]\times\{p\}$ for
  $p\in (P\setminus\bdy D_i)$ the ends of active arcs not involved in the
  saddles.
\end{itemize}
We call $\bigcup_i D_i$ the \emph{support} of the multi-saddle
cobordism.

Next, given crossingless matchings $a,b,c\in\Crossingless{n}$, a
\emph{merge cobordism}
$a\Wmirror{b}\amalg b\Wmirror{c}\to a\Wmirror{c}$ is a composition of
saddle cobordisms, one for each arc in $b$.  Again, this merge
cobordism is well-defined up to unique isomorphism in
$\CobD$. Similarly, given $a,b\in\Crossingless{m}$, a flat
$(2m,2n)$-tangle $T\in\FlatTangles{m}{n}$, and
$c,d\in\Crossingless{n}$ there are \emph{merge cobordisms}
$a\Wmirror{b}\amalg bT\Wmirror{c}\to aT\Wmirror{c}$ and
$bT\Wmirror{c}\amalg c\Wmirror{d}\to bT\Wmirror{d}$. As usual, these merge
cobordisms are well-defined up to unique isomorphisms. The
\emph{support} of a merge cobordism is the union of the supports of
the sequence of saddle cobordisms. We will also call the union of a
merge cobordism with a trivial cobordism a merge cobordism. More
generally, a \emph{multi-merge cobordism} is a composition, in
$\CobD$, of merge cobordisms.

A \emph{birth cobordism} is a genus $0$ decorated cobordism from the
empty set to $a\Wmirror{a}$, for some $a\in\Crossingless{n}$.  Birth
cobordisms are unions of \emph{cups}; see
Figure~\ref{fig:divided-cob}. We call the union of the disks bounded
by $a\Wmirror{a}$ the \emph{support} of the birth cobordism. A
\emph{multi-birth cobordism} is the union of finitely many birth
cobordisms with disjoint supports and a trivial cobordism.

We note some commutation relations for cobordisms:
\begin{proposition}\label{prop:saddle-commute}
  Let $\Sigma_1\co Z_1\to Z_2$ and $\Sigma_2\co Z_2\to Z_3$ be saddle
  cobordisms supported on disjoint disks $D_1$ and $D_2$. Let
  $\Sigma'_2\co Z_1\to Z'_2$ and $\Sigma'_1\co Z'_2\to Z_3$ be saddle
  cobordisms supported on $D_2$ and $D_1$, respectively. Then
  $\Sigma_2\circ \Sigma_1$ is isotopic to $\Sigma'_1\circ \Sigma'_2$
  rel boundary.
\end{proposition}
\begin{proof}
  This is straightforward, and is left to the reader.
\end{proof}

We state a corollary somewhat informally; it can be formalized along
the lines of the statement of Proposition~\ref{prop:saddle-commute},
but the precise version seems more confusing than enlightening:
\begin{corollary}\label{cor:cob-commute}
  Suppose each of $\Sigma_1\co Z_1\to Z_2$ and $\Sigma_2\co Z_2\to Z_3$ is
  a multi-saddle or a multi-merge cobordism, and the supports of
  $\Sigma_1$ and $\Sigma_2$ are disjoint. Then $\Sigma_1$ and
  $\Sigma_2$ commute up to isotopy, in the obvious sense. 
\end{corollary}

Finally, we note some relations involving births:
\begin{proposition}\label{prop:births} Birth and merge cobordisms
  satisfy the following relations:
  \begin{enumerate}
  \item\label{item:unique-birth} Let $Z_2$ be a divided $1$-manifold
    and $Z_1\subset Z_2$ a subset which is itself a divided
    $1$-manifold. Then all multi-birth cobordisms from $Z_1$ to $Z_2$,
    in which the circles $Z_2\setminus Z_1$ are born, are
    isotopic.
  \item\label{item:birth-commute} If $\Sigma_1$ is a multi-birth,
    multi-merge, or multi-saddle cobordism and $\Sigma_2$ is
    a multi-birth cobordism, and the supports of $\Sigma_1$ and
    $\Sigma_2$ are disjoint then $\Sigma_1$ and $\Sigma_2$ commute up
    to isotopy.
  \item\label{item:birth-merge} If $\Sigma_1\co aT\Wmirror{b}\to a\Wmirror{a}\amalg aT\Wmirror{b}$
    (respectively
    $\Sigma_1\co aT\Wmirror{b}\to aT\Wmirror{b}\amalg b\Wmirror{b}$) is a
    birth cobordism and
    $\Sigma_2\co a\Wmirror{a}\amalg aT\Wmirror{b}\to aT\Wmirror{b}$
    (respectively
    $\Sigma_2\co aT\Wmirror{b}\amalg b\Wmirror{b} \to aT\Wmirror{b}$) is a
    merge cobordism then $\Sigma_2\circ \Sigma_1$ is isotopic to a
    trivial cobordism $aT\Wmirror{b}\to aT\Wmirror{b}$.
  \item\label{item:birth-multi-merge} If
    $\Sigma_1\co aT\Wmirror{b}\amalg bT'\Wmirror{c}\to
    aT\Wmirror{b}\amalg b\Wmirror{b}\amalg bT'\Wmirror{c}$ is a birth
    cobordism and
    $\Sigma_2\co aT\Wmirror{b}\amalg b\Wmirror{b}\amalg
    bT'\Wmirror{c}\to aTT'\Wmirror{c}$ is a multi-merge cobordism then
    $\Sigma_2\circ\Sigma_1$ is isotopic to a merge cobordism
    $aT\Wmirror{b}\amalg bT'\Wmirror{c}\to aTT'\Wmirror{c}$.
  \end{enumerate}
\end{proposition}
\begin{proof}
  Parts~(\ref{item:unique-birth}),~(\ref{item:birth-commute}) and~(\ref{item:birth-merge}) are
  straightforward from the
  definitions. Part~(\ref{item:birth-multi-merge}) follows from
  Parts~(\ref{item:unique-birth}) and~(\ref{item:birth-merge}).
\end{proof}

\subsection{A frenzy of functors}
Section~\ref{sec:CobE-to-Burn} recalls the Khovanov-Burnside functor,
which we can view as a multifunctor $\mV\co \mCobD\to
\mBurnside$:
\begin{lemma}
  There is a strict multifunctor $\mV\co \mCobD\to\mBurnside$ defined
  as follows:
  \begin{itemize}
  \item On objects, $\mV(Z)=V_{\HKKa}(Z)$, the set of labelings of $Z$
    by $\{1,X\}$.
  \item On basic multimorphisms, $\mV\bigl(\Sigma\co (Z_1,\dots,Z_n)\to Z\bigr)$
    is the correspondence 
    \[
    V_{\HKKa}(\Sigma)\co \mV(Z_1)\times\dots\times\mV(Z_n)\to \mV(Z).
    \]
    On general multimorphisms of $\mCobD$ (which are trees with
    vertices labeled by basic multimorphisms), $V_{\HKKa}$ is gotten
    by composing, in some order compatible with the tree, the
    correspondences $V_{\HKKa}(\Sigma_v)$ associated to the vertices
    $v$.

    Given $f\in\Hom_{\mCobD}(Z_1,\dots,Z_n;Z)$, we have two
    correspondences from
    $V_{\HKKa}(Z_1)\times\dots\times V_{\HKKa}(Z_n)$ to
    $V_{\HKKa}(Z)$: the correspondence $\mV(f)$, which is a composition of a sequence of
    correspondences associated to cobordisms, and the correspondence
    $V_{\HKKa}(f^\circ)$, which is the correspondence associated to
    the composition of those cobordisms. The coherence isomorphisms
    for the lax functor $V_{\HKKa}$ give an isomorphism
    $C(f)\co \mV(f)\to V_{\HKKa}(f^\circ)$. Now, given
    $f,g\in \Hom_{\mCobD}(Z_1,\dots,Z_n;Z)$ and
    $\phi\in\Hom(f,g)$, let $\phi^\circ$ be the corresponding morphism in $\CobD$
    from $f^\circ$ to $g^\circ$ and define
    \[
      \mV(\phi)= C(g)^{-1}\circ V_{\HKKa}(\phi^\circ)\circ C(f).
    \]
  \end{itemize}
\end{lemma}
\begin{proof}
  We must check that:
  \begin{enumerate}
  \item\label{item:mV-groupoid} Given $\phi\in\Hom(f,g)$ and $\psi\in\Hom(g,h)$,
    $\mV(\psi\circ\phi)=\mV(\psi)\circ\mV(\phi)$, so that $\mV$
    defines a map of groupoids.
  \item The functor $\mV$ respects the identity maps. This is trivial.
  \item\label{item:mV-functor} The functor $\mV$ respects composition of trees.
  \end{enumerate}
  For Point~(\ref{item:mV-groupoid}), we have
    \[
      \mV(\psi)\circ \mV(\phi)= C(h)^{-1}\circ V_{\HKKa}(\psi^\circ)\circ V_{\HKKa}(\phi^\circ)\circ C(f)
      =C(h)^{-1}\circ V_{\HKKa}(\psi^\circ\circ \phi^\circ)\circ C(f)=\mV(\psi\circ\phi),
    \]
    where the second equality uses functoriality of $V_{\HKKa}$
    (Proposition~\ref{prop:CobD-to-Burnside}). For
    Point~(\ref{item:mV-functor}), at the level of objects of the
    multimorphism groupoids this is immediate from associativity of
    composition in $\mBurnside$. For morphisms in the multimorphism
    groupoids this uses naturality of the coherence maps $C(f)$.
\end{proof}
\setcounter{OutlineCompletion}{85}

\begin{lemma}\label{lem:mHfunc}
  There is a multifunctor $\mHfunc{n}\co \mHshape{n}\to\mCobD$ 
  from the
  multicategory $\mHshape{n}$ to the cobordism
  multicategory $\mCobD$ defined as follows:
  \begin{itemize}
  \item On objects, $\mHfunc{n}((a,b))=a\Wmirror{b}$, which is a divided
    $1$-manifold as described in Section~\ref{sec:cabinet}.
  \item On basic multimorphisms, $\mHfunc{n}$ sends $f_{a_1,\dots,a_k}\co
    (a_1,a_2),\dots,(a_{k-1},a_k)\to (a_1,a_k)$ to some particular, chosen multi-merge cobordism
    \[
      \mHfunc{n}(f_{a_1,\dots,a_k})\co 
      a_1\Wmirror{a_2}\amalg \dots\amalg a_{k-1}\Wmirror{a_k}\to a_1\Wmirror{a_k}
    \]
    if $k>1$ and to the birth cobordism
    \[
    \mHfunc{n}(f_{a_1})\from \emptyset\to a_1\Wmirror{a_1}
    \]
    if $k=1$.  The functor $\mHfunc{n}$ assigns to an object
    in $\Hom_{\mHshape{n}}((a_1,a_2),\dots,(a_{k-1},a_k); (a_1,a_k))$
    with underlying tree $\aTree$  the composition (in
    $\mCobD$), according to $\aTree$, of the multi-merge or birth
    cobordisms chosen for each vertex.
  \end{itemize}
\end{lemma}
\begin{proof}
  We must check that $\mHfunc{n}$ extends to the morphisms in the
  multimorphism groupoids (i.e., $2$-morphisms), and that it respects
  multi-compositions. The fact that $\mHfunc{n}$ extends to
  $2$-morphisms follows from Corollary~\ref{cor:cob-commute} and
  Proposition~\ref{prop:births} (the second of which is only relevant
  when stumps are involved).
  The fact that $\mHfunc{n}$ respects composition is purely formal on
  the level of $1$-multimorphisms (from the definition of the
  canonical thickening). At the level of $2$-morphisms, it follows
  from the fact that given multimorphisms $\Sigma,\Sigma'$ in
  $\mCobD$, there is at most one $2$-morphism from $\Sigma$ to
  $\Sigma'$.
\end{proof}
\setcounter{OutlineCompletion}{87}

Given a flat, poxed $(2m,2n)$-tangle $T\in\FlatTangles{m}{n}$ there is a
multifunctor $\mTfunc{T}\co \mTshape{m}{n}\to \mCobD$ defined similarly
to $\mHfunc{n}$. Indeed, on the subcategories
$\mHshape{m},\mHshape{n}\subset \mTshape{m}{n}$ the functor $\mTfunc{T}$
is exactly $\mHfunc{m}$, $\mHfunc{n}$. On objects $(a,T,b)$, let
$\mTfunc{T}((a,T,b))=aT\Wmirror{b}$, which is a divided $1$-manifold as in Construction~\ref{const:divided-tangle}. On the basic multimorphisms
\[
f_{a_1,\dots,a_i,T,b_1,\dots,b_j}\co (a_1,a_2),\dots,(a_{i-1},a_i),(a_i,T,b_1),(b_1,b_2),\dots,(b_{j-1},b_j)\to (a_1,T,b_j)
\]
the functor $\mTfunc{T}(f_{a_1,\dots,a_i,T,b_1,\dots,b_j})$ is some
chosen multi-merge cobordism corresponding to the obvious merging. As
usual, this extends formally to general objects in the multimorphism
groupoids.

\begin{lemma}
  This construction extends uniquely to a multifunctor
  $\mTfunc{T}\co \mTshape{m}{n}\to \mCobD$.
\end{lemma}
\begin{proof}
  The proof is essentially the same as the proof of
  Lemma~\ref{lem:mHfunc}, and is left to the reader.
\end{proof}

Next, fix a poxed $(2m,2n)$-tangle $T\in\TangleDiags{m}{n}$ (see
Definition~\ref{def:pox}) with $N$
ordered crossings. We associate to $T$ a multifunctor
\[
\mTfuncNF{T}\co \CCat{N}\ttimes \mTshape{m}{n}\to \mCobD
\]
as follows. First, choose a collection of disjoint disks $D_i$ around
the crossings of $T$, and for each $v\in\{0,1\}^N$ choose a particular
flat tangle $T_v$ representing the $v$-resolution of $T$, so that
$T_v$ agrees with (the projection of) $T$ outside the disks
$D_i$.

Now, objects of $\CCat{N}\ttimes\mTshape{m}{n}$ are of three kinds:
\begin{itemize}
\item Pairs $(a_1,a_2)$ where 
  $a_1,a_2\in\Crossingless{m}$. In this case we define
  $\mTfuncNF{T}(a_1,a_2)=a_1\Wmirror{a_2}$, which we give the structure of a
  divided $1$-manifold as described in Construction~\ref{const:divided-crossingless}.
\item Pairs $(b_1,b_2)$ where
  $b_1,b_2\in\Crossingless{n}$. In this case we (again) define
  $\mTfuncNF{T}(b_1,b_2)=b_1\Wmirror{b_2}$.
\item Quadruples $(v,a,T,b)$ where $v\in\{0,1\}^N$,
  $a\in\Crossingless{m}$ and $b\in\Crossingless{n}$. In this case, we
  define $\mTfuncNF{T}(v,a,T,b)=aT_v\Wmirror{b}$. We give
  $aT_v\Wmirror{b}$ the structure of a divided $1$-manifold as
  described in Construction~\ref{const:divided-tangle}.
\end{itemize}
As always, defining $\mTfuncNF{T}$ on multimorphism groupoids takes more
work. To define $\mTfuncNF{T}$ on objects of the multimorphism groupoids
it suffices to define $\mTfuncNF{T}$ for the following two elementary
morphisms:
\begin{itemize}
\item A basic multimorphism coming from a morphism in $\CCat{N}$, i.e., a map $f_{\CCat{N}}\co (v,a,T,b)\to(w,a,T,b)$. Define
  $\mTfuncNF{T}(f_{\CCat{N}})$ to be a multi-saddle cobordism from 
  $aT_v\Wmirror{b}$ to $aT_w\Wmirror{b}$ (see Section~\ref{sec:cabinet}).
\item A basic multimorphism coming from a morphism
  $f_{\mTshape{m}{n}}$ in $\mTshapeStct{m}{n}^0$.  In this case define
  $\mTfuncNF{T}(f_{\mTshape{m}{n}})$ to be the cobordism
  $\mTfunc{T_v}(f_{\mTshape{m}{n}})$ (associated to the flat tangle $T_v$).
\end{itemize}
On a general object, $\mTfuncNF{T}$ is defined by composing these
multimorphisms according to the tree. (Since this composition happens
in $\mCobD$, given a multimorphism $f$ in
$\CCat{N}\ttimes\mTshape{m}{n}$ with underlying tree $\aTree$,
$\mTfuncNF{T}(f)$ is the same tree $\aTree$ with vertices labeled by
the divided cobordisms corresponding to the labels in $f$.)

Since there is a unique isomorphism between isotopic divided
cobordisms, to extend $\mTfuncNF{T}$ to morphisms in the multimorphism
groupoids it suffices to show that if two morphisms $\aTree$,
$\aTree'$ in $\CCat{N}\ttimes \mTshape{m}{n}$ have a morphism between
them the divided cobordisms $\mTfuncNF{T}(\aTree)^\circ$ and
$\mTfuncNF{T}(\aTree')^\circ$ are isotopic.

\begin{lemma}\label{lem:cob-weak-commute}
  If $\aTree$ and $\aTree'$ are multimorphisms in
  $\CCat{N}\ttimes \mTshape{m}{n}$ with the same source and target
  then the divided cobordisms $\mTfuncNF{T}(\aTree)^\circ$ and
  $\mTfuncNF{T}(\aTree')^\circ$ are isotopic.  
\end{lemma}
\begin{proof}
  Both $\mTfuncNF{T}(\aTree)^\circ$ and $\mTfuncNF{T}(\aTree')^\circ$
  are compositions of:
  \begin{itemize}
  \item multi-merge cobordisms of crossingless matchings,
  \item saddle cobordisms supported on small disks around certain
    crossings of $T$, which are disjoint from the crossingless
    matchings being merged, and
  \item multi-birth cobordisms, corresponding to stump leaves, each of which is followed by a multi-merge cobordism.
  \end{itemize}
  By Proposition~\ref{prop:births}, if we let $\aTree_0$ (respectively $\aTree'_0$) be the result of removing all stump leaves from $\aTree$ then $\mTfuncNF{T}(\aTree)^\circ$ and $\mTfuncNF{T}(\aTree_0)^\circ$ are isotopic, as are $\mTfuncNF{T}(\aTree')^\circ$ and $\mTfuncNF{T}(\aTree'_0)^\circ$. 
  Now, since the source and target of $\aTree_0$ and $\aTree'_0$ are the same,
  the cobordisms $\mTfuncNF{T}(\aTree_0)^\circ$ and
  $\mTfuncNF{T}(\aTree'_0)^\circ$ have saddles at the same crossings and
  merge the same crossingless matchings. Thus, the result follows from
  Corollary~\ref{cor:cob-commute} and the fact that all multi-merge
  cobordisms with the same source and target are isotopic.
\end{proof}

\begin{proposition}
  The map $\mTfuncNF{T}$ does, indeed, define a multifunctor
  $\CCat{N}\ttimes \mTshape{m}{n}\to \mCobD$.
\end{proposition}
\begin{proof}
  By Lemma~\ref{lem:cob-weak-commute}, the map $\mTfuncNF{T}$ is
  well-defined.  We must check that it respects multi-composition. At
  the level of objects of the multimorphism groupoids, since we
  defined $\mTfuncNF{T}(f)$ by composing the values of $\mTfuncNF{T}$
  on basic multimorphisms, this is immediate from the
  definition. Since each 2-morphism set in $\mCobD$ is empty or has
  $1$ element, at the level of morphisms of the multimorphism
  groupoids there is nothing to check.
\end{proof}
\setcounter{OutlineCompletion}{90}

\subsection{The initial invariant}\label{sec:tang-to-burn}
In this section, we will construct combinatorial tangle invariants as
equivalence classes of multifunctors to the Burnside
multicategory. Explicitly, to the $2n$ points
$[2n]_{\std}\subset(0,1)$, we associate the functor from the
multicategory $\mHshape{n}$ to the
Burnside
multicategory $\mBurnside$ 
\[
\mHinv{n}\coloneqq\mV\circ\mHfunc{n}\from\mHshape{n}\to\mBurnside,
\]
and to a tangle diagram $T\in\TangleDiags{m}{n}$ connecting
$\{0\}\times\{0\}\times[2m]_{\std}$ to
$\{0\}\times\{1\}\times[2n]_{\std}$, we associate the pair
$(\mTinvNF{T},N_+)$, where $\mTinvNF{T}$ is the functor
\[
  \mTinvNF{T}\coloneqq\mV\circ\mTfuncNF{T}\co
  \CCat{N}\ttimes\mTshape{m}{n}\to \mBurnside
\]
and $N_+$ is the number of positive crossings in the oriented tangle
diagram $T$. We will refer to this sort of pairs often, so we give it a name:
\begin{definition}\label{def:stable-func}
  A \emph{stable functor} from $\CCat{N}\ttimes\mTshape{m}{n}$ to
  $\mBurnside$ is a pair
  \[
    (\text{functor }F\co\CCat{N}\ttimes\mTshape{m}{n}\to\mBurnside,\text{integer }S)
  \]
  so that the restriction of $F$ to the subcategory $\mHshape{m}$
  (respectively $\mHshape{n}$) of $\CCat{N}\ttimes\mTshape{m}{n}$ is
  $\mHinv{m}$ (respectively $\mHinv{n}$).
\end{definition}

\subsubsection{Recovering the Khovanov invariants}
Given a functor $F_n\from\mHshape{n}\to\mBurnside$, we can compose
with the forgetful functor $\Forget\co\mBurnside\to\mAbelianGroups$ to
obtain a functor $\Forget\circ F_n\co \mHshape{n}\to\mAbelianGroups$.
Since $\mAbelianGroups$ is trivially enriched, the functor $\Forget\circ F_n$
descends to an un-enriched multifunctor, still denoted $\Forget\circ
F_n$, from the strictification $\mHshape{n}^0$ to $\mAbelianGroups$.

Similarly, given a stable functor
$(F\from\CCat{N}\ttimes\mTshape{m}{n}\to\mBurnside,S)$ 
we get a
functor
$\Forget\circ
F\from\strictify{\CCat{N}\ttimes\mTshape{m}{n}}\to\mAbelianGroups$. We
can associate to the pair $(\Forget\circ F,S)$ a functor
\[
  \Total{\Forget\circ F,S}\co \mTshapeStct{m}{n}^0\to \mComplexes,
\]
which restricts to $\Forget\circ F_m$ and $\Forget\circ F_n$ on the
subcategories $\mHshape{m}$ and $\mHshape{n}$, as follows. Given an
object $(a,b)\in\Ob(\mTshapeStct{m}{n})^0$ we let
\[
  \Total{\Forget\circ F,M}(a,b)=(\Forget\circ F)(a,b),
\]
viewed as a chain complex concentrated in grading $0$. Given an
object $(a,T,b)\in\Ob(\mTshapeStct{m}{n}^0)$ there is an associated
subcategory $\CCat{N}\times(a,T,b)$ of
$\strictify{\CCat{N}\ttimes\mTshape{m}{n}}$ isomorphic to the cube
$\CCat{N}$: it is the full subcategory spanned by objects of the form
$(v,a,T,b)$. Let $\Total{\Forget\circ F,M}(a,T,b)$ be the totalization
of the cube of abelian groups $\Forget\circ
F|_{\CCat{N}\times(a,T,b)}$, cf.~Equation~\eqref{eq:totalization},
followed by a downward grading shift by the integer $S$ (so that the
chain complex is supported in gradings $[-S,N-S]$).

\begin{lemma}\label{lem:is-Kh}
  The Khovanov arc algebra $\KTalg{m}$ (respectively $\KTalg{n}$) is
  the functor $\Forget\circ\mHinv{m}\from\mHshape{m}^0\to\mAbelianGroups$
  (respectively
  $\Forget\circ\mHinv{n}\from\mHshape{n}^0\to\mAbelianGroups$) which is
  the restriction of $\Total{\Forget\circ \mTinvNF{T},N_+}$ to
  $\mHshape{m}^0$ (respectively $\mHshape{n}^0$), and the Khovanov
  tangle invariant $\KTfunc(T)$ is the functor
  $\Total{\Forget\circ \mTinvNF{T},N_+}\from\mTshape{m}{n}\to\mComplexes$,
  reinterpreted per Principle~\ref{prin:multi-is-bim}.
\end{lemma}
\begin{proof}
  This is an exercise in unwinding the definitions.
\end{proof}
\setcounter{OutlineCompletion}{100}

\subsubsection{Invariance}\label{sec:initial-invariance}
Next we describe in what sense is the functor
$\mHinv{n}\from\mHshape{n}\to\mBurnside$ an invariant of $2n$ points,
and in what sense is the stable functor
$(\mTinvNF{T}\co \CCat{N}\ttimes\mTshape{m}{n}\to \mBurnside,N_+)$ an
invariant for the underlying tangle. First we consider $\mHinv{n}$.

Superficially, the functor $\mHinv{n}\co \mHshape{n}^0\to \mBurnside$
depended on a number of choices:
\begin{enumerate}[label=(C-\arabic*),ref=C-\arabic*]
\item\label{item:choice-1} The choice of curves representing each
  isotopy class of crossingless matching in $\Crossingless{n}$.
\item\label{item:choice-2} The choice of divided multi-merge
  cobordisms.
\item\label{item:choice-3} The choice of embeddings in the
  definitions of the Burnside multicategory
  (Section~\ref{sec:mBurnside}).
\end{enumerate}
To deal with this, we could make specific once-and-for-all choices; or
we can invoke the following:
\begin{definition}
  A \emph{natural isomorphism} $\eta$ between multifunctors $F,G$ from a
  groupoid enriched multicategory $\Cat$ to $\mBurnside$ is a
  collection of bijections $\eta_x\from F(x)\to G(x)$ for all
  objects $x\in\Ob(\Cat)$, and $\eta_\phi\from F(\phi)\to G(\phi)$
  for all multimorphisms $\phi\in\Hom(x_1,\dots,x_n;y)$ which are
  compatible with the 2-morphisms and the source and the target
  maps in the following sense: for any objects $x_1,\dots,x_n,y\in\Ob(\Cat)$, any
  multimorphisms $\phi,\psi\in\Hom(x_1,\dots,x_n;y)$, any 2-morphism
  $\kappa\from\phi\to\psi$, and any element $w\in F(\phi)$,
  \begin{align*}
    \eta_\psi\bigl(F(\kappa)(w)\bigr)&=G(\kappa)\bigl(\eta_\phi(w)\bigr),&
    (\eta_{x_1},\dots,\eta_{x_n})\bigl(s(w)\bigr)&=s\bigl(\eta_\phi(w)\bigr),&
    \eta_{y}\bigl(t(w)\bigr)&=t\bigl(\eta_\phi(w)\bigr).
  \end{align*}
\end{definition}
\begin{lemma}\label{lem:transitive-system}
  Let $\mHinv{n}^1,\mHinv{n}^2\co \mHshape{n}\to\mBurnside$ be the
  functors associated to two different choices of curves, multi-merge
  cobordisms, and embeddings of associated sets. Then there is a
  natural isomorphism $\eta^{12}\co \mHinv{n}^1\to
  \mHinv{n}^2$. Further, these maps $\eta$ form a transitive system,
  in the sense that $\eta^{11}$ is the identity and if
  $\mHinv{n}^3\co \mHshape{n}\to\mBurnside$ is the functor associated
  to a third collection of choices then
  $\eta^{13}=\eta^{23}\circ\eta^{12}$.
\end{lemma}
\begin{proof}
  Since the $2$-morphisms in the Burnside multicategory pay no
  attention to the embeddings of the correspondences, the identity
  $2$-morphisms give a transitive system of natural isomorphisms
  associated to changing the embeddings of correspondences. Similarly,
  any two choices of divided multi-merge cobordisms are uniquely
  isomorphic (because isotopic divided cobordisms are uniquely
  isomorphic), so different choices of decorated cobordisms give
  naturally isomorphic functors, and these natural isomorphisms are
  transitive. Next, any two choices of representatives of the
  crossingless matchings are related by an obvious
  divided cobordism, the trace of an isotopy between the two representatives, and this divided cobordism is unique up to unique
  isomorphism. Independence from the choice of curves representing the
  crossingless matchings follows. 
  Finally, the maps in these three transitive systems commute with
  each other in an obvious sense, so we can view them all together as
  a single transitive system. This completes the proof.
\end{proof}

Now, consider the groupoid $\Cat$ with:
\begin{itemize}
\item Objects sets of choices~(\ref{item:choice-1})--(\ref{item:choice-3}).
\item A unique morphism between each pair of objects.
\end{itemize}
Lemma~\ref{lem:transitive-system} asserts
that we have a functor $\Cat\to\Fun(\mHshape{n},\mBurnside)$, where
$\Fun(\mHshape{n},\mBurnside)$ is the category of functors from
$\mHshape{n}\to\mBurnside$ with morphisms being natural
isomorphisms. Existence of this functor on the contractible groupoid
$\Cat$ expresses the fact that different choices are canonically
isomorphic.

Following the standard colimit procedure, we can harness the above
fact to construct $\mHinv{n}$ as a functor independent of choices. For
any object $x$ and any multimorphism $\phi$ of $\mHshape{n}$, define
\[
\mHinv{n}(x)=\coprod_{i\in\Ob(C)}\mHinv{n}^i(x)/\sim\text{ and }\mHinv{n}(\phi)=\coprod_{i\in\Ob(C)}\mHinv{n}^i(\phi)/\sim,
\]
where the equivalence relation $\sim$ identifies $u\in\mHinv{n}^i(x)$
(respectively, $w\in \mHinv{n}^i(\phi)$) with
$\eta^{i,j}_x(u)\in\mHinv{n}^j(x)$ (respectively,
$\eta^{i,j}_\phi(w)\in \mHinv{n}^i(\phi)$) for any $i,j\in\Ob(\Cat)$,
with the source, target, and 2-morphism maps defined componentwise.

For the rest of the paper, we will elide the fact that
$\mHinv{n}\co \mHshape{n}\to \mBurnside$ depended on choices, and
expect the reader to either assume we made once-and-for-all choices in
defining $\mHinv{n}$, or insert the discussion above where appropriate.

Next we turn to $\mTinvNF{T}$.
\begin{definition}
  Given multifunctors
  $F,G\co\CCat{N}\ttimes\mTshape{m}{n}\to\mBurnside$, and any integer
  $S$, a \emph{natural transformation} connecting the stable functors
  $(F,S)$ to $(G,S)$ is a multifunctor $H\co
  \CCat{N+1}\ttimes\mTshape{m}{n}\to\mBurnside$ so that
  $H|_{\{0\}\times\CCat{N}\ttimes\mTshape{m}{n}}=F$ and
  $H|_{\{1\}\times\CCat{N}\ttimes\mTshape{m}{n}}=G$. A natural
  transformation from $(F,S)$ to $(G,S)$ induces a homomorphism of \dg modules
  $\Total{\Forget\circ F,S}\to \Total{\Forget\circ G,S}$ in an obvious
  way, where $\Total{\Forget\circ F,S}$ and $\Total{\Forget\circ G,S}$
  are being viewed as \dg bimodules as per
  Section~\ref{sec:multi-vec}. We call $H$ a
  \emph{quasi-isomorphism} if the induced chain map is a
  quasi-isomorphism.
\end{definition}

\begin{proposition}
  Up to quasi-isomorphism, the stable functor $(\mTinvNF{T},N_+)$ is
  independent of the choices of pox, resolutions, and cobordisms in the
  definition of $\mTfuncNF{T}$.
\end{proposition}
\begin{proof}
  First, since the value of $\mV$ on objects and 1-morphisms is given by the
  functor $V_{\HKKa}\co\CobE\to\BurnsideCat$, which does not depend on the pox,
  adding more pox does not change $\mV$. Thus, $\mTinvNF{T}$ is independent of
  the choice of pox.
  
  Next, fix choices $\mTfuncNF{T}^0$ and $\mTfuncNF{T}^1$ of resolutions and
  cobordisms, with respect to the same pox. We will define a natural
  transformation $H\co \CCat{N+1}\ttimes\mTshape{m}{n}\to\mCobD$ from
  $\mTfuncNF{T}^0$ to $\mTfuncNF{T}^1$ and then compose with $\mV$ to get a
  natural transformation from $\mTinvNF{T}^0$ to $\mTinvNF{T}^1$.

  On the subcategories $\mHshape{m}$ and
  $\mHshape{n}$ of $\CCat{N}\ttimes\mTshape{m}{n}$,
  $\mTfuncNF{T}^0$ and $\mTfuncNF{T}^1$ already agree.
  From the definition of $\ttimes$, to define $H$ on the objects of the
  multimorphism groupoids, it suffices to define $H$ on the maps
  $f_{\CCat{N+1}}\times \Id_{(a,T,b)}$, where $f_{\CCat{N+1}}\co
  (0,v)\to (1,w)$ is a morphism from $\{0\}\times\CCat{N}$ to
  $\{1\}\times\CCat{N}$, since $H$ has already been defined on the
  other type of elementary morphisms. Define $H(f_{\CCat{N+1}}\times
  \Id_{(a,T,b)})$ to be any multi-saddle cobordism from the resolution
  $T_{v}$ with respect to the first set of choices to
  the resolution $T_{w}$ with respect to the second set of choices.
  (This is actually a slight variant of the multi-saddle cobordisms
  from Section~\ref{sec:cabinet}: there, outside certain of the $D_i$ the
  cobordism was a product, while here it is the trace of an isotopy
  between the different choices of resolutions. In particular, if
  $v=w$ the cobordism is a deformed copy of the identity
  cobordism.)
  The extension of
  $H$ to morphisms in the multimorphism groupoids proceeds without
  incident as in the construction of $\mTfuncNF{T}$ using
  Lemma~\ref{lem:cob-weak-commute}.

  The induced diagram of chain complexes $\Total{\Forget\circ
    \mV\circ H,N_+}$ sends the arrows $(0,v)\to (1,v)$ to identity
  maps. Thus, the map
  $\Total{\Forget\circ\mTinvNF{T}^0,N_+}\to\Total{\Forget\circ
    \mTinvNF{T}^1,N_+}$ is the identity map,
  and hence is a quasi-isomorphism (indeed, an isomorphism).
\end{proof}

\begin{convention}
  For the rest of the paper, we will usually suppress the choice of pox,
  resolutions, and cobordisms in the definition of $\mTfuncNF{T}$, and view
  $\mTfuncNF{T}$ as associated to the tangle diagram $T$.
\end{convention}

\begin{definition}
A \emph{face inclusion} is a functor $i\from\CCat{M}\to\CCat{N}$ that
is injective on objects and preserves relative gradings (see
\cite[Definition~5.5]{LLS-cube}). Let $|i|$ be the absolute grading
shift of $i$, given by $|i(v)|-|v|$ for any $v\in\Ob(\CCat{M})$, where
$|\cdot|$ denotes the height (number of $1$'s) in the cube. Given a
stable functor $(F\co\CCat{M}\ttimes\mTshape{m}{n}\to\mBurnside,S)$
and a face inclusion $i\co \CCat{M}\hookrightarrow \CCat{N}$ there is
an induced stable functor
$(i_!F\co \CCat{N}\ttimes\mTshape{m}{n}\to\mBurnside,S+N-M-|i|)$,
where $i_!F$ is defined as follows:
\begin{itemize}
\item On objects of the form $(a,b)$, $(i_!F)(a,b)=F(a,b)$. On objects
  of the form $(v,a,T,b)$,
  \[
    (i_!F)(v,a,T,b)=\begin{cases}F(u,a,T,b)&\text{if $v=i(u)$ is in the image of $i$,}\\
      \emptyset&\text{otherwise.}\end{cases}
  \]
\item On multimorphisms, if all of the input and output leaves of a
  tree $\aTree$ are labeled by elements $(v,a,T,b)$ with $v$ in the
  image of $i$ or by pairs $(a_1,a_2)$ or $(b_1,b_2)$, then the same must be true for all intermediate edges
  and vertices, so there is a tree $\aTree'$ with
  $i(\aTree')=\aTree$ (in the obvious sense), and we define
  $(i_!F)(\aTree)=F(\aTree')$.
  Otherwise, $(i_!F)(\aTree)$ is the
  empty correspondence. (Note that, in the second case, at least one
  of the source or target of $(i_!F)(\aTree)$ is the empty set.)
\end{itemize}
We call $(i_!F,S+N-M-|i|)$ a \emph{stabilization} of $(F,S)$ and
$(F,S)$ a \emph{destabilization} of $(i_!F,S+N-M-|i|)$. The
\dg bimodules $\Total{\Forget\circ F,S}$ and
$\Total{\Forget\circ i_!F,S+N-M-|i|}$ are isomorphic, and the isomorphism is
canonical up to an overall sign.
\end{definition}

Call stable functors
$(F\co \CCat{M}\ttimes\mTshape{m}{n}\to\mBurnside,R)$ and
$(G\co \CCat{N}\ttimes\mTshape{m}{n}\to\mBurnside,S)$ \emph{stably
  equivalent} if $(F,R)$ and $(G,S)$ are related by a sequence of
quasi-isomorphisms, stabilizations, and destabilizations.

There are some convenient ways to produce equivalences:
\begin{definition}\label{def:insular-subfunctor}
  Given a functor $F\co \CCat{N}\ttimes\mTshape{m}{n}\to\mBurnside$,
  an \emph{insular subfunctor} of $F$ is a collection of subsets
  $G(v,a,T,b)\subset F(v,a,T,b)$, such that for any $x_i\in
  F(a_i,a_{i+1})$, $y\in G(u,a_k,T,b_1)$, $z_i\in F(b_i,b_{i+1})$,
  $w\in F(v,a_1,T,b_\ell)\setminus G(v,a_1,T,b_\ell)$, and
  \[
    f\in\Hom((a_1,a_2),\dots,(a_{k-1},a_k),(u,a_k,T,b_1),(b_1,b_2),\dots,(b_{\ell-1},b_\ell);(v,a_1,T,b_\ell)),
  \]
  \begin{equation}
    s^{-1}(x_1,\dots,x_{k-1},y,z_1,\dots,z_{\ell-1})\cap t^{-1}(w)=\emptyset\subset F(f).\label{eq:introvert}
  \end{equation}
  
  Extend $G$ to a functor
  $G\co \CCat{N}\ttimes\mTshape{m}{n}\to\mBurnside$ by defining
  $G(a,b)=F(a,b)$ for $(a,b)\in\Ob(\mHshape{m})\cup\Ob(\mHshape{n})$
  and, for $f\in\Hom(p_1,\dots,p_n;q)$,
  \[
    G(f)=s^{-1}(G(p_1)\times\dots\times G(p_n))\cap t^{-1}(G(q))\subset F(f)
  \]
  with source and target maps induced by $s$ and $t$, and maps of $2$-morphisms induced
  by $F$ in the obvious way. The fact that $G$ respects composition
  follows from Equation~\eqref{eq:introvert}.
  
  Given an insular subfunctor $G$ of $F$ there is a \emph{quotient
    functor} $F/G\co \CCat{N}\ttimes\mTshape{m}{n}\to\mBurnside$
  defined by:
  \begin{itemize}
  \item $(F/G)(a,b)=F(a,b)$,
  \item $(F/G)(v,a,T,b)=F(v,a,T,b)\setminus G(v,a,T,b)$, the
    complement of $G(v,a,T,b)$, 
  \item $(F/G)(f)=s^{-1}((F/G)(p_1)\times\dots\times (F/G)(p_n))\cap t^{-1}((F/G)(q))\subset F(f)$ for $f\in\Hom(p_1,\dots,p_n;q)$, and
  \item the value of $F/G$ on $2$-morphisms is induced by $F$.
  \end{itemize}
  Again, the fact that this defines a functor follows from Equation~\eqref{eq:introvert}.
\end{definition}

Given an insular subfunctor $G$ of $F$, and any integer $S$, there
is an induced short exact sequence of \dg bimodules
\[
  0\to\Total{\Forget\circ G,S}\into \Total{\Forget\circ F,S}\onto \Total{\Forget\circ (F/G),S}\to0.
\]

\begin{lemma}\label{lem:introvert}
  Fix any integer $S$. If $G$ is an insular subfunctor of $F$ then
  there is a natural transformation $\eta$ from $(G,S)$ to $(F,S)$ so that the
  induced map of differential bimodules is the inclusion map defined
  above. There is also a natural transformation $\theta$ from $(F,S)$ to
  $(F/G,S)$ so that the induced map of differential bimodules is the
  quotient map defined above. In particular, if the inclusion
  (respectively quotient) map of chain complexes is a
  quasi-isomorphism then the map $\eta$ (respectively $\theta$) is an
  equivalence.
\end{lemma}
\begin{proof}
  To define $\eta$ (respectively $\theta$), for
  \[
    f\in \Hom((a_1,a_2),\dots,(a_{k-1},a_k),((0,u),a_k,T,b_1),(b_1,b_2),\dots,(b_{\ell-1},b_\ell);((1,v),a_1,T,b_\ell))
  \]
  a basic multimorphism there is a corresponding basic multimorphism
  \[
    \wt{f}\in \Hom((a_1,a_2),\dots,(a_{k-1},a_k),(u,a_k,T,b_1),(b_1,b_2),\dots,(b_{\ell-1},b_\ell);(v,a_1,T,b_\ell)).
  \]
  Define $\eta(f)=G(\wt{f})$ (respectively
  $\theta(f)=(F/G)(\wt{f})$). Similarly, on 2-morphisms $\eta$
  (respectively $\theta$) is induced by $G$ (respectively $F/G$).  It
  is straightforward to verify that these definitions make $\eta$ and
  $\theta$ into natural transformations with the desired properties.
\end{proof}

\begin{theorem}\label{thm:comb-inv}
  The stable equivalence class of $\mTinvNF{T}$ is invariant under
  Reidemeister moves, and so gives a tangle invariant.  Further, the
  chain map
  \[
    \Total{\Forget\circ\mTinvNF{T_1},N_+(T_1)}\to \Total{\Forget\circ\mTinvNF{T_2},N_+(T_2)}
  \]
  induced by a sequence of Reidemeister moves relating $T_1$ and $T_2$
  agrees, up to a sign and homotopy, with Khovanov's invariance maps~\cite[Section
  4]{Kho-kh-tangles}.
\end{theorem}
\begin{proof}
  This is essentially a translation of the invariance proof for the
  Khovanov homotopy type~\cite[Section 6]{RS-khovanov} (itself a modest extension
  of invariance proofs for Khovanov homology) to the language of this
  paper.

  \newcommand{\vreflect}[1]{\reflectbox{\rotatebox[origin=c]{180}{#1}}}

  \begin{figure}
    \centering
      \subfloat[RI.]{
    \xymatrix@C=0.1\textwidth{
      \vcenter{\hbox{\includegraphics[height=0.3\textheight]{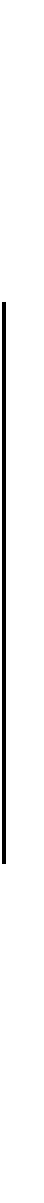}}}\ar@{->}[r]& \vcenter{\hbox{\includegraphics[height=0.3\textheight]{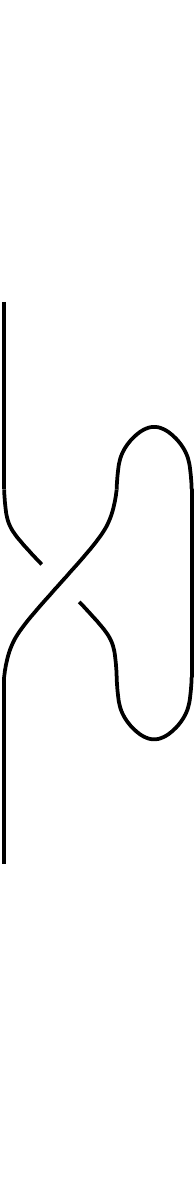}}}
    }
  }
  \hspace{0.03\textwidth}
  \subfloat[RII.]{
    \xymatrix@C=0.1\textwidth{
      \vcenter{\hbox{\includegraphics[height=0.3\textheight]{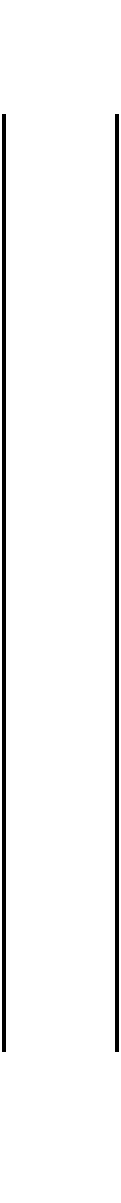}}}\ar@{->}[r]& \vcenter{\hbox{\includegraphics[height=0.3\textheight]{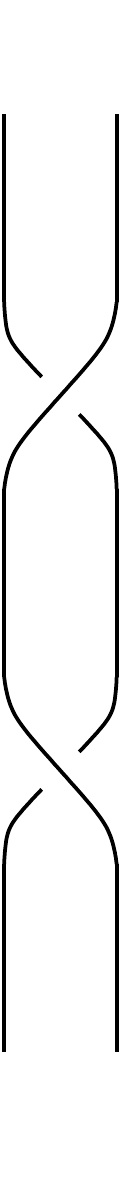}}}
    }
  }
  \hspace{0.03\textwidth}
  \subfloat[RIII.]{
    \xymatrix@C=0.1\textwidth{
      \vcenter{\hbox{\includegraphics[height=0.3\textheight]{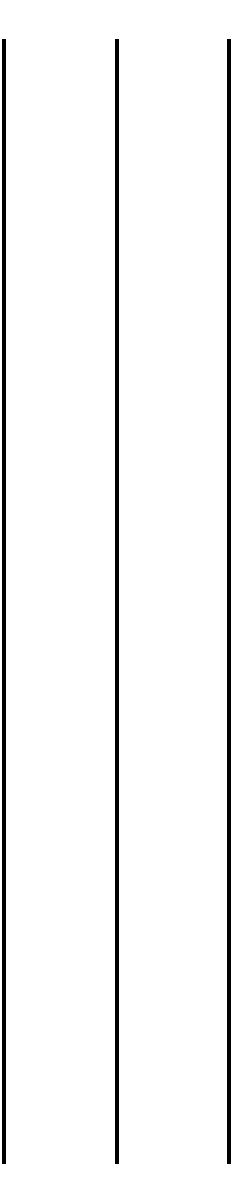}}}\ar@{->}[r]& \vcenter{\hbox{\includegraphics[height=0.3\textheight]{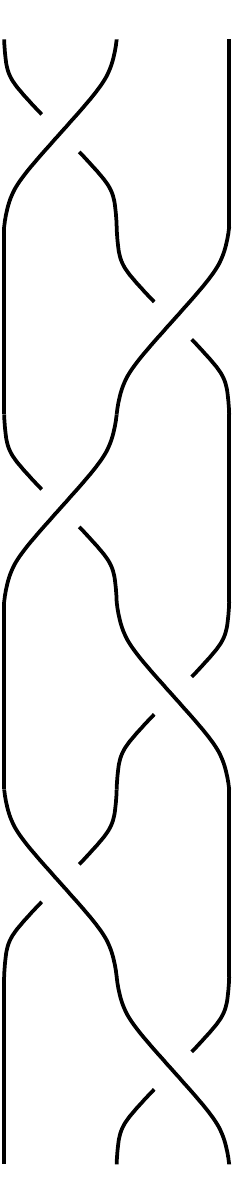}}}
    }
  }
    \caption{\textbf{Reidemeister moves.} The orientations of the strands are arbitrary. This figure originally appeared in~\cite{RS-khovanov}.}
    \label{fig:ReidemeisterMoves}
  \end{figure}

  It suffices to verify invariance under reordering of the crossings
  and the three Reidemeister moves shown in
  Figure~\ref{fig:ReidemeisterMoves}, because this Reidemeister I and
  the Reidemeister II move generate the other Reidemeister I move, and
  the usual Reidemeister III move is generated by this braid-like
  Reidemeister III move and Reidemeister II moves (see \cite[Section
  7.3]{Baldwin-hf-s-seq}).

  If $T\in\TangleDiags{m}{n}$ is a $(2m,2n)$-tangle diagram with $N$
  ordered crossings, and if $T'\in\TangleDiags{m}{n}$ is the same
  tangle diagram, but with its crossings reordered by some permutation
  $\sigma\in\symGrp_N$, then the stable functor $(\mTfuncNF{T'},N_+)$
  is the stabilization $(i_!\mTfuncNF{T},N_+)$, where
  $i\from\CCat{N}\to\CCat{N}$ is the face inclusion
  $(v_1,\dots,v_N)\mapsto (v_{\sigma(1)},\dots,v_{\sigma(N)})$.

  Next we turn to the Reidemeister I move. Let
  $T\in\TangleDiags{m}{n}$ be a $(2m,2n)$-tangle diagram with $N$
  ordered crossings, of which $N_+$ are positive, and $T'$ the result
  of performing a Reidemeister I move to $T$ as in
  Figure~\ref{fig:ReidemeisterMoves}, so $T'$ has one more positive
  crossing $c$ than $T$; assume $c$ is the $(N+1)\st$ crossing of
  $T'$. Note that the $1$-resolution of $c$ gives a tangle isotopic to
  $T$ and the $0$-resolution of $c$ gives the disjoint union of $T$
  and a small circle $C$. For each object $(v,a,T,b)\in
  \Ob(\CCat{N+1}\ttimes\mTshape{m}{n})$ define $G(v,a,T,b)\subset
  \mTfuncNF{T'}(v,a,T,b)$ as
  \[
  G(v,a,T,b)=\begin{cases}
    \mTfuncNF{T'}(v,a,T,b)&\text{if $v_{N+1}=1$}\\
    \{w\in\mTfuncNF{T'}(v,a,T,b)\mid\text{$w$ assigns $1$ to $C$}\}&\text{if $v_{N+1}=0$}.
  \end{cases}
  \]
  (Compare~\cite[Figure 6.2]{RS-khovanov}.) We claim that $G$ is an
  insular subfunctor of $\mTfuncNF{T'}$ and that the chain
  complex associated to $G$ is acyclic. The second statement is
  clear. For the first, note that every element $w\in
  \mTfuncNF{T'}(v,a,T,b)\setminus G(v,a,T,b)$ is supported over the
  $0$-resolution at $c$, and assigns $X$ to the small circle $C$. The
  maps associated to the algebra action respect the labeling of $C$,
  and the edges in the cube go from the $0$-resolution to the
  $1$-resolution, and hence either do not change the crossing $c$ or
  map to a resolution in which $G(v,a,T,b)=\mTfuncNF{T'}(v,a,T,b)$.

  Thus, by Lemma~\ref{lem:introvert}, $(\mTfuncNF{T'},N_++1)$ is
  stably equivalent to $(\mTfuncNF{T'}/G,N_++1)$. If $i\from
  \CCat{N}\to\CCat{N+1}$ is the face inclusion $(v_1,\dots,v_N)\mapsto
  (v_1,\dots,v_N,0)$, forgetting the circle $C$ gives an isomorphism
  from $(\mTfuncNF{T'}/G,N_++1)$ to $(i_!\mTfuncNF{T},N_++1)$, which
  is a stabilization of $(\mTfuncNF{T},N_+)$.

  The proofs of Reidemeister II and III invariance are similar
  adaptations of the proofs from our previous
  paper~\cite[Propositions~6.3 and~6.4]{RS-khovanov}. For Reidemeister
  II invariance, that proof defines a contractible insular
  subfunctor $G_1$ of $\mTfuncNF{T'}$ and an insular subfunctor
  $G_3$ of the quotient $G_2=\mTfuncNF{T'}/G_1$ so that the quotient
  $G_4=G_2/G_3$ is contractible, and $G_3$ is isomorphic to
  $\mTfuncNF{T}$ modulo the correct grading shifts. (See
  particularly~\cite[Figure~6.3]{RS-khovanov}, where circles labeled
  $1$ are denoted $+$ and circles labeled $X$ are denoted $-$.) The
  new point is that all of these subsets are preserved by the algebra
  action; but this is immediate from their definitions, which only
  involve restricting to certain vertices of the cube or restricting
  the labels of certain closed circles. Similarly, for Reidemeister
  III invariance the old proof gives a sequence of insular
  subfunctors inducing equivalences. Further details are left to the
  reader.

  \begin{figure}
    \centering
    \includegraphics{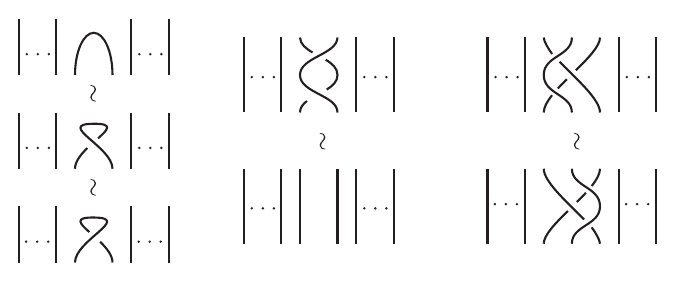}
    \caption{\textbf{Local Reidemeister moves.} Khovanov's invariance proof shows that the bimodules before and after each Reidemeister move are quasi-isomorphic; and in fact there is an essentially unique, up to sign, quasi-isomorphism between them.}
    \label{fig:local-Reid}
  \end{figure}

  The second part of the statement follows from the fact that,
  locally, up to sign there is a unique homotopy class of homotopy equivalences of
  $(\KTalg{n},\KTalg{n-2})$-bimodules (respectively
  $(\KTalg{n},\KTalg{n})$-bimodules) corresponding to a Reidemeister I
  move (respectively II or III move). (See
  Figure~\ref{fig:local-Reid}.) Both the map on the chain complexes induced by
  the construction above and Khovanov's map respect composition of
  tangles and so are induced from local maps. See our previous
  paper~\cite[Proposition 3.4]{RS-s-invariant} for further details.
\end{proof}

\section{From combinatorics to topology}\label{sec:comb-to-top}
\subsection{Construction of the spectral categories and bimodules}\label{sec:build-spec-bim}
We warm up by giving a functor $G\co \mHshape{n}^0\to \mSpectra$ refining
the arc algebras. In Section~\ref{sec:tang-to-burn} we defined a functor
$\mHinv{n}\co \mHshape{n}\to\mBurnside$. The Burnside multicategory
maps to the multicategory of permutative categories $\PermuCat$, by
taking a set $X$ to the category $\Sets/X$ of finite sets over $X$, and a
correspondence $A\co X\to Y$ to the functor $\Sets/X\to\Sets/Y$ given
by fiber product with $A$ (cf.\ Section~\ref{sec:EM-background}).  Elmendorf-Mandell define a multifunctor
$\PermuCat\to \mSpectra$, $K$-theory, where $\mSpectra$ is the
multicategory of symmetric spectra (with multicategory structure
induced by the smash product)~\cite[Theorem
1.1]{EM-top-machine}. (Again, see Section~\ref{sec:EM-background}.)  So, composing with this
functor gives us a functor
\[
  \mHshape{n}\to \mSpectra.
\]
Rectification as in Definition~\ref{def:rectification}
combined with Lemma~\ref{lem:unenrich}, turns this into a functor
\begin{equation}\label{eq:G-on-mHshape}
  G\co \mHshape{n}^0\to \mSpectra.
\end{equation}
 
The story for tangles is similar.  Given a tangle diagram
$T\in\TangleDiags{m}{n}$ (with $N$ ordered crossings, of which $N_+$
are positive), in Section~\ref{sec:tang-to-burn} we defined a stable functor $(\mTinvNF{T}\co
\CCat{N}\ttimes\mTshape{m}{n}\to\mBurnside,N_+)$.  Compose
$\mTinvNF{T}$ with the map $\mBurnside\to\PermuCat$ to get a functor
$\CCat{N}\ttimes\mTshape{m}{n}\to\PermuCat$.
Applying Elmendorf-Mandell's $K$-theory functor~\cite[Theorem
1.1]{EM-top-machine} as before gives us a functor
\[
  \CCat{N}\ttimes\mTshape{m}{n}\to \mSpectra.
\]
Rectification as in Definition~\ref{def:rectification}
turns this into a functor
\[
  F\co \strictify{\CCat{N}\ttimes\mTshape{m}{n}}\to \mSpectra
\]
from the strictified product. Note that
$\mHshape{m}\cup\mHshape{n}$ is a blockaded subcategory of
$\mTshape{m}{n}$, so by Lemma~\ref{lem:rectify-restrict}, on
$\mHshape{m}^0\cup\mHshape{n}^0$ the functor $F$ agrees with
the map $G$ from Equation~\ref{eq:G-on-mHshape}.\setcounter{OutlineCompletion}{110} 

Recall from Section~\ref{sec:tang-to-burn} that for each pair of
crossingless matchings $a\in\Crossingless{m}$ and
$b\in\Crossingless{n}$ we have a cube $\CCat{N}\times(a,T,b)$ in
$\strictify{\CCat{N}\ttimes\mTshape{m}{n}}$. The restriction of $F$ to
$\CCat{N}\times(a,T,b)$ is a functor
$F|_{(a,T,b)}\co \CCat{N}\to\mSpectra$.  Next we take the iterated
mapping cone of $F|_{(a,T,b)}$. That is, adjoin an additional object
$*$ to $\CCat{1}$ with a single morphism $0\to *$, to obtain a larger
category $\CCat{1}_+$. (This category is denoted $P$ in
Corollary~\ref{cor:totalization}.) Let
$\CCat{N}_+=(\CCat{1}_+)^N$. Extend $F|_{(a,T,b)}$ to
$F|_{(a,T,b)}^+\co \CCat{N}_+\to \mSpectra$ by declaring that
$F|_{(a,T,b)}^+(x)=\{\pt\}$, a single point, if
$x\not\in\Ob(\CCat{N})$. Then the iterated mapping cone of
$F|_{(a,T,b)}$ is the homotopy colimit $\hocolim F|_{(a,T,b)}^+$.

Now, define
\[
  G\co \mTshapeStct{m}{n}^0\to \mSpectra
\]
by defining
\begin{align*}
  G(a,b)&=F(a,b) &
  G(a,T,b)&=\sh^{-N_+}\hocolim_{\CCat{N}_+} F|_{(a,T,b)}^+.
\end{align*}
In fact, on the entire subcategory $\mHshape{m}^0\cup\mHshape{n}^0$,
define $G$ to agree with $F$ (and hence also the map $G$ from Equation~\eqref{eq:G-on-mHshape}). The map
\[
G(f_{a_1,\dots,a_k,T,b_1,\dots,b_\ell})\co
G(a_1,a_2)\wedge \cdots\wedge G(a_k,T,b_1)\wedge\cdots\wedge G(b_{\ell-1},b_\ell)\to G(a_1,T,b_\ell)
\]
is the composition
\begin{align*}
  &G(a_1,a_2)\wedge \cdots\wedge G(a_k,T,b_1)\wedge\cdots\wedge G(b_{\ell-1},b_\ell)\\
  &\qquad\qquad\qquad= 
    F(a_1,a_2)\wedge \cdots\wedge \left[\sh^{-N_+}\hocolim_{\CCat{N}_+} F|_{(a_k,T,b_1)}^+\right]\wedge\cdots\wedge F(b_{\ell-1},b_\ell)\\
  &\qquad\qquad\qquad\cong \sh^{-N_+}\hocolim_{\CCat{N}_+}\left[F(a_1,a_2)\wedge \cdots\wedge F|_{(a_k,T,b_1)}^+ \wedge\cdots\wedge F(b_{\ell-1},b_\ell)\right]\\
  &\qquad\qquad\qquad\to \sh^{-N_+}\hocolim_{\CCat{N}_+}F|_{(a_1,T,b_\ell)}^+=G(a_1,T,b_\ell),
\end{align*}
where the last map comes from naturality of the shift functor and
homotopy colimits (see Propositions~\ref{prop:shift-fun}
and~\ref{prop:hocolim-props}) and the fact that $F$ is a
multifunctor.

\begin{lemma}
  This definition makes $G$ into a multifunctor.
\end{lemma}
\begin{proof}
  Again, this follows from naturality of shift functors and homotopy colimits,
  and the fact that $F$ is a multifunctor.
\end{proof}\setcounter{OutlineCompletion}{120}

\begin{proposition}\label{prop:homology-right}
  Composing $G$ and the chain functor $\mSpectra\to\mComplexes$
  gives a map $\mTshapeStct{m}{n}^0\to\mComplexes$ which is
  quasi-isomorphic to the Khovanov tangle invariant (reinterpreted as
  in Section~\ref{sec:multi-vec}).
\end{proposition}

The following result will be useful in proving Proposition~\ref{prop:homology-right}.
\begin{lemma}
  \label{lem:connective-zigzag}
  Let $\MultiCat$ be a multicategory and suppose $K$ is any multifunctor
  $\MultiCat \to \mComplexes$.
  Let $\tau_{\geq 0}$ be the \emph{connective cover} functor on
  $\mComplexes$, sending a complex $C$ to the following subcomplex:
  \[
  (\tau_{\geq 0}C)_k =
  \begin{cases}
    C_k &\text{if }k > 0\\
    \ker(d_0) &\text{if }k = 0\\
    0 &\text{if }k < 0.
  \end{cases}
  \]
  Then there are natural transformations
  \[
  K \leftarrow \tau_{\geq 0} \circ K \to H_0 \circ K
  \]
  of multifunctors $\MultiCat \to \mComplexes$. If, for any
  $x \in \Ob(\MultiCat)$, the complex $K(x)$ has no homology in
  negative (respectively positive, nonzero) degrees, the left-hand map
  (respectively the right-hand map, each of the maps) is a natural
  quasi-isomorphism.
\end{lemma}

\begin{proof}
  The map $\tau_{\geq 0}$ is a multifunctor $\mComplexes \to
  \mComplexes$, and comes with a natural transformation $\tau_{\geq 0}
  \to \Id$ (an inclusion map of complexes), inducing an
  isomorphism on homology in non-negative degrees, and a natural transformation $\tau_{\geq 0} \to H_0$ of
  multifunctors (a quotient map of complexes), inducing an isomorphism on $H_0$.
  Putting these together, for a functor $K$ as described the composite
  maps
  \[
  K \leftarrow \tau_{\geq 0} \circ K \to H_0 \circ K
  \]
  are natural transformations of multifunctors $\MultiCat \to
  \mComplexes$; and the left-hand (resp.~right-hand) arrow is a
  quasi-isomorphism if $K$ has homology groups supported in
  non-negative (resp. non-positive) degrees.
\end{proof}

\begin{proof}[Proof of Proposition~\ref{prop:homology-right}]
  We begin by observing that the functor
  \[
  C_* \circ F\co \strictify{\CCat{N}\ttimes\mTshape{m}{n}}\to
  \mComplexes
  \]
  has homology concentrated in degree zero: the spectra $G(a,b)$ and
  $F|_{(a,T,b)}(v)$ are wedge sums of copies of the sphere spectrum
  $\SphereS$. Therefore, the previous lemma provides us with a
  quasi-isomorphism between the multifunctor $C_* \circ F$ and the
  multifunctor $H_0 \circ F$.

  The identification $H_0(G(a,b))\cong 1_a\KTalg{n}1_b$ is
  obvious: $F(a,b)$ is a wedge
  sum of spheres, one for each Khovanov generator. (See Section~\ref{sec:EM-background}.) Similarly, for each
  vertex $v\in\CCat{N}$, $F|_{(a,T,b)}(v)$ is a wedge sum of copies of
  the sphere spectrum $\SphereS$, one for each element of
  $\mTinvNF{T}(a,T,b)$, so $H_0(F|_{(a,T,b)}(v))\cong
  \Forget(\mTinvNF{T}(v,a,T,b))$.  Further, the map on homology
  associated to each edge $v\to w$ of the cube is the map
  $\Forget(\mTinvNF{T}((v,a,T,b)\to(w,a,T,b))$.

  We must check that the composition maps agree with the
  Khovanov composition maps. For definiteness, consider the map
  $F(a,b)\smas F(b,T,c)\to F(a,T,c)$. There is a corresponding map
  \[
    (H_0 \circ F)|_{(a,b)}\otimes (H_0 \circ F)|_{(b,T,c)}(v)
    \to (H_0 \circ F)|_{(a,T,c)}(v)
  \]
  that is natural in $v \in \CCat{N}$. Tracing through the
  isomorphisms above, this is exactly the Khovanov multiplication
  \[
    1_a\KTalg{n}1_b\otimes 1_b\KTfunc(T_v)1_c\to 1_a\KTfunc(T_v)1_c.
  \]

  Thus, the multifunctor $H_0 \circ F$ represents (up to shift)
  precisely the cubical diagram of bimodules over the arc algebras
  whose totalization is $1_a\KTfunc(T)1_b$. As quasi-isomorphisms
  preserve shifts and homotopy colimits (see Proposition~\ref{prop:hocolim-props}), our
  quasi-isomorphism from $F$ to $H_0 \circ F$ becomes a
  quasi-isomorphism
  \begin{equation}\label{eq:CG-hocolim}
    C_* G(a,T,b) \simeq \hocolim_{\CCat{N}_+} (H_0 \circ F)|_{(a,T,b)}[-N_+].
  \end{equation}
  By Corollary~\ref{cor:totalization}, this homotopy colimit is
  precisely the total complex
  $\Total{\Forget\circ \mTinvNF{T}|_{(a,T,b)}},N_+)$, which is the
  bimodule $1_a\KTfunc(T)1_b$. Since the quasi-isomorphisms respected
  composition and Equation~\eqref{eq:CG-hocolim} is natural, the identification
  $C_* G(a,T,b)\simeq 1_a\KTfunc(T)1_b$ respects multiplication.
  This proves the result.
\end{proof}

We could stop here, and define $G$ to be our stable homotopy refinement of
the Khovanov tangle invariants, but we can make the invariant look a
little closer to Khovanov's invariant by reinterpreting it as a
spectral category. That is, we will refine $\KTalg{n}$ to a category
$\KTSpecCat{n}$ with:
\begin{itemize}
\item Objects crossingless matchings.
\item $\Hom(a,b)$ a symmetric spectrum.
\item Composition a map $\Hom(b,c)\smas\Hom(a,b)\to \Hom(a,c)$.
\item Identity elements which are maps $\SphereS\to \Hom(a,a)$.
\end{itemize}
(This is a spectrum-level analogue of a linear category,
cf.~Section~\ref{sec:multi-vec}. See~\cite{BM-top-spectral} for a more
in-depth review of spectral categories.)  Associated to a
$(2m,2n)$-tangle $T$ we will construct a left-$\KTSpecCat{m}$,
right-$\KTSpecCat{n}$ bimodule $\KTSpecBim{T}$, i.e., a functor
$\KTSpecBim{T}\co(\KTSpecCat{m})^\op\times\KTSpecCat{n}\to \Spectra$.

We construct $\KTSpecCat{n}$ as follows. Let
\[
  \Hom_{\KTSpecCat{n}}(a,b)=G(a,b).
\]
Composition is defined by 
\[
\Hom_{\KTSpecCat{n}}(b,c)\smas\Hom_{\KTSpecCat{n}}(a,b)=G(b,c)\smas G(a,b)\cong G(a,b)\smas G(b,c)
 \stackrel{G(f_{a,b,c})}{\relbar\joinrel\relbar\joinrel\relbar\joinrel\longrightarrow}
  G(a,c)=\Hom_{\KTSpecCat{n}}(a,c).
\]
Identity elements are given by
\[
\SphereS\stackrel{G(f_{a})}{\relbar\joinrel\longrightarrow} G(a,a)=\Hom_{\KTSpecCat{n}}(a,a).
\]

Turning to $\KTSpecBim{T}$, let 
\[
  \KTSpecBim{T}(a,b)=G(a,T,b).
\]
On morphisms, the map is given by
\begin{align*}
  &\Hom_{(\KTSpecCat{m})^\op\times\KTSpecCat{n}}((a,b),(a',b'))\smas \KTSpecBim{T}(a,b)=G(a',a)\smas G(b,b')\smas G(a,T,b)\\
&\qquad\qquad\qquad\cong G(a',a)\smas G(a,T,b)\smas G(b,b')\stackrel{G(f_{a',a,T,b,b'})}{\relbar\joinrel\relbar\joinrel\relbar\joinrel\longrightarrow} G(a',T,b')= \KTSpecBim{T}(a',b').
\end{align*}

\begin{lemma}\label{lem:non-unital}
  These definitions make $\KTSpecCat{n}$ into a spectral category and
  $\KTSpecBim{T}$ into a $(\KTSpecCat{m},\KTSpecCat{n})$-bimodule.
\end{lemma}
\begin{proof}
  We only need to check the associativity and identity axioms, which are immediate
  from the definitions and the fact that $G$ was a multifunctor.
\end{proof}

Note that, in a similar spirit to Section~\ref{sec:multi-vec}, we can
reinterpret $\KTSpecCat{n}$ as a ring spectrum
\[
  \KTSpecRing{n}=\bigvee_{a,b\in\Ob(\KTSpecCat{n})}\Hom_{\KTSpecCat{n}}(a,b)
\]
with multiplication given by composition when defined and trivial when
composition is not defined. 
(Our ordering convention is that the product $a\cdot b$ stands for $b\circ a$.)
Similarly, $\KTSpecBim{T}$ induces an
$(\KTSpecRing{m},\KTSpecRing{n})$-bimodule spectrum
\[
  \KTSpecBimRing{T}=\bigvee_{\substack{a\in\Ob(\KTSpecCat{m})\\b\in\Ob(\KTSpecCat{n})}}\KTSpecBim{T}(a,b).
\]

Finally, we will use the following technical lemma,
to simplify the definition of the derived tensor product and topological Hochschild homology:
\begin{lemma}\label{lem:pointwise-cofibrant}
  The spectral categories $\KTSpecCat{n}$ and spectral bimodules
  $\KTSpecBim{T}$ are pointwise cofibrant. That is,
  $\Hom_{\KTSpecCat{n}}(x,y)$ and $\KTSpecBim{T}(x,y)$ are cofibrant
  symmetric spectra for all pairs of objects $x,y$.
\end{lemma}
\begin{proof}
  This is clear since the spectra are produced by rectification from
  Definition~\ref{def:rectification}, which gives a cofibrant diagram
  which is hence pointwise cofibrant
  (Lemma~\ref{lem:cofib-implies-levelwise}), and then taking homotopy
  colimits and shifting, which preserves cofibrancy
  (Lemma~\ref{lem:cofib-behaves}).
\end{proof}

\subsection{Invariance of the bimodule associated to a tangle}\label{sec:bimod-inv}
Before turning to the bimodule, consider invariance of the spectral category $\KTSpecCat{n}$.
Superficially, the functor $G\co \mHshape{n}^0\to \mSpectra$, and hence
the spectral category $\KTSpecCat{n}$, depended on a number of
choices:
\begin{enumerate}
\item The choices~(\ref{item:choice-1})--(\ref{item:choice-3}) from Section~\ref{sec:initial-invariance}.
\item\label{item:new-choice-2} Any choices in the Elmendorf-Mandell machine and the
  rectification procedure.
\end{enumerate}
As noted in Sections~\ref{sec:EM-background}
and~\ref{sec:rect-background}, Choice~(\ref{item:new-choice-2}) is, in
fact, canonical. As discussed in Section~\ref{sec:initial-invariance},
Choices~(\ref{item:choice-1})--(\ref{item:choice-3}) can be made
canonical by a colimit-type construction. So, $\KTSpecCat{n}$ is, in
fact, completely well-defined.

Turning next to $\KTSpecBim{T}$, we will show that this spectral
bimodule is well-defined up to the following equivalence:
\begin{definition}
  Given spectral categories $\Cat$ and $\Dat$ and spectral
  $(\Cat,\Dat)$-bimodules $\SModule$ and $\SNodule$, a
  \emph{homomorphism} $\mathscr{F}\co \SModule\to \SNodule$ is a
  natural transformation from $\SModule$ to $\SNodule$. A homomorphism
  is an \emph{equivalence} if for each $a\in\Ob(\Cat)$ and
  $b\in\Ob(\Dat)$, the map
  \[
    \mathscr{F}(a,b)\co \SModule(a,b)\to\SNodule(a,b)
  \]
  is an equivalence of spectra. The symmetric, transitive closure of
  this notion of equivalence is an equivalence relation; two bimodules
  are \emph{equivalent} if they are related by this equivalence
  relation (i.e., if there is a zig-zag of equivalences between them).
\end{definition}

\begin{proposition}\label{prop:equiv-gives-equiv}
  If $(F_0\co\CCat{N_0}\ttimes\mTshape{m}{n}\to\mBurnside,S_0)$ and
  $(F_1\co \CCat{N_1}\ttimes\mTshape{m}{n}\to\mBurnside,S_1)$ are
  stably equivalent functors then the induced spectral bimodules
  $\mathscr{G}_0$ and $\mathscr{G}_1$ over
  $(\KTSpecCat{m},\KTSpecCat{n})$ are equivalent.
\end{proposition}

\begin{proof}
  We first consider the case of quasi-isomorphisms. So, assume
  $N_0=N_1=N$, $S_0=S_1=S$, and $F_{01}\from
  \CCat{N+1}\ttimes\mTshape{m}{n}\to\mBurnside$ satisfies
  $F_{01}|_{\{i\}\times\CCat{N}\ttimes\mTshape{m}{n}}=F_i$ for $i=0,1$, and 
  $\Total{\Forget\circ F_{01}}$ is acyclic. Let $\Id_{F_1}$ denote the identity
  quasi-isomorphism from $F_1$ to itself, viewed as a
  multifunctor $\CCat{N+1}\ttimes\mTshape{m}{n}\to\mBurnside$.

  Consider the full subcategory
  $\{01\to 11\leftarrow 10\}\times\CCat{N}\ttimes\mTshape{m}{n}$ of
  $\CCat{2}\times\CCat{N}\ttimes\mTshape{m}{n}$. The multifunctors
  $F_{01}$ and $\Id_{F_1}$ can be patched together to produce a single
  multifunctor
  $F_{\vee}\from \{01\to 11\leftarrow
  10\}\times\CCat{N}\ttimes\mTshape{m}{n}\to\mBurnside$ which agrees
  with $F_{01}$ on $\{01\to 11\}\times\CCat{N}\ttimes\mTshape{m}{n}$
  (with the obvious identification of
  $\{01\to 11\}\times\CCat{N}\ttimes\mTshape{m}{n}$ and
  $\CCat{1}\times\CCat{N}\ttimes\mTshape{m}{n}$) and agrees
  with $\Id_{F_1}$ on
  $\{11\leftarrow 10\}\times\CCat{N}\ttimes\mTshape{m}{n}$.

  We now apply the construction from
  Section~\ref{sec:build-spec-bim} to this functor. Composing with
  the functor $\mBurnside\to\PermuCat\to\mSpectra$ and rectifying, we
  get a functor
  \[
    G_{\vee}\from (\{01\to 11\leftarrow
  10\}\times\CCat{N}\ttimes\mTshapeStct{m}{n})^0\to\mSpectra
  \]
  and since $\{01\}\times\CCat{N}\ttimes\mTshape{m}{n}$ and
  $\{10\}\times\CCat{N}\ttimes\mTshape{m}{n}$ are blockaded
  subcategories, the restrictions of $G_{\vee}$ to
  $(\{01\}\times\CCat{N}\ttimes\mTshapeStct{m}{n})^0$ and
  $(\{10\}\times\CCat{N}\ttimes\mTshapeStct{m}{n})^0$ agree with $G_0$
  and $G_1$, the rectifications of the compositions of $F_0$
  and $F_1$ with the functor $\mBurnside\to\PermuCat\to\mSpectra$.
  
  Let $H_0$ and $H_1$ be the result of applying the mapping
  cone construction from Section~\ref{sec:build-spec-bim} to $G_0$ and
  $G_1$, and let $\mathscr{K}_0$ and $\mathscr{K}_1$ be the spectral
  bimodules obtained from $H_0$ and $H_1$ by shifting by
  $S$. Let $H_\vee$ be the result of applying the same mapping
  cone construction to $G_\vee$, and let $\mathscr{K}_\vee$ be the
  spectral bimodule obtained by shifting by $S+1$. Finally, let
  $H_\to$ and $H_\leftarrow$ be the results of the mapping
  cone construction applied to $G_\vee$ restricted to
  $(\{01\to 11\}\times\CCat{N}\ttimes\mTshapeStct{m}{n})^0$ and
  $(\{11\leftarrow 10\}\times\CCat{N}\ttimes\mTshapeStct{m}{n})^0$,
  respectively.

  It is clear from the mapping cone construction that for each $a,b$,
  there are cofibration sequences
  \begin{align*}
    &\cdots\to H_\leftarrow(a,b)\to H_\vee(a,b)\to\Sigma H_0(a,b)\to\cdots\\
    &\cdots\to H_\to(a,b)\to H_\vee(a,b)\to\Sigma H_1(a,b)\to\cdots
  \end{align*}
  and these maps are natural with respect to morphisms in $(\KTSpecCat{m})^\op\times\KTSpecCat{n}$. Moreover,
  $H_\leftarrow(a,b)$ and $H_\to(a,b)$ are contractible since
  $\Total{\Forget\circ \Id_{F_1}}$ and $\Total{\Forget\circ F_{01}}$
  are acyclic. Therefore, for each $i=0,1$, the map
  $H_\vee\to\Sigma H_i$ is an equivalence of spectral
  bimodules. Shifting by $S+1$, we get that the map
  $\mathscr{K}_\vee\to\sh^{-S-1}\Sigma{H}_i$ is an equivalence as
  well; moreover, we also have an equivalence
  $\sh^{-S-1}\Sigma{H}_i\to \sh^{-S-1}\sh
  H_i\to\sh^{-S}H_i=\mathscr{K}_i$
  (cf.~Proposition~\ref{prop:sh-adjoints}).

  For stabilizations $i_! F$ of
  $(F\from\CCat{N}\ttimes\mTshape{m}{n}\to\mBurnside,S)$, it is enough
  to consider the two face inclusions of $\CCat{N}\into\CCat{N+1}$ as
  $\{0\}\times\CCat{N}$ and as $\{1\}\times\CCat{N}$ (since any
  arbitrary face inclusion is a composition of such face
  inclusions and permutations of the factors of $\CCat{N}$, and
  invariance under permutations is clear). In each case, let $G$ and $i_!G$ be the corresponding
  rectified functors to $\mSpectra$, $H$ and $i_!H$ the results after
  applying the mapping cone constructions, and $\mathscr{K}$ and
  $i_!\mathscr{K}$ the corresponding spectral bimodules after shifting
  by $S$ and $S+1-|i|$, respectively.

  In the first case, $\{0\}\times\CCat{N}\ttimes\mTshape{m}{n}$ is a
  blockaded subcategory, so the restriction of $i_!G$ to
  $(\{0\}\times\CCat{N}\ttimes\mTshapeStct{m}{n})^0$ agrees with $G$, and
  from the mapping cone construction, we have an equivalence
  $i_!H\to\Sigma H$ of spectral bimodules. As before, after shifting
  by $S+1$, we get an equivalence
  $i_!\mathscr{K}\to\sh^{-S-1}\Sigma H\to\mathscr{K}$ as well.

  In the second case, $\{1\}\times\CCat{N}\ttimes\mTshape{m}{n}$ is
  not a blockaded subcategory, so the previous proof does not quite
  work. Nevertheless, we can proceed as in the case of
  quasi-isomorphisms.

  Let $\Id_{F}$ be the identity quasi-isomorphism from $F$ to itself,
  viewed as a multifunctor
  $\CCat{N+1}\ttimes\mTshape{m}{n}\to\mBurnside$.  Consider the full
  subcategory of $\CCat{3}\times\CCat{N}\ttimes\mTshape{m}{n}$ spanned
  by $100,110,101,011,111\in\CCat{3}$. Two copies of the multifunctors
  $\Id_F$ can be patched together to produce a single multifunctor
  $F_{\text{big}}$ from this category to $\mBurnside$ which agrees with
  $\Id_{F}$ on $\{110\to 111\}\times\CCat{N}\ttimes\mTshape{m}{n}$ and
  $\{011\to 111\}\times\CCat{N}\ttimes\mTshape{m}{n}$, and is $0$ on
  the rest; schematically, the functor looks like:
  \[
  \begin{tikzpicture}[yscale=-1]
    \begin{scope}
    \foreach \m/\n/\i/\j in {0/100/0/0,F/110/0/1,0/101/1/1,F/011/2/1,F/111/1/2}{
      \node (\n) at (\i,\j) {\n};
    }

    \draw[->] (100) -- (110);
    \draw[->] (100) -- (101);
    \draw[->] (110) -- (111);
    \draw[->] (101) -- (111);
    \draw[->] (011) -- (111);
    \end{scope}

    \begin{scope}[xshift=2in]
    \foreach \m/\n/\i/\j in {0/100/0/0,F/110/0/1,0/101/1/1,F/011/2/1,F/111/1/2}{
      \node (\n) at (\i,\j) {$\m$};
    }

    \draw[->] (100) -- (110);
    \draw[->] (100) -- (101);
    \draw[->] (110) -- (111) node[midway,anchor=north east] {\tiny $\Id_F$};
    \draw[->] (101) -- (111);
    \draw[->] (011) -- (111) node[midway,anchor=north west] {\tiny $\Id_F$};
    \end{scope}

  \end{tikzpicture}
  \]
  Let $G_{\text{big}}$ be the corresponding rectified functor,
  $H_{\text{big}}$ the result after applying the mapping cone
  construction, and $\mathscr{K}_{\text{big}}$ the result after
  shifting by $S+1$.  Once again, since $\{100\to
  110\}\times\CCat{N}\ttimes\mTshape{m}{n}$ and
  $\{011\}\times\CCat{N}\ttimes\mTshape{m}{n}$ are blockaded
  subcategories, the restrictions of $G_{\text{big}}$ to $(\{100\to
  110\}\times\CCat{N}\ttimes\mTshapeStct{m}{n})^0$ and
  $(\{011\}\times\CCat{N}\ttimes\mTshapeStct{m}{n})^0$ agree with $i_!G$
  and $G$. 

  Let $H_{\vee}$ and $H_{\lozenge}$ be the results of the mapping cone
  construction applied to $G_{\text{big}}$ restricted to the full
  subcategories generated by $101,011,111\in\CCat{3}$ and
  $100,110,101,111\in\CCat{3}$, respectively. As before, there are
  natural cofibration sequences
  \begin{align*}
    &\cdots\to H_{\lozenge}(a,b)\to H_{\text{big}}(a,b)\to\Sigma H(a,b)\to\cdots\\
    &\cdots\to H_{\vee}(a,b)\to H_{\text{big}}(a,b)\to\Sigma i_!H(a,b)\to\cdots
  \end{align*}
  and moreover, $H_{\lozenge}(a,b)$ and $H_\vee(a,b)$ are contractible
  for each $a,b$. Therefore, the maps $H_{\text{big}}\to\Sigma H$ and
  $H_{\text{big}}\to\Sigma i_!H$ are equivalences of spectral
  bimodules. Shifting by $S+1$, the maps
  $\mathscr{K}_{\text{big}}\to\sh^{-S-1}\Sigma{H}\to\mathscr{K}$ and
  $\mathscr{K}_{\text{big}}\to\sh^{-S-1}\Sigma{i_!H}\to i_!\mathscr{K}$
  are equivalences as well.
\end{proof}

\begin{theorem}\label{thm:spec-invariance}
  Up to equivalence of $(\KTSpecCat{m},\KTSpecCat{n})$-bimodules,
  $\KTSpecBim{T}$ is an invariant of the isotopy class of the
  $(2m,2n)$-tangle $T$. Further, the map on homology induced by a
  sequence of Reidemeister moves agree, up to a sign, with Khovanov's
  invariance map~\cite[Section 4]{Kho-kh-tangles}.
\end{theorem}
\begin{proof}
  This is immediate from Theorem~\ref{thm:comb-inv} and
  Proposition~\ref{prop:equiv-gives-equiv}.
\end{proof}

\section{Gluing}\label{sec:gluing}
In this section we prove that gluing tangles corresponds to the
derived tensor product of spectral bimodules
(Theorem~\ref{thm:gluing}). We start by introducing one more shape
multicategory, adapted to studying triples of tangles
$(T_1,T_2,T_1T_2)$. We then recall the tensor product of spectral
bimodules and, with these tools in hand, prove the gluing theorem.

Fix non-negative integers $m,n,p$. 
The \emph{gluing multicategory} $\npstrictify{\mGlue{m}{n}{p}}$,
which is the shape multicategory associated to
$(\Crossingless{m},\Crossingless{n},\Crossingless{p})$
(cf.~Definition~\ref{def:shape-multicat-set-seq}). Explicitly, $\npstrictify{\mGlue{m}{n}{p}}$ has
objects:
\begin{itemize}
\item Pairs $(a_1,a_2)$ of crossingless matchings on $2m$ points.
\item Pairs $(b_1,b_2)$ of crossingless matchings on $2n$ points.
\item Pairs $(c_1,c_2)$ of crossingless matchings on $2p$ points.
\item Triples $(a,T_1,b)$ where $a$ is a crossingless matching of $2m$
  points, $b$ is a crossingless matching of $2n$ points, and $T_1$ is
  a placeholder (a mnemonic for a $(2m,2n)$ tangle).
\item Triples $(b,T_2,c)$ where $b$ is a crossingless matching of $2n$
  points, $c$ is a crossingless matching of $2p$ points, and $T_2$ is
  a placeholder (a mnemonic for a $(2n,2p)$ flat tangle).
\item Triples $(a,T_1T_2,c)$ where $a$ is a crossingless matching of
  $2m$ points, $c$ is a crossingless matching of $2p$ points, and
  $T_1T_2$ is a placeholder (a mnemonic for the composition of $T_1$
  and $T_2$).
\end{itemize}
So, the objects of $\npstrictify{\mTshape{m}{n}}$, $\npstrictify{\mTshape{n}{p}}$, and
$\npstrictify{\mTshape{m}{p}}$ are contained in the gluing multicategory, and in
fact we let these three multicategories be full subcategories of the
gluing multicategory. There is one more kind of multimorphism in the
gluing multicategory: a unique multimorphism
\[
(a_1,a_2),\dots,(a_{i-1},a_i),(a_i,T_1,b_1),(b_1,b_2),\dots,(b_{j-1},b_j),(b_j,T_2,c_1),(c_1,c_2),\dots,(c_{k-1},c_k)\to (a_1,T_1T_2,c_k)
\]
where the $a_\ell$ (respectively $b_\ell$, $c_\ell$) are crossingless
matchings of $2m$ (respectively $2n$, $2p$) points. Let
$\mGlue{m}{n}{p}$ be the canonical groupoid enrichment of
$\npstrictify{\mGlue{m}{n}{p}}$.

Next we define a category
$\CCat{N_1|N_2}\ttimes\mGlue{m}{n}{p}$
similar to (and extending)
$\CCat{N}\ttimes\mTshape{m}{n}$. The objects of
$\CCat{N}\ttimes\mGlue{m}{n}{p}$ are of the following forms:
\begin{itemize}
\item Pairs $(x,y)$ in $\Ob(\mHshape{m})$ or $\Ob(\mHshape{n})$ or
  $\Ob(\mHshape{p})$.
\item Quadruples $(v,a,T_1,b)$ where $v\in\Ob(\CCat{N_1})$,
  $a\in\Crossingless{m}$, and $b\in\Crossingless{n}$.
\item Quadruples $(v,b,T_2,c)$ where $v\in\Ob(\CCat{N_2})$,
  $b\in\Crossingless{n}$, and $c\in\Crossingless{p}$.
\item Quadruples $(v,a,T_1T_2,c)$ where $v\in\Ob(\CCat{N_1+N_2})$,
  $a\in\Crossingless{m}$, and $c\in\Crossingless{p}$.
\end{itemize}
So,
\[
  \Ob(\CCat{N_1|N_2}\ttimes\mGlue{m}{n}{p})=
  \Ob(\CCat{N_1}\ttimes\mTshape{m}{n})\cup 
  \Ob(\CCat{N_2}\ttimes\mTshape{n}{p})\cup
  \Ob(\CCat{N_1+N_2}\ttimes\mTshape{m}{p}).
\]

A \emph{basic multimorphism} for
$\CCat{N_1|N_2}\ttimes\mGlue{m}{n}{p}$ is one of:
\begin{itemize}
\item A basic multimorphism in $\CCat{N_1}\ttimes\mTshape{m}{n}$,
  $\CCat{N_2}\ttimes\mTshape{n}{p}$, or
  $\CCat{N_1+N_2}\ttimes\mTshape{m}{p}$, or
\item A (unique) multimorphism
  \begin{multline*}
    (a_1,a_2),\dots,(a_{j-1},a_{j}),(v,a_j,T_1,b_1),(b_1,b_2),\dots,(b_{k-1},b_k),(w,b_k,T_2,c_1),(c_1,c_2)\dots,(c_{\ell-1},c_\ell)\\
    \to ((v,w),a_1,T_1T_2,c_\ell).
  \end{multline*}
\end{itemize}
The multimorphisms in $\CCat{N_1|N_2}\ttimes\mGlue{m}{n}{p}$ are
planar, rooted trees whose edges are decorated by objects in
$\CCat{N_1|N_2}\ttimes\mGlue{m}{n}{p}$ and whose vertices are
decorated by basic multimorphisms compatible with the decorations on
the edges. If two multimorphisms have the same source and target then
we declare that there is a unique morphism in the corresponding
multimorphism groupoid between them.

Let $\strictify{\CCat{N_1|N_2}\ttimes\mGlue{m}{n}{p}}$ be the
strictification of $\CCat{N_1|N_2}\ttimes\mGlue{m}{n}{p}$. We have the
following analogue of Lemma~\ref{lem:strictify-mTshape}:
\begin{lemma}
  The projection
  $\CCat{N_1|N_2}\ttimes\mGlue{m}{n}{p}\to
  \strictify{\CCat{N_1|N_2}\ttimes\mGlue{m}{n}{p}}$ is a weak
  equivalence.
\end{lemma}
\begin{proof}
  The proof is essentially the same as the proofs of
  Lemmas~\ref{lem:unenrich} and~\ref{lem:strictify-mTshape}.
\end{proof}

Fix a $(2m,2n)$-tangle $T_1$ with $N_1$ crossings and a $(2n,2p)$-tangle $T_2$
with $N_2$ crossings and let $T_1T_2$ denote the composition of $T_1$ and
$T_2$. Choose enough pox on $T_1$ and $T_2$ so that $T_1T_2$ is a poxed tangle
(Definition~\ref{def:pox}). Then we have multifunctors
$\mTfuncNF{T_1}\co \CCat{N_1}\ttimes \mTshape{m}{n}\to \mCobD$,
$\mTfuncNF{T_2}\co \CCat{N_2}\ttimes \mTshape{n}{p}\to \mCobD$, and
$\mTfuncNF{T_1T_2}\co \CCat{N_1+N_2}\ttimes \mTshape{m}{p}\to \mCobD$.
\begin{lemma}\label{lem:glue-multifunc-exists}
  There is a multifunctor
  $\mGlFunc\co \CCat{N_1|N_2}\ttimes \mGlue{m}{n}{p}\to \mCobD$
  extending $\mTfuncNF{T_1}$, $\mTfuncNF{T_2}$, and
  $\mTfuncNF{T_1T_2}$, and so that for any $a\in\Crossingless{m}$,
  $b\in\Crossingless{n}$, $c\in\Crossingless{p}$, and $(v,w)\in\CCat{N_1|N_2}$,
  $G\bigl((v,a,T_1,b),(w,b,T_2,c)\to((v,w),a,T_1T_2,c)\bigr)$ is a
  multi-merge cobordism (connecting $\Wmirror{b}b$ to the identity).
 \end{lemma}
\begin{proof}
  This is a straightforward adaptation of the construction of
  $\mTfuncNF{T}$, and is left to the reader.
\end{proof}
Composing $\mGlFunc$ with the Khovanov-Burnside functor gives a
functor $\mV\circ \mGlFunc\co \mGlue{m}{n}{p}\to
\mBurnside$. Proceeding as in the construction of the tangle
invariants in Section~\ref{sec:build-spec-bim} we obtain a functor
\[
  \mathit{Gl}\co \strictify{\mGlue{m}{n}{p}}\to\mSpectra.
\]
The functor $\mathit{Gl}$ restricts to $G_{T_1}$ on
$\mTshapeStct{m}{n}^0$ and $G_{T_2}$ on $\mTshapeStct{n}{p}^0$. (This uses the
fact that $\mTshape{m}{n}$ and $\mTshape{n}{p}$ are blockaded
subcategories of $\mGlue{m}{n}{p}$ and
Lemma~\ref{lem:rectify-restrict}.) By Lemma~\ref{lem:rect-equiv}, on
$\mTshapeStct{m}{p}^0$, the functor $\mathit{Gl}$ is naturally equivalent
to $G_{T_1T_2}$, but because of the rectification step, may not agree
with $G_{T_1T_2}$ exactly. Since there are no morphisms out of the
subcategory $\mTshapeStct{m}{p}^0$, we can compose $\mathit{Gl}$ with the
equivalence from $\mathit{Gl}|_{\mTshapeStct{m}{p}^0}$ to $G_{T_1T_2}$ to
obtain a new functor whose restriction to
$\mathit{Gl}|_{\mTshapeStct{m}{p}^0}$ agrees with $G_{T_1T_2}$. Abusing
notation, from now on we use $\mathit{Gl}$ to denote this new
functor.

We recall two notions of tensor product of modules over a spectral
category:
\begin{definition}
  Let $\Cat$, $\Dat$, and $\Eat$ be spectral categories, $\SModule$ a
  $(\Cat,\Dat)$-bimodule and $\SNodule$ a
  $(\Dat,\Eat)$-bimodule. Assume that $\Dat$, $\SModule$ and
  $\SNodule$ are pointwise cofibrant
  (cf.~Lemma~\ref{lem:pointwise-cofibrant}). The
  \emph{tensor product} of $\SModule$ and $\SNodule$ over $\Dat$,
  $\SModule\otimes_{\Dat}\SNodule$, is the $(\Cat,\Eat)$-bimodule $P$ where $P(a,c)$
  is the coequalizer of the diagram
  \[
      \coprod_{b,b'\in\Ob(\Dat)}\SModule(a,b)\smas \Hom_{\Dat}(b,b')\smas \SNodule(b',c)
      \rightrightarrows
      \coprod_{b\in\Ob(\Dat)}\SModule(a,b)\smas \SNodule(b,c).
  \]
  (Here, the two maps correspond to the action of $\Hom(b,b')$ on
  $\SModule(a,b)$ and on $\SNodule(b',c)$, respectively.)

  The \emph{derived tensor product} of $\SModule$ and $\SNodule$ over $\Dat$,
  $\SModule\DTP_{\Dat}\SNodule$, is
  \begin{align*}
    P(a,c)=\hocolim\Bigl(
    \cdots
    \mathrel{\substack{\textstyle\rightarrow\\[-0.5ex]
    \textstyle\rightarrow \\[-0.5ex]
    \textstyle\rightarrow\\[-0.5ex]
    \textstyle\rightarrow}}
      &\coprod_{b,b',b''\in\Ob(\Dat)}\SModule(a,b)\smas \Hom_{\Dat}(b,b')\smas \Hom_{\Dat}(b',b'')\smas \SNodule(b'',c)\\
    \mathrel{\substack{\textstyle\rightarrow\\[-0.5ex]
    \textstyle\rightarrow \\[-0.5ex]
    \textstyle\rightarrow}}
      &\coprod_{b,b'\in\Ob(\Dat)}\SModule(a,b)\smas \Hom_{\Dat}(b,b')\smas \SNodule(b',c)\\
      \rightrightarrows
      &\coprod_{b\in\Ob(\Dat)}\SModule(a,b)\smas \SNodule(b,c)
    \Bigr).
  \end{align*}
  
  There is an evident quotient map
  $\SModule\DTP_{\Dat}\SNodule\to \SModule\otimes_{\Dat} \SNodule$.
\end{definition}

The derived tensor product is functorial and preserves equivalences
in the following sense. Given a map $\Dat \to \Dat'$, modules
$\SModule$ and $\SNodule$ over $\Dat$, modules $\SModule'$ and
$\SNodule'$ over $\Dat'$, and maps $\SModule \to \SModule'$ and
$\SNodule \to \SNodule'$ intertwining the actions of $\Dat$ and $\Dat'$, there is a map
\[
\SModule\DTP_{\Dat}\SNodule \to \SModule'\DTP_{\Dat'}\SNodule'.
\]
If the maps $\Dat \to \Dat'$, $\SModule \to \SModule'$, and $\SNodule \to
\SNodule'$ are equivalences this map of derived tensor products is an
equivalence.

Replacing smash products with tensor products gives the derived tensor product of chain complexes (assuming that the constituent complexes are all flat over $\ZZ$). Again, the derived tensor product is functorial and preserves quasi-isomorphisms of complexes.

Reinterpreting $\mathit{Gl}$, for each triple of crossingless
matchings $a,b,c$ we have a map
\[
  \mathit{Gl}((a,T_1,b),(b,T_2,c)\to(a,T_1T_2,c))\co G(a,T_1,b)\smas G(b,T_2,c)\to G(a,T_1T_2,c).
\]
\begin{lemma}
  The map $\mathit{Gl}$ induces a map of bimodules
  $\KTSpecBim{T_1}\otimes_{\KTSpecCat{n}}\KTSpecBim{T_2}\to \KTSpecBim{T_1T_2}$.
\end{lemma}
\begin{proof}
  By definition,
  \[
    \bigl(\KTSpecBim{T_1}\otimes_{\KTSpecCat{n}}\KTSpecBim{T_2}\bigr)(a,c)
    =\coprod_{b\in\Crossingless{n}}G_{T_1}(a,T,b)\smas G_{T_2}(b,T,c)/\sim.
  \]
  The map $\mathit{Gl}$ gives maps
  \[
\xymatrix@C=30ex{
    \coprod_{b\in\Crossingless{n}}G_{T_1}(a,T,b)\smas G_{T_2}(b,T,c)\ar[r]^-{\coprod_b\mathit{Gl}((a,T_1,b),(b,T_2,c)\to(a,T_1T_2,c))}& G_{T_1T_2}(a,T,c).}
  \]
  We must check that these maps respect the equivalence relation
  $\sim$ and the actions of $\KTSpecCat{m}$ and $\KTSpecCat{p}$; but
  both statements are immediate from the fact that the map
  $\mathit{Gl}$ is a multifunctor (and the definition of
  $\mGlue{m}{n}{p}^0$).
\end{proof}

Composing with the quotient map
$\KTSpecBim{T_1}\DTP_{\KTSpecCat{n}}\KTSpecBim{T_2}\to
\KTSpecBim{T_1}\otimes_{\KTSpecCat{n}}\KTSpecBim{T_2}$ gives a map
$\KTSpecBim{T_1}\DTP_{\KTSpecCat{n}}\KTSpecBim{T_2}\to\KTSpecBim{T_1T_2}$.

We recall a fact about the classical Khovanov bimodules:
\begin{lemma}\label{lem:Khovanov-DTP}
  If $T$ is an $(2m,2n)$ flat tangle then the bimodule $\KTfunc(T)$ is
  left-projective and right-projective. So, given a $(2m,2n)$-tangle
  $T_1$ and a $(2n,2p)$-tangle $T_2$ there are quasi-isomorphisms
  \[
    \KTfunc(T_1)\DTP_{\KTalg{n}}\KTfunc(T_2)\simeq
    \KTfunc(T_1)\otimes_{\KTalg{n}}\KTfunc(T_2)\simeq \KTfunc(T_1T_2).
  \]
  Further, the second quasi-isomorphism is induced by the evident multi-merge
  cobordisms.
\end{lemma}
\begin{proof}
  Khovanov proved that the bimodules associated to flat tangles are
  left and right projective; he used the word \emph{sweet} for
  finitely-generated bimodules with this property~\cite[Proposition
  3]{Kho-kh-tangles}.  So, the first isomorphism follows from the
  definition of the derived tensor product and sweetness. The second
  isomorphism is Khovanov's gluing theorem (repeated above as
  Proposition~\ref{prop:Kh-pairing}); his proof also shows that it
  comes from the multi-merge cobordisms.
\end{proof}

\begin{lemma}\label{lem:gluing-commutes}
  Given a $(2m,2n)$-tangle $T_1$ and a $(2n,2p)$-tangle $T_2$, there is
  a commutative diagram of isomorphisms in the derived category
  of complexes
  \[
  \xymatrix{
    C_*\bigl(\KTSpecBim{T_1}\DTP_{\KTSpecCat{n}}\KTSpecBim{T_2}\bigr)
    \ar[dr]_{\mathit{Gl}} &
    C_*(\KTSpecBim{T_1})\DTP_{C_*(\KTSpecCat{n})}C_*(\KTSpecBim{T_2})
    \ar[r] \ar[d] \ar[l]&
    \KTfunc(T_1)\DTP_{\KTalg{n}}\KTfunc(T_2) \ar[d] \\
    &C_*\bigl(\KTSpecBim{T_1T_2}\bigr) \ar[r] &
    \KTfunc(T_1T_2),
  }
  \]
  where the right-hand horizontal arrows are induced by the
  quasi-isomorphisms of Proposition~\ref{prop:homology-right} and the
  right-most vertical arrow is the quasi-isomorphism from
  Lemma~\ref{lem:Khovanov-DTP}.
\end{lemma}
\begin{proof}
  We begin by applying $C_*$ to the diagram defining the derived
  tensor product $\KTSpecBim{T_1}\DTP_{\KTSpecCat{n}}\KTSpecBim{T_2}$.
  Using both the natural quasi-isomorphism $\hocolim C_* \to C_*
  \hocolim$ and monoidality of $C_*$, we get the quasi-isomorphism
  \[
    C_*(\KTSpecBim{T_1})\DTP_{C_*(\KTSpecCat{n})}C_*(\KTSpecBim{T_2})
    \to C_*\bigl(\KTSpecBim{T_1}\DTP_{\KTSpecCat{n}}\KTSpecBim{T_2}\bigr).
  \]
  Define the map
  $C_*(\KTSpecBim{T_1})\DTP_{C_*(\KTSpecCat{n})}C_*(\KTSpecBim{T_2})
  \to C_*(\KTSpecBim{T_1 T_2})$ to be the composition of this
  quasi-isomorphism and the map on chains induced by the gluing map $\mathit{Gl}$.

  We now address the right-hand square. Recall that
  Lemma~\ref{lem:connective-zigzag} constructs natural transformations of multifunctors $\mSpectra
  \to \mComplexes$
  \[
    C_* \leftarrow \tau_{\geq 0} \circ C_* \to H_0,
  \]
  where the left-hand arrow is always an isomorphism in
  non-negative homology degrees and the right-hand one is always
  an isomorphism in homology degree zero. In particular, this gives us
  natural quasi-isomorphisms of dg-categories
  \[
  C_* \KTSpecCat{n} \leftarrow \tau_{\geq 0} C_* \KTSpecCat{n} \to H_0
  \KTSpecCat{n},
  \]
  where the right-hand term is Khovanov's arc algebra
  $\KTalg{n}$. Similarly, we can apply these truncation transformations
  to the spectral bimodule $\KTSpecBim{T}$, obtaining quasi-isomorphisms
  \begin{align*}
    C_* \KTSpecBim{T} &= C_* \bigl(\sh^{-N_+} \hocolim_{\CCat{N}_+}
    F|_{(a,T,b)}^+\bigr)\\
    &\leftarrow \hocolim_{\CCat{N}_+} C_*(F|_{(a,T,b)}^+)[-N_+]\\
    &\leftarrow \hocolim_{\CCat{N}_+} \bigl(\tau_{\geq 0}
    C_*(F|_{(a,T,b)}^+)\bigr)[-N_+]\\
    &\to \hocolim_{\CCat{N}_+} \bigl(H_0 \circ
    F|_{(a,T,b)}^+ )\bigr)[-N_+]\\
    &= \KTfunc(T).
  \end{align*}
  These maps are compatible with bimodule structures: all terms are
  bimodules over $(\tau_{\geq 0} C_* \KTSpecCat{m},\tau_{\geq 0} C_*
  \KTSpecCat{n})$, and these bimodule structures are compatible with the structure of a
  bimodule over the untruncated chain complex $(C_*\KTSpecCat{m},C_*\KTSpecCat{n})$ on $C_* \KTSpecBim{T}$ and of a
  bimodule over the arc algebras $(\KTalg{m},\KTalg{n})$ on $\KTfunc(T)$.

  Let
  \[
    D_* (T)=\hocolim_{\CCat{N}_+} \bigl(\tau_{\geq 0}
    C_*(F|_{(a,T,b)}^+)\bigr)[-N_+].
  \]
  We now apply derived tensor products and the gluing pairing $\mathit{Gl}$,
  obtaining a diagram
  \[
  \xymatrix{
    C_* \KTSpecBim{T_1} \DTP_{C_* \KTSpecCat{n}} C_* \KTSpecBim{T_2} \ar[d]
    & D_*(T_1) \DTP_{\tau_{\geq 0} C_* \KTSpecCat{n}} D_*(T_2)
    \ar[d] \ar[l] \ar[r]
    & \KTfunc(T_1) \DTP_{\KTalg{n}} \KTfunc(T_2) \ar[d]\\
    C_* \KTSpecBim{T_1 T_2}
    & D_*(T_1 T_2) \ar[l] \ar[r]
    & \KTfunc(T_1 T_2).
  }
  \]
  As we just showed, the bottom horizontal maps are quasi-isomorphisms. Since
  the derived tensor product preserves homotopy colimits, the top
  horizontal maps are also quasi-isomorphisms. It follows from
  compatibility of the maps with the bimodule structures that both
  squares commute, where the right-most arrow is the map induced by
  the evident multi-merge cobordisms.
  Lemma~\ref{lem:Khovanov-DTP} implies that this is exactly Khovanov's gluing
  quasi-isomorphism.
\end{proof}

\begin{theorem}\label{thm:gluing}
  The gluing functor
  $\KTSpecBim{T_1}\DTP_{\KTSpecCat{n}}\KTSpecBim{T_2}\to\KTSpecBim{T_1T_2}$
  is an equivalence of bimodules.
\end{theorem}

\begin{proof}
  Lemma~\ref{lem:gluing-commutes} shows that the induced map of chain complexes agrees with
  the map $\KTfunc(T_1) \DTP_{\KTalg{n}} \KTfunc(T_2) \to \KTfunc (T_1
  T_2)$, which is a quasi-isomorphism. As the spectra in question are
  connective, the result follows from the homology Whitehead theorem
  (Theorem~\ref{citethm:homologywhitehead}).
\end{proof}

\section{Quantum gradings}\label{sec:q-gr}
So far, we have suppressed the quantum gradings; in this section we
reintroduce them.

\begin{definition}
  The \emph{grading multicategory} $\GrCat$ has:
  \begin{itemize}
  \item One object for each integer $n$, and
  \item A unique multimorphism $(m_1,\dots,m_k)\to m_1+\dots+m_k$ for
    each $m_1,\dots,m_k\in \ZZ$.
  \end{itemize}
\end{definition}
As usual, we can view the grading multicategory as trivially enriched
in groupoids.

\begin{definition}
  The \emph{naive product} of multicategories $\Cat$ and $\Dat$,
  $\Cat\times\Dat$, has objects pairs
  $(c,d)\in\Ob(\Cat)\times\Ob(\Dat)$, multimorphism sets
  \[
    \Hom_{\Cat\times\Dat}((c_1,d_1),\dots,(c_n,d_n);(c,d))=\Hom_{\Cat}(c_1,\dots,c_n;c)\times\Hom_{\Dat}(d_1,\dots,d_n;d),
  \]
  and the obvious composition and identity maps.
\end{definition}

Given a multicategory $\Cat$ and a multifunctor
$F\co\GrCat\times\Cat\to\mBurnside$ satisfying
  \begin{itemize}
  \item[(F)] for all objects $x\in\Ob(\Cat)$, $F(n,x)$ is empty for
    all but finitely many $n$,
  \end{itemize}
 there is an associated multifunctor
$\sint F\co \Cat\to\Dat$ defined by
\[
  (\sint F)(x)=\coprod_{n\in\ZZ}F(n,x)
\]
and, given $f\in\Hom_{\Cat}(x_1,\dots,x_k;y)$, the correspondence 
\begin{align*}
(\sint F)(f)\co& \left(\coprod_{m_1\in\ZZ}F(m_1,x_1)\right)\times\cdots\times \left(\coprod_{m_k\in\ZZ}F(m_k,x_k)\right)\\
&\qquad\qquad=\coprod_{(m_1,\dots,m_k)\in\ZZ^k}F(m_1,x_1)\times\cdots\times F(m_k,x_k)\longrightarrow 
\coprod_{n\in\ZZ}F(n,y)
\end{align*}
satisfies
\[
  s^{-1}\left(F(m_1,x_1)\times\cdots\times F(m_k,x_k)\right)\cap t^{-1}\left(F(n,y)\right)=
  \begin{cases}
    F(((m_1,\dots,m_k)\to n)\times f) &\text{if $n=m_1+\dots+m_k$}\\
    \emptyset & \text{otherwise}.
  \end{cases}
\]

We will lift the functors
$\mHinv{m}\co \mHshape{m}\to\mBurnside$ and
$\mTinvNF{T}\co \CCat{N}\ttimes\mTshape{m}{n}\to\mBurnside$ to functors
\begin{align*}
  \GmHinv{m}&\co \GrCat\times \mHshape{m}\to\mBurnside\\
  \GmTinvNF{T}&\co \GrCat\times(\CCat{N}\ttimes\mTshape{m}{n})\to\mBurnside.
\end{align*}
By ``lift'' we mean that there are natural isomorphisms
\begin{equation}\label{eq:gr-lift}
  \sint\GmHinv{m}\cong \mHinv{m} \qquad\qquad \sint\GmTinvNF{T}\cong \mTinvNF{T}.
\end{equation}

We start by defining the lifts at the level of objects, by copying
Khovanov's definitions of the quantum gradings on the arc algebras and
modules. Specifically, given an object $(a,b)\in \Ob(\mHshape{m})$ and
an element $x\in \mHinv{m}(a,b)$ which labels $p(x)$ circles by
$1$ and $n(x)$ circles by $X$, we define the \emph{quantum
  grading}
\begin{equation}\label{eq:q-gr-alg}
  \gr_q(x)=n(x)-p(x)+m
\end{equation}
and let
\[
  \GmHinv{m}(k,(a,b))=\{x\in\mHinv{m}(a,b)\mid \gr_q(x)=k\}.
\]
Similarly, for $(v,a,T,b)\in\Ob(\CCat{N}\ttimes\mTshape{m}{n})$ and
$x\in\mTinvNF{T}(v,a,T,b)$ which labels $p(x)$ circles by $1$ and $n(x)$ circles by $X$ we define
\begin{equation}\label{eq:q-gr-tangle}
  \gr_q(x)=n(x)-p(x)+n-|v|,
\end{equation}
where $|v|$ is the number of $1$s in $v$, and let
\[
  \GmTinvNF{T}(k,(v,a,T,b))=\{x\in\mTinvNF{T}(v,a,T,b)\mid \gr_q(x)=k\}.
\]

\begin{example}
  For $(a,a)\in\Ob(\mHshape{m})$, the quantum grading of an element
  $x\in\mHinv{m}(a,b)$ is $2$ times the number of circles labeled
  $X$, and in particular ranges between $0$ and $2m$. The unit
  element, in which all circles are labeled $1$, is in
  quantum grading $0$.
\end{example}

\begin{lemma}
  These definitions of $\GmHinv{m}$ and $\GmTinvNF{T}$
  extend uniquely to the morphism groupoids of $\GmHinv{m}$
  and $\GmTinvNF{T}$ satisfying Equations~\eqref{eq:gr-lift}.
\end{lemma}
\begin{proof}
  Uniqueness is clear. Existence follows from the fact that the
  multiplication on the Khovanov arc algebras and bimodules respects
  the quantum gradings.
\end{proof}

Using $\GmHinv{m}$ and $\GmTinvNF{T}$ in place of
$\mHinv{m}$ and $\mTinvNF{T}$ in
Section~\ref{sec:build-spec-bim} gives functors
\[
  \GG\co \GrCat\times\mHshape{n}^0\to \mSpectra\qquad\qquad\text{and}\qquad\qquad
  \GG\co \GrCat\times\mTshapeStct{m}{n}^0\to \mSpectra.
\]
These give a graded spectral category $\GKTSpecCat{n}$ and graded
$(\GKTSpecCat{m},\GKTSpecCat{n})$-bimodule
$\GKTSpecBim{T}$, with same objects, by setting
\begin{align*}
  \Hom_{\GKTSpecCat{n}}(a,b)_k&=\GG(k,(a,b))\\
    \GKTSpecBim{T}(a,b)_k&=\GG(k,(a,T,b))
\end{align*}
(where the subscript $k$ denotes the $k\th$ graded part). These refine
the spectral category and bimodule introduced in
Section~\ref{sec:build-spec-bim} in the sense that
\begin{align*}
  \Hom_{\GKTSpecCat{n}}(a,b)&\simeq\bigvee_k\Hom_{\GKTSpecCat{n}}(a,b)_k\\
    \GKTSpecBim{T}(a,b)&\simeq\bigvee_k\GKTSpecBim{T}(a,b)_k,
\end{align*}
canonically, where the left side is the definition in
Section~\ref{sec:build-spec-bim} and the right side is the definition
in this section. So, the fact that we are using the same notation for
the definitions in this section and in
Section~\ref{sec:build-spec-bim} will not cause confusion.

The proof of invariance (Sections~\ref{sec:initial-invariance}
and~\ref{sec:bimod-inv}) goes through without essential changes. The
graded analogue of the gluing theorem is:
\begin{theorem}\label{thm:graded-gluing}
  The gluing map induces an equivalence of graded spectral bimodules
  \[
  \KTSpecBim{T_1}\DTP_{\KTSpecCat{n}}\KTSpecBim{T_2}\simeq \KTSpecBim{T_1T_2}.
  \]
\end{theorem}
The proof differs from the proof of Theorem~\ref{thm:gluing} only in
that the notation is more cumbersome.

\begin{remark}
  There is an asymmetry in Formula~\eqref{eq:q-gr-tangle}: the number
  of points $2n$ on the right of the tangle appears, but the number of
  points $2m$ on the left of the tangle does not.
\end{remark}

\begin{remark}
  The quantum gradings we have defined agree with the gradings in
  Khovanov's paper on the arc algebras~\cite{Kho-kh-tangles}, but not
  with those in his first paper on Khovanov
  homology~\cite{Kho-kh-categorification}. See also
  Remark~\ref{rem:grading-convention}.
\end{remark}

\section{Some computations and applications}\label{sec:appl}
\subsection{The connected sum theorem}
We start by noting that our previous connected sum theorem can be
understood as a special case of tangle gluing. Recall:
\begin{theorem}[{\cite[Theorem 8]{LLS-khovanov-product}}]
  Given any knots $K_1$, $K_2$ there are $\KTSpecCat{1}$-module structures on $\KTSpecBim{K_i}$ so that
  $\KTSpecBim{K_1\#K_2}\simeq \KTSpecBim{K_1}\DTP_{\KTSpecCat{1}}\KTSpecBim{K_2}$.
\end{theorem}
\begin{proof}
  Delete a small interval from $K_i$ to obtain a $(0,2)$-tangle $T_1$
  and a $(2,0)$-tangle $T_2$. Since there is a unique crossingless
  matching $c$ of $2$ points, $\KTSpecBim{T_i}$ consists of a single
  spectrum $\KTSpecBim{K_1}\simeq \KTSpecBim{T_1}(\emptyset,c)$ (respectively
  $\KTSpecBim{K_2}\simeq \KTSpecBim{T_2}(c,\emptyset)$), together with a map
  \begin{align*}
    \KTSpecBim{T_1}(\emptyset,c)\smas \Hom_{\KTSpecCat{1}}(c,c)&\to\KTSpecBim{T_1}(\emptyset,c)\\
    \Hom_{\KTSpecCat{1}}(c,c)\smas \KTSpecBim{T_1}(c,\emptyset)&\to\KTSpecBim{T_2}(c,\emptyset)    
  \end{align*}
  making $\KTSpecBim{T_1}(\emptyset,c)$ (respectively
  $\KTSpecBim{T_2}(c,\emptyset)$) into a module spectrum over the ring
  spectrum $\Hom_{\KTSpecCat{1}}(c,c)$. So, the statement is immediate
  from Theorem~\ref{thm:gluing}.
\end{proof}

\begin{remark}
  In \cite[Theorem 8]{LLS-khovanov-product}, the derived tensor
  product over $\KTSpecCat{1}$ was denoted $\otimes_{\mathbb{H}^1}$,
  and the Khovanov spectra were denoted $\KhSpace(K_i)$. The
  construction of this paper is the `opposite' of the construction of
  the previous paper (see Remark~\ref{rem:hom-not-cohom}) and
  therefore $\KTSpecBim{K_i}=\KhSpace(m(K_i))$ where $m(K_i)$ is the
  mirror knot.
\end{remark}

Next we note that the K\"unneth spectral sequence for structured
spectra implies a K\"unneth spectral sequence for Khovanov generalized
homology (e.g., Khovanov $K$-theory, Khovanov bordism, \dots):
\begin{theorem}
  Suppose $K$ is decomposed as a union of a $(0,2n)$-tangle $T_1$ and
  a $(2n,0)$-tangle $T_2$. Then for any generalized homology theory
  $h_*$ there is a spectral sequence
  \[
    \Tor^{h_*(\KTSpecCat{n})}_{p,q}(h_*(\KTSpecBim{T_1}),h_*(\KTSpecBim{T_2}))\Rightarrow h_{p+q}(\KTSpecBim{K}).
  \]
\end{theorem}
\begin{proof}
  This is a corollary~\cite[Theorem 6.4]{EKMM-top}, after using the
  equivalence of symmetric spectra and EKMM spectra.
\end{proof}

\subsection{Hochschild homology and links in \texorpdfstring{$S^1\times S^2$}{S1XS2}}\label{sec:S1S2}
Using Hochschild homology, Rozansky defined a knot homology for links
in $S^1\times S^2$ with even winding number around
$S^1$~\cite{Rozansky-Kh-S1S2} (see
also~\cite{Willis-Kh-S1S2}). In this section we note that Rozansky's
invariant admits a stable homotopy refinement, and conjecture that the
refinement is a knot invariant.

Given an $(n,n)$-tangle $T$ in $[0,1]\times \DD^2$, there are three
ways one can close $T$:
\begin{enumerate}
\item Identify $(0,p)\sim(1,p)$ to obtain a knot $K_{S^1\times\DD^2}\subset S^1\times\DD^2$.
\item Include $S^1\times\DD^2$ as a neighborhood of the unknot in
  $S^3$, and let $K_{S^3}\subset S^3$ be the image of
  $K_{S^1\times\DD^2}$.
\item Include $S^1\times\DD^2$ in
  $S^1\times S^2=(S^1\times\DD^2)\cup_\bdy (S^1\times\DD^2)$, and let
  $K_{S^1\times S^2}\subset S^1\times S^2$ be the image of
  $K_{S^1\times\DD^2}$.
\end{enumerate}

It is clear that every link in $S^1\times\DD^2$, $S^3$, and
$S^1\times S^2$ arises this way. If we require that $n$ be even (which
we shall) then the links which arise in $S^1\times\DD^2$ and
$S^1\times S^2$ are exactly those with even winding number around
$S^1$.

Rozansky's invariant of a knot $K$ in $S^1\times S^2$ is the
Hochschild homology of $\KTfunc(T)$, where $T$ is a tangle whose
closure is $K$. Correspondingly, the stable homotopy lift is the
topological Hochschild homology of $\KTSpecBim{T}$, the definition of
which we recall briefly:
\begin{definition}
  Given a pointwise cofibrant spectral category $\Cat$ and a
  $(\Cat,\Cat)$-bimodules $\SModule$, the \emph{topological Hochschild
    homology} $\THH_{\Cat}(\SModule)=\THH(\SModule)$ of $\SModule$ is
  the homotopy colimit of the diagram
  \[
    \cdots
    \mathrel{\substack{\textstyle\rightarrow\\[-0.5ex]
    \textstyle\rightarrow \\[-0.5ex]
    \textstyle\rightarrow\\[-0.5ex]
    \textstyle\rightarrow}}
    \!\!\!\!\!\!\coprod_{a_1,a_2,a_3\in\Ob(\Cat)}\!\!\!\!\!\!\SModule(a_3,a_1)\smas\Cat(a_1,a_2)\smas \Cat(a_2,a_3)
    \mathrel{\substack{\textstyle\rightarrow\\[-0.5ex]
    \textstyle\rightarrow \\[-0.5ex]
    \textstyle\rightarrow}}
    \!\!\!\!\coprod_{a_1,a_2\in\Ob(\Cat)}\!\!\!\!\SModule(a_2,a_1)\smas\Cat(a_1,a_2)\rightrightarrows\!\!\coprod_{a_1\in\Ob(\Cat)}\!\!\SModule(a_1,a_1)
  \]
  where $\Cat(a,b)$ denotes $\Hom_{\Cat}(a,b)$ and the maps
  \[
    d_i\co \SModule(a_n,a_1)\smas\Cat(a_1,a_2)\smas \cdots\smas\Cat(a_{n-1},a_n)\to \coprod\SModule(b_{n-1},b_1)\smas \Cat(b_1,b_2)\smas \cdots\smas\Cat(b_{n-2},b_{n-1})
  \]
  are given by composition
  $\Cat(a_{i},a_{i+1})\smas \Cat(a_{i+1},a_{i+2})\to\Cat(a_i,a_{i+2})$ if
  $1\leq i\leq n-2$ and the actions
  $\SModule(a_n,a_1)\smas\Cat(a_1,a_2) \to \SModule(a_n,a_2)$ and
  $\Cat(a_{n-1},a_n)\smas\SModule(a_n,a_1)\to\SModule(a_{n-1},a_{1})$ if
  $i=0$ or $n-1$, respectively.
\end{definition}
(Compare~\cite[Proposition 3.5]{BM-top-spectral}. Recall from
Lemma~\ref{lem:pointwise-cofibrant} that $\KTSpecCat{n}$ is pointwise
cofibrant.)

\begin{proposition}\label{prop:S1D2invt}
  If $T$ and $T'$ induce isotopic knots in $S^1\times \DD^2$ then for
  each $j\in\ZZ$,
  \[
    \THH(\KTSpecBim{T,j})\simeq \THH(\KTSpecBim{T',j}).
  \]
\end{proposition}
\begin{proof}
  Given a $(2n,2n)$-tangle $T$ decomposed as a composition of two
  smaller tangles, $T=T_1\circ T_2$, we will call the tangle
  $T_2\circ T_1$ a \emph{rotation of $T$}.  If $T$ and $T'$ induce
  isotopic knots in $S^1\times\DD^2$ then $T$ and $T'$ are related by
  a sequence of Reidemeister moves and rotations. Topological
  Hochschild homology is invariant under quasi-isomorphisms of
  spectral bimodules~\cite[Proposition 3.7]{BM-top-spectral}, so by
  Theorem~\ref{thm:spec-invariance} Reidemeister moves do not change
  $\THH(\KTSpecBim{T,j})$. Topological Hochschild homology is a trace,
  in the sense that given spectral categories $\Cat$, $\Dat$, a
  $(\Cat,\Dat)$-bimodule $\SModule$ and a $(\Dat,\Cat)$-bimodule
  $\SNodule$,
  \[
    \THH_{\Cat}(\SModule\DTP_{\Dat}\SNodule)\simeq \THH_{\Dat}(\SNodule\DTP_{\Cat}\SModule)
  \]
  \cite[Proposition 6.2]{BM-top-spectral}.  Thus, it
  follows from Theorem~\ref{thm:gluing} that $\THH(\KTSpecBim{T,j})$
  is invariant under rotation as well.
\end{proof}

\begin{remark}
  Since we have only defined an invariant of a $(2m,2n)$-tangle, any
  link in $S^1\times\DD^2$ which arise from our construction has
  even winding number.
\end{remark}

\begin{proposition}
  The singular homology of $\THH(\KTSpecBim{T})$ is Rozansky's
  invariant $\Roz{K_{S^2\times S^1}}$.
\end{proposition}
\begin{proof}
  The proof is similar to the proof of
  Proposition~\ref{prop:homology-right}, and is left to the reader.
\end{proof}

\begin{conjecture}\label{conj:S1S2invt}
  If $T$ and $T'$ induce isotopic knots in $S^1\times S^2$ then for
  each $j\in\ZZ$,
  \[
    \THH(\KTSpecBim{T,j})\simeq \THH(\KTSpecBim{T',j}).
  \]
\end{conjecture}

As Rozansky notes, given Proposition~\ref{prop:S1D2invt}, to verify
Conjecture~\ref{conj:S1S2invt} it suffices to verify that
$\THH(\KTSpecBim{T,j})$ is invariant under dragging the first strand
around the others~\cite[Theorem 2.2]{Rozansky-Kh-S1S2}.

\subsection{Where the ladybug matching went: an example}
Our longtime readers will recall that a key step in the construction
of $\KTSpecBim{K}$ is the ladybug matching, which provides an
identification across each 2-dimensional face in the cube of
resolutions. (This matching is equivalent to the rule for composing
genus $0$ cobordisms to get a genus $1$ cobordism in
Section~\ref{sec:CobE-to-Burn}.) In particular, the ladybug matching
is relevant for certain pairs of crossings in a diagram $K$. Such
readers may wonder where the ladybug matching has gone, now that
the Khovanov homotopy type can be constructed by composing a sequence
of 1-crossing tangles. We answer this question, with an example.

\begin{figure}
  \centering
  \includegraphics{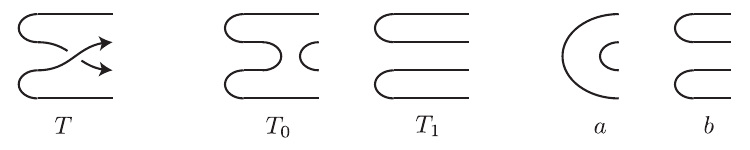}
  \caption{\textbf{Where have all the ladybugs gone?} Left: a tangle $T$. Center: the resolutions $T_0$ and $T_1$ of $T$. Right: the crossingless matchings $a$ and $b$.}
  \label{fig:ladybug-gone}
\end{figure}

Consider the $(0,4)$-tangle $T$ shown in
Figure~\ref{fig:ladybug-gone}. If we let $a$ and $b$ be the two
crossingless matchings on $4$ strands, labeled as in that figure, then 
\begin{align*}
 \KTSpecBim{T}(a)&=\Cone\left(\SphereS_{a,1\otimes 1}\vee\SphereS_{a,1\otimes X}\vee\SphereS_{a,X\otimes 1}\vee\SphereS_{a,X\otimes X}\longrightarrow \SphereS_{a,1}\vee\SphereS_{a,X}\right)\\
&=\Cone\left(\SphereS_{a,1\otimes 1}\to\SphereS_{a,1}\right)\vee\Cone\left(\SphereS_{a,1\otimes X}\vee\SphereS_{a,X\otimes 1}\to\SphereS_{a,X}\right)\vee\Cone\left(\SphereS_{a,X\otimes X}\to\pt\right),\\
  \KTSpecBim{T}(b)&=\Cone\left(\SphereS_{b,1}\vee\SphereS_{b,X}\longrightarrow\SphereS_{b,1\otimes 1}\vee\SphereS_{b,1\otimes X}\vee\SphereS_{b,X\otimes 1}\vee\SphereS_{b,X\otimes X}\right)\\
&=\left(\SphereS_{b,1\otimes 1}\right)\vee\Cone\left(\SphereS_{b,1}\to\SphereS_{b,1\otimes X}\vee\SphereS_{b,X\otimes 1}\right)\vee\Cone\left(\SphereS_{b,X}\to\SphereS_{b,X\otimes X}\right),
\end{align*}
where we have used subscripts to indicate the Khovanov generator
corresponding to each summand. These mapping cones are indicated in
Figure~\ref{fig:cones-of-spheres-new} (where $\SphereS$ has been depicted
as $S^1$).

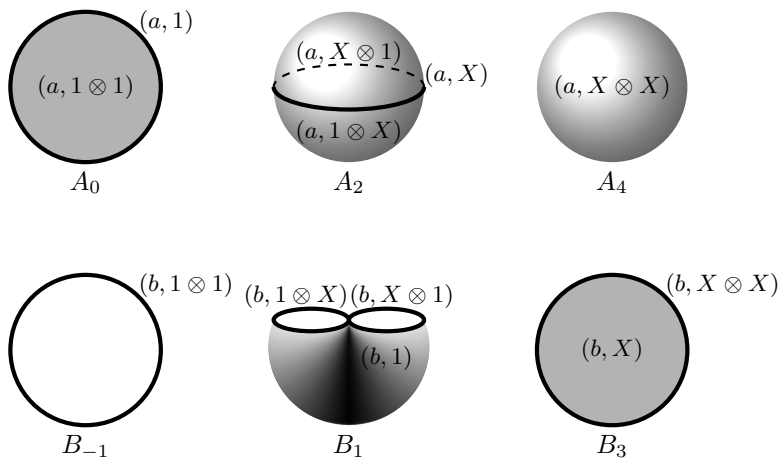
\begin{figure}
  \centering
  \begin{tikzpicture}
    
    \begin{scope}
      \draw[ultra thick,fill=black!30] (0,0) circle (1cm);
      \node[anchor=south west,inner sep=0pt,outer sep=0pt] at (45:1cm) {\small $(a,1)$};
      \node at (0,0) {\small $(a,1\otimes 1)$};
      \node[anchor=north] at (0,-1) {$A_0$};
    \end{scope}

    \begin{scope}[xshift=3.5cm]
      \shade [ball color=white] (0,0) circle [radius=1cm];
      \begin{scope}[yscale=0.3]
        \draw[ultra thick] (-1,0) arc (180:360:1cm); 
        \draw[dashed,thick] (1,0) arc (0:180:1cm); 
      \end{scope}
      \node[anchor=south west,inner sep=0pt,outer sep=0pt] at (1,0) {\small $(a,X)$};
      \node at (0,-0.6) {\small $(a,1\otimes X)$};
      \node at (0,0.45) {\small $(a,X\otimes 1)$};
      \node[anchor=north] at (0,-1) {$A_{2}$};
    \end{scope}

    \begin{scope}[xshift=7cm]
      \shade [ball color=white] (0,0) circle [radius=1cm];
      \node at (0,0) {\small $(a,X\otimes X)$};
      \node[anchor=north] at (0,-1) {$A_{4}$};
    \end{scope}

    \begin{scope}[yshift=-3.5cm]
      \draw[ultra thick] (0,0) circle (1cm);
      \node[anchor=south west,inner sep=0pt,outer sep=0pt] at (45:1cm) {\small $(b,1\otimes 1)$};
      \node[anchor=north] at (0,-1) {$B_{-1}$};
    \end{scope}

    \begin{scope}[xshift=3.5cm,yshift=-3.5cm]
      \foreach \cc in {0,0.005,...,1}{
        \pgfmathsetmacro\r{acos(\cc)}
        \foreach \sr in {\r,180-\r}{
          \pgfmathsetmacro\cirlen{1+0.4*sin(\sr)}
          \pgfmathsetmacro\cirwidabs{max(0.05,0.3*\cc)}
          \pgfmathsetmacro\ltint{0.9*abs(\sr-90)/90}
          \definecolor{currentcolor}{rgb}{\ltint, \ltint, \ltint}
          \begin{scope}[yshift=0.4cm,yscale=-1,rotate=\sr,yscale=\cirwidabs,xscale=\cirlen]
            \fill[currentcolor] (0.5,0) circle (0.5cm);
          \end{scope}
        }}

      \node at (0.5,-0.1) {\small $(b,1)$};
      \begin{scope}[yshift=0.4cm,yscale=0.3]
      \draw[ultra thick,fill=white] (0.5,0) circle (0.5cm);
      \draw[ultra thick,fill=white] (-0.5,0) circle (0.5cm);
      \node[anchor=south east,inner sep=0pt] at (0,0.5) {\small $(b,1\otimes X)$};
      \node[anchor=south west,inner sep=0pt] at (0,0.5) {\small $(b,X\otimes 1)$};
      \end{scope}
    \node[anchor=north] at (0,-1) {$B_1$};
    \end{scope}

    \begin{scope}[xshift=7cm,yshift=-3.5cm]
      \draw[ultra thick,fill=black!30] (0,0) circle (1cm);
      \node[anchor=south west,inner sep=0pt,outer sep=0pt] at (45:1cm) {\small $(b,X\otimes X)$};
      \node at (0,0) {\small $(b,X)$};
      \node[anchor=north] at (0,-1) {$B_3$};
    \end{scope}

  \end{tikzpicture}
  \caption{\textbf{Some mapping cones.} The space $\KTSpecBim{T}(a)$
    is the wedge sum of the spaces $A_0$, $A_2$, and $A_{4}$,
    while $\KTSpecBim{T}(b)$ is the wedge sum of the spaces
    $B_{-1}$, $B_1$, and $B_3$ (the subscripts denote the
    quantum gradings). A cellular decomposition is shown with the
    cells labeled by the corresponding Khovanov generators. The space
    $A_2$ is built from one $1$-cell labeled $(a,X)$ two $2$-cells
    labeled $(a,1\otimes X)$ and $(a,X\otimes 1)$. The space $B_1$ has
    two $1$-cells labeled $(b,1\otimes X)$ and $(b,X\otimes 1)$ and
    one $2$-cell labeled $(b,1)$.}
  \label{fig:cones-of-spheres-new}
\end{figure}

Consider now the spaces
$A_2=\Cone\left(\SphereS_{a,1\otimes X}\vee\SphereS_{a,X\otimes
    1}\to\SphereS_{a,X}\right)$ and
$B_{1}=\Cone\left(\SphereS_{b,1}\to\SphereS_{b,1\otimes
    X}\vee\SphereS_{b,X\otimes 1}\right)$. The operation
$\KTSpecBim{T}(b)\otimes \Hom_{\KTSpecCat{2}}(b,a) \to
\KTSpecBim{T}(a)$ gives a map
\[
  B_1\smas \SphereS_{b\Wmirror{a},1}\to A_2,
\]
where $\SphereS_{b\Wmirror{a},1}$ is the wedge summand of
$\Hom_{\KTSpecCat{2}}(b,a)$ which labels the single circle in
$b\Wmirror{a}$ by $1$ (which lives in quantum grading $1$). This map
sends half of $B_1$ to the top half in $A_2$ and half of $B_1$ to the
bottom half in $A_2$. Which half is sent to which half is determined
by the ladybug matching. The two maps are, of course, homotopic, by
rotating the sphere $A_2$ by $\pi$ or $-\pi$, but the homotopy is not
canonical.

\vspace{-0.3cm}
\bibliographystyle{MyPlain}
\bibliography{newbibfile}

\providecommand{\bysame}{\leavevmode\hbox to3em{\hrulefill}\thinspace}
\providecommand{\MR}{\relax\ifhmode\unskip\space\fi MR }
\providecommand{\MRhref}[2]{%
  \href{http://www.ams.org/mathscinet-getitem?mr=#1}{#2}
}
\providecommand{\href}[2]{#2}
\begin{thebibliography}{EKMM97}

\bibitem[Ada74]{Ada-top-stablehomotopy}
J.~F. Adams, \emph{Stable homotopy and generalised homology}, University of
  Chicago Press, Chicago, Ill., 1974, Chicago Lectures in Mathematics.
  \MR{0402720 (53 \#6534)}

\bibitem[APS06]{APS-kh-tangle}
Marta~M. Asaeda, J{\'o}zef~H. Przytycki, and Adam~S. Sikora,
  \emph{Categorification of the skein module of tangles}, Primes and knots,
  Contemp. Math., vol. 416, Amer. Math. Soc., Providence, RI, 2006, pp.~1--8.
  \MR{2276132}

\bibitem[Ati90]{Atiyah-knot-book}
Michael Atiyah, \emph{The geometry and physics of knots}, Lezioni Lincee.
  [Lincei Lectures], Cambridge University Press, Cambridge, 1990. \MR{1078014}

\bibitem[Bal11]{Baldwin-hf-s-seq}
John~A. Baldwin, \emph{On the spectral sequence from {K}hovanov homology to
  {H}eegaard {F}loer homology}, Int. Math. Res. Not. IMRN (2011), no.~15,
  3426--3470. \MR{2822178 (2012g:57021)}

\bibitem[Bar05]{Bar-kh-tangle-cob}
Dror Bar-Natan, \emph{Khovanov's homology for tangles and cobordisms}, Geom.
  Topol. \textbf{9} (2005), 1443--1499. \MR{2174270 (2006g:57017)}

\bibitem[BF78]{BF-top-spectra}
A.~K. Bousfield and E.~M. Friedlander, \emph{Homotopy theory of {$\Gamma
  $}-spaces, spectra, and bisimplicial sets}, Geometric applications of
  homotopy theory ({P}roc. {C}onf., {E}vanston, {I}ll., 1977), {II}, Lecture
  Notes in Math., vol. 658, Springer, Berlin, 1978, pp.~80--130.

\bibitem[BK72]{BK-top-book}
A.~K. Bousfield and D.~M. Kan, \emph{Homotopy limits, completions and
  localizations}, Lecture Notes in Mathematics, Vol. 304, Springer-Verlag,
  Berlin, 1972. \MR{0365573 (51 \#1825)}

\bibitem[BM12]{BM-top-spectral}
Andrew~J. Blumberg and Michael~A. Mandell, \emph{Localization theorems in
  topological {H}ochschild homology and topological cyclic homology}, Geom.
  Topol. \textbf{16} (2012), no.~2, 1053--1120. \MR{2928988}

\bibitem[BS11]{BS-kh-tangle}
Jonathan Brundan and Catharina Stroppel, \emph{Highest weight categories
  arising from {K}hovanov's diagram algebra {I}: cellularity}, Mosc. Math. J.
  \textbf{11} (2011), no.~4, 685--722, 821--822. \MR{2918294}

\bibitem[BV73]{BV-other-book}
J.~M. Boardman and R.~M. Vogt, \emph{Homotopy invariant algebraic structures on
  topological spaces}, Lecture Notes in Mathematics, Vol. 347, Springer-Verlag,
  Berlin-New York, 1973. \MR{0420609}

\bibitem[CEF15]{CEF-top-FImodules}
Thomas Church, Jordan~S. Ellenberg, and Benson Farb, \emph{F{I}-modules and
  stability for representations of symmetric groups}, Duke Math. J.
  \textbf{164} (2015), no.~9, 1833--1910. \MR{3357185}

\bibitem[CF94]{CF-kh-categorify}
Louis Crane and Igor~B. Frenkel, \emph{Four-dimensional topological quantum
  field theory, {H}opf categories, and the canonical bases}, J. Math. Phys.
  \textbf{35} (1994), no.~10, 5136--5154, Topology and physics. \MR{1295461}

\bibitem[Cho06]{Chorny-top-small}
Boris Chorny, \emph{A generalization of {Q}uillen's small object argument}, J.
  Pure Appl. Algebra \textbf{204} (2006), no.~3, 568--583. \MR{2185618}

\bibitem[CJS95]{CJS-gauge-floerhomotopy}
R.~L. Cohen, J.~D.~S. Jones, and G.~B. Segal, \emph{Floer's
  infinite-dimensional {M}orse theory and homotopy theory}, The {F}loer
  memorial volume, Progr. Math., vol. 133, Birkh\"auser, Basel, 1995,
  pp.~297--325. \MR{1362832 (96i:55012)}

\bibitem[CK14]{CK-kh-tangle}
Yanfeng Chen and Mikhail Khovanov, \emph{An invariant of tangle cobordisms via
  subquotients of arc rings}, Fund. Math. \textbf{225} (2014), no.~1, 23--44.
  \MR{3205563}

\bibitem[Coh10]{Cohen10:Floer-htpy-cotangent}
Ralph~L. Cohen, \emph{The {F}loer homotopy type of the cotangent bundle}, Pure
  Appl. Math. Q. \textbf{6} (2010), no.~2, Special Issue: In honor of Michael
  Atiyah and Isadore Singer, 391--438. \MR{2761853}

\bibitem[Don90]{Donaldson-gauge-poly}
S.~K. Donaldson, \emph{Polynomial invariants for smooth four-manifolds},
  Topology \textbf{29} (1990), no.~3, 257--315. \MR{1066174}

\bibitem[EKMM97]{EKMM-top}
A.~D. Elmendorf, I.~Kriz, M.~A. Mandell, and J.~P. May, \emph{Rings, modules,
  and algebras in stable homotopy theory}, Mathematical Surveys and Monographs,
  vol.~47, American Mathematical Society, Providence, RI, 1997, With an
  appendix by M. Cole. \MR{1417719}

\bibitem[ELST16]{ELST-trivial}
Brent Everitt, Robert Lipshitz, Sucharit Sarkar, and Paul Turner,
  \emph{Khovanov homotopy types and the {D}old-{T}hom functor}, Homology
  Homotopy Appl. \textbf{18} (2016), no.~2, 177--181. \MR{3547241}

\bibitem[EM06]{EM-top-machine}
A.~D. Elmendorf and M.~A. Mandell, \emph{Rings, modules, and algebras in
  infinite loop space theory}, Adv. Math. \textbf{205} (2006), no.~1, 163--228.
  \MR{2254311 (2007g:19001)}

\bibitem[ET14]{ET-kh-spectrum}
Brent Everitt and Paul Turner, \emph{The homotopy theory of {K}hovanov
  homology}, Algebr. Geom. Topol. \textbf{14} (2014), no.~5, 2747--2781.
  \MR{3276847}

\bibitem[Flo88]{Floer-gauge-instanton}
Andreas Floer, \emph{An instanton-invariant for {$3$}-manifolds}, Comm. Math.
  Phys. \textbf{118} (1988), no.~2, 215--240. \MR{956166}

\bibitem[GS16]{GarnerShulman}
Richard Garner and Michael Shulman, \emph{Enriched categories as a free
  cocompletion}, Advances in Mathematics \textbf{289} (2016), 1 -- 94.

\bibitem[HKK16]{HKK-Kh-htpy}
Po~Hu, Daniel Kriz, and Igor Kriz, \emph{Field theories, stable homotopy theory
  and {K}hovanov homology}, Topology Proc. \textbf{48} (2016), 327--360.

\bibitem[HKS19]{HKS-Kh-sln}
Po~Hu, Igor Kriz, and Petr Somberg, \emph{Derived representation theory of
  {L}ie algebras and stable homotopy categorification of {$sl_k$}}, Adv. Math.
  \textbf{341} (2019), 367--439. \MR{3872851}

\bibitem[HSS00]{HSS-top-symmetric}
Mark Hovey, Brooke Shipley, and Jeff Smith, \emph{Symmetric spectra}, J. Amer.
  Math. Soc. \textbf{13} (2000), no.~1, 149--208. \MR{1695653 (2000h:55016)}

\bibitem[Isb69]{Isbell-other-coherent}
John~R. Isbell, \emph{On coherent algebras and strict algebras}, J. Algebra
  \textbf{13} (1969), 299--307. \MR{0249484}

\bibitem[Jac04]{Jac-kh-cobordisms}
Magnus Jacobsson, \emph{An invariant of link cobordisms from {K}hovanov
  homology}, Algebr. Geom. Topol. \textbf{4} (2004), 1211--1251 (electronic).
  \MR{2113903 (2005k:57047)}

\bibitem[Jon85]{Jones-knot-poly}
Vaughan F.~R. Jones, \emph{A polynomial invariant for knots via von {N}eumann
  algebras}, Bull. Amer. Math. Soc. (N.S.) \textbf{12} (1985), no.~1, 103--111.
  \MR{766964}

\bibitem[Kho00]{Kho-kh-categorification}
Mikhail Khovanov, \emph{A categorification of the {J}ones polynomial}, Duke
  Math. J. \textbf{101} (2000), no.~3, 359--426. \MR{1740682 (2002j:57025)}

\bibitem[Kho02]{Kho-kh-tangles}
\bysame, \emph{A functor-valued invariant of tangles}, Algebr. Geom. Topol.
  \textbf{2} (2002), 665--741 (electronic). \MR{1928174 (2004d:57016)}

\bibitem[Kho06]{Kho-kh-cobordism}
\bysame, \emph{An invariant of tangle cobordisms}, Trans. Amer. Math. Soc.
  \textbf{358} (2006), no.~1, 315--327. \MR{2171235 (2006g:57046)}

\bibitem[KL09]{KL-kh-cat-gp}
Mikhail Khovanov and Aaron~D. Lauda, \emph{A diagrammatic approach to
  categorification of quantum groups. {I}}, Represent. Theory \textbf{13}
  (2009), 309--347. \MR{2525917}

\bibitem[KLS18]{KLS-gauge-unfolded}
Tirasan Khandhawit, Jianfeng Lin, and Hirofumi Sasahira, \emph{Unfolded
  {S}eiberg-{W}itten {F}loer spectra, {I}: {D}efinition and invariance}, Geom.
  Topol. \textbf{22} (2018), no.~4, 2027--2114. \MR{3784516}

\bibitem[KM]{KM-gauge-swspectrum}
Peter Kronheimer and Ciprian Manolescu, \emph{Periodic {F}loer pro-spectra from
  the {S}eiberg-{W}itten equations}, arXiv:math/0203243.

\bibitem[Kra18]{Kragh:transfer-spectra}
Thomas Kragh, \emph{The {V}iterbo transfer as a map of spectra}, J. Symplectic
  Geom. \textbf{16} (2018), no.~1, 85--226. \MR{3798329}

\bibitem[Lam69]{Lambek-other-multicat}
Joachim Lambek, \emph{Deductive systems and categories. {II}. {S}tandard
  constructions and closed categories}, Category {T}heory, {H}omology {T}heory
  and their {A}pplications, {I} ({B}attelle {I}nstitute {C}onference,
  {S}eattle, {W}ash., 1968, {V}ol. {O}ne), Springer, Berlin, 1969, pp.~76--122.
  \MR{0242637}

\bibitem[Lau10]{Lauda-kh-sl2}
Aaron~D. Lauda, \emph{A categorification of quantum {${\rm sl}(2)$}}, Adv.
  Math. \textbf{225} (2010), no.~6, 3327--3424. \MR{2729010}

\bibitem[LLS17]{LLS-cube}
Tyler Lawson, Robert Lipshitz, and Sucharit Sarkar, \emph{The cube and the
  {B}urnside category}, Categorification in Geometry, Topology, and Physics,
  Contemp. Math., vol. 684, Amer. Math. Soc., Providence, RI, 2017, pp.~63--85.

\bibitem[LLS20]{LLS-khovanov-product}
\bysame, \emph{Khovanov homotopy type, {B}urnside category and products}, Geom.
  Topol. \textbf{24} (2020), no.~2, 623--745. \MR{4153651}

\bibitem[LS14a]{RS-khovanov}
Robert Lipshitz and Sucharit Sarkar, \emph{A {K}hovanov stable homotopy type},
  J. Amer. Math. Soc. \textbf{27} (2014), no.~4, 983--1042. \MR{3230817}

\bibitem[LS14b]{RS-s-invariant}
\bysame, \emph{A refinement of {R}asmussen's {$s$}-invariant}, Duke Math. J.
  \textbf{163} (2014), no.~5, 923--952. \MR{3189434}

\bibitem[LS14c]{RS-steenrod}
\bysame, \emph{A {S}teenrod square on {K}hovanov homology}, J. Topol.
  \textbf{7} (2014), no.~3, 817--848. \MR{3252965}

\bibitem[Lur]{Lurie-top-HA}
Jacob Lurie, \emph{Higher algebra},
  {\url{http://www.math.harvard.edu/~lurie/papers/HA.pdf}}.

\bibitem[Lur09]{Lurie-top-htt}
\bysame, \emph{Higher topos theory}, Annals of Mathematics Studies, vol. 170,
  Princeton University Press, Princeton, NJ, 2009. \MR{2522659}

\bibitem[Man03]{Man-gauge-swspectrum}
Ciprian Manolescu, \emph{Seiberg-{W}itten-{F}loer stable homotopy type of
  three-manifolds with {$b_1=0$}}, Geom. Topol. \textbf{7} (2003), 889--932
  (electronic). \MR{2026550 (2005b:57060)}

\bibitem[MMSS01]{MMSS-diagram-spectra}
M.~A. Mandell, J.~P. May, S.~Schwede, and B.~Shipley, \emph{Model categories of
  diagram spectra}, Proc. London Math. Soc. (3) \textbf{82} (2001), no.~2,
  441--512. \MR{1806878 (2001k:55025)}

\bibitem[Ras10]{Ras-kh-slice}
Jacob Rasmussen, \emph{Khovanov homology and the slice genus}, Invent. Math.
  \textbf{182} (2010), no.~2, 419--447. \MR{2729272 (2011k:57020)}

\bibitem[Rob16]{Roberts-kh-tangle}
Lawrence~P. Roberts, \emph{A type {$D$} structure in {K}hovanov homology}, Adv.
  Math. \textbf{293} (2016), 81--145. \MR{3474320}

\bibitem[Rou]{Rouquier-Kh-cat}
Raphael Rouquier, \emph{{2}-{K}ac-{M}oody algebras}, arXiv:0812.5023.

\bibitem[Roz]{Rozansky-Kh-S1S2}
Lev Rozansky, \emph{A categorification of the stable {$SU(2)$}
  {W}itten-{R}eshetikhin-{T}uraev invariant of links in {$S^2\times S^1$}},
  arXiv:1011.1958.

\bibitem[RT91]{RT-knot-quantum}
N.~Reshetikhin and V.~G. Turaev, \emph{Invariants of {$3$}-manifolds via link
  polynomials and quantum groups}, Invent. Math. \textbf{103} (1991), no.~3,
  547--597. \MR{1091619}

\bibitem[Sch]{S-top-symmetricspectra}
Stefan Schwede, \emph{Symmetric spectra (v3.0)},
  \url{http://www.math.uni-bonn.de/people/schwede/SymSpec-v3.pdf}.

\bibitem[Sch08]{Sch-top-htpygrp}
\bysame, \emph{On the homotopy groups of symmetric spectra}, Geom. Topol.
  \textbf{12} (2008), no.~3, 1313--1344. \MR{2421129}

\bibitem[See]{Seed-Kh-square}
Cotton Seed, \emph{Computations of the {L}ipshitz-{S}arkar {S}teenrod square on
  {K}hovanov homology}, arXiv:1210.1882.

\bibitem[Web16]{Webster-kh-tensor}
Ben Webster, \emph{Tensor product algebras, {G}rassmannians and {K}hovanov
  homology}, Physics and mathematics of link homology, Contemp. Math., vol.
  680, Amer. Math. Soc., Providence, RI, 2016, pp.~23--58. \MR{3591642}

\bibitem[Wil]{Willis-Kh-S1S2}
Michael Willis, \emph{{K}hovanov homology for links in {$\#^r(S^2\times
  S^1)$}}, arXiv:1812.06584.

\bibitem[Wit88]{Witten-gauge-TQFT}
Edward Witten, \emph{Topological quantum field theory}, Comm. Math. Phys.
  \textbf{117} (1988), no.~3, 353--386. \MR{953828}

\bibitem[Wit89]{Witten-knot-Jones}
\bysame, \emph{Quantum field theory and the {J}ones polynomial}, Comm. Math.
  Phys. \textbf{121} (1989), no.~3, 351--399. \MR{990772}

\bibitem[Zar]{Zarev-hf-BS}
Rumen Zarev, \emph{Bordered {F}loer homology for sutured manifolds},
  arXiv:0908.1106.

\end{thebibliography}
\vspace{1cm}
\end{document}